\documentclass[reqno,english]{article}
\usepackage{amsfonts,amsmath,latexsym,verbatim,amscd,mathrsfs,color,array}

\usepackage{amsmath,amssymb,amsthm,amsfonts,graphicx,color}
\usepackage{amssymb}
\usepackage{pdfsync}
\usepackage{epstopdf}
\usepackage[colorlinks=true]{hyperref}

\makeatletter
\def\@currentlabel{2.1}\label{e:dispaa}
\def\@currentlabel{2.21}\label{e:dispau}
\def\@currentlabel{2.22}\label{e:dispav}
\def\@currentlabel{2.23}\label{e:dispaw}
\def\@currentlabel{2.24}\label{e:dispax}
\def\theequation{\thesection.\@arabic\c@equation}
\makeatother

\numberwithin{equation}{section}

\newcommand{\ttt}{\tilde }

\newcommand{\pp}{ {\partial} }

\newcommand{\R} {\mathbb R}
\newcommand{\cuad}{{\sqcap\kern-.68em\sqcup}}
\newcommand{\abs}[1]{\mid #1 \mid}

\newcommand{\no}{{\noindent}}

\newcommand{\bnu}{\nu_0}
\newcommand{\p}{\frac{n+2}{n-2}}

\newcommand{\mmm}{\frac{4}{n+2}}

\newcommand{\al}{\alpha}
\newcommand{\si}{\sigma}

\newcommand{\dist}{{\rm dist}\, }

\newcommand{\be}{\begin{equation}}
\newcommand{\bee}{\begin{equation*}}
\newcommand{\ee}{\end{equation}}
\newcommand{\eee}{\end{equation*}}
\newcommand{\bs}{\begin{split}}
\newcommand{\esp}{\end{split}}
\newcommand{\ssk}{\smallskip}
\newcommand{\msk}{\medskip}

\newcommand{\la}{\lambda}

\newcommand{\inti}{{\int_{-\infty}^\infty}}

\newcommand{\zz}{z}

\newcommand{\equ}[1]{(\ref{#1})}

\renewcommand{\theequation}{\thesection.\arabic{equation}}
 \newtheorem{lem}{Lemma}[section]

\newtheorem{coro}{Corollary}[section]
\newtheorem{definition}{Definition}[section]
\newtheorem{defn}{Definition}[section]

\newtheorem{theorem}{Theorem}[section]
\newtheorem{prop}{Proposition}[section]
\newtheorem{corollary}{Corollary}[section]
\newtheorem{remark}{Remark}[section]

\newtheorem{claim}{Claim}[section]
\newtheorem{notation}{Notation}[section]

\newcommand{\bremark}{\begin{remark} \em}
\newcommand{\eremark}{\end{remark} }

\title{Type II  ancient  compact  solutions to  the Yamabe flow}
\author{Panagiota Daskalopoulos, Manuel del Pino, Natasa Sesum}

\begin{document}

\maketitle

\abstract
We construct  new type II  ancient compact solutions to the  Yamabe flow.  Our solutions 
are rotationally symmetric and converge, as $t \to -\infty$,  to a tower of two spheres. 
Their  curvature operator changes sign. 
We allow two time-dependent parameters in our ansatz.  We use  perturbation  theory,  via fixed point arguments,    
based on sharp estimates on ancient solutions of the approximated linear equation 
and   careful estimation of  the error terms which allow us  to make the right choice of  parameters. 
Our technique may be viewed as a  parabolic analogue of  gluing  two  exact solutions to the rescaled equation, that is the spheres, with narrow cylindrical necks to obtain a new ancient solution to the Yamabe flow. The techniques 
in this article may be generalized to the gluing of $n$ spheres. 
\section{Introduction}
\label{sec-intro}

Let $(M,g_0)$ be a compact manifold without boundary  of dimension $n \geq 3$. If $g = v^{\frac 4{n-2}} \, g_0$
is a metric conformal to $g_0$, the scalar curvature $R$  of $g$ is given in terms of the
scalar curvature $R_0$ of $g_0$ by  
$$R= v^{-\frac {n+2}{n-2}} \, \big ( - \bar c_n \Delta_{g_0} v + R_0 \, v \big )$$
where $\Delta_{g_0}$ denotes the Laplace Beltrami operator with respect to $g_0$ and $\bar c_n =  4 (n-1)/(n-2)$.

In 1989 R. Hamilton introduced the  {\em Yamabe flow} 
\begin{equation}
\label{eq-YF}
\frac{\partial g}{\partial t} = -R\, g
\end{equation}
as an approach to solve the {\em Yamabe problem}  on manifolds of positive conformal Yamabe invariant.
It is the negative $L^2$-gradient flow of the total scalar curvature, restricted to a given conformal class. The flow may  be interpreted as deforming a Riemannian metric to a conformal metric of constant scalar curvature, when this flow converges. 

Hamilton \cite{H}  showed the existence of the  normalized Yamabe flow (which is the re-parametrization of (\ref{eq-YF}) to keep the volume fixed)   for all time; moreover, in the case when the scalar curvature of the initial metric is negative,  he showed   the exponential convergence  of the  flow  to a  metric of constant scalar curvature.

Since then, there has been a number of works on  the  convergence of the  Yamabe flow
on a compact manifold to a metric of constant scalar curvature.  Chow \cite{Ch} showed the  convergence of the flow,  under the conditions that the initial metric is locally conformally flat and of positive Ricci curvature. 
The convergence of the flow for any locally conformally flat initially metric was shown by Ye  \cite{Y}. 

More recently, Schwetlick and Struwe  \cite{SS}  obtained the convergence 
of the  Yamabe flow on a general compact manifold under 
a suitable Kazdan-Warner type of condition that rules out the formation of bubbles 
and that is verified  (via the positive mass Theorem) in dimensions $3 \leq n \leq 5$. 
  The convergence  result,   in its full generality, was established by
 Brendle  \cite{S2} and \cite{S1}   (up to a technical assumption, in dimensions $n \ge 6$,   on the rate of vanishing of Weyl tensor at the points at which it vanishes):  
starting with any smooth  metric on a compact manifold, the normalized Yamabe flow   converges to a metric of constant scalar curvature.  

In the special case where the background manifold $M_0$ is the  sphere $S^n$ and $g_0$ is the standard
spherical metric $g_{_{S^n}}$, the Yamabe flow evolving a metric $g= v^{\frac 4{n-2}}(\cdot,t) 
 \, g_{_{S^n}}$ takes (after  rescaling in time by a constant) 
	the form of the {\em  fast diffusion equation} 
\begin{equation}
\label{eq-YFS}
(v^\frac{n+2}{n-2})_t  = \Delta_{S^n} v -c_n v,  \qquad c_n = \frac{n(n-2)}{4}. 
\end{equation}
Starting with any smooth metric $g_0$ on $S^n$, it follows by the results in \cite{Ch}, \cite{Y} and \cite{DS} that
the solution of \eqref{eq-YFS} with initial data $g_0$ will become singular at some finite time $t < T$ and
$v$ becomes spherical at time $T$,  which means that after 
a normalization, the normalized flow converges to the spherical metric.  In addition, $v$ becomes extinct  at $T$.

A metric $g =  v^{\frac 4{n-2}} \, g_{_{S^n}}$ may also be expressed as a metric on $\R^n$ via  stereographic 
projection. It follows that if $g =  \bar v^{\frac 4{n-2}} (\cdot,t) \, g_{_{\R^n}}$ (where $g_{_{\R^n}}$ denotes the standard metric
on $\R^n$)  evolves by the Yamabe flow \eqref{eq-YF},  then $\bar v$ satisfies (after a rescaling in time) the fast diffusion equation on $\R^n$
\begin{equation}
\label{eq-ufd}
(\bar v^p)_t = \Delta \bar v, \quad \qquad p:= \frac{n+2}{n-2}.
\end{equation}
Observe that if $g =  \bar v^{\frac 4{n-2}} (\cdot,t) \, g_{_{\R^n}}$ represents a smooth solution when lifted on $S^n$,
then $\bar{v}(\cdot,t)$ satisfies the growth condition 
$$\bar{v}(y,t) = O  (|y|^{-(n-2)}), \qquad \mbox{as} \,\, |y| \to \infty.$$

\medskip

\begin{defn}
\label{defn-ancient}
The solution $g = v(\cdot, t)\, g_0$ to \eqref{eq-YF}  is called ancient if it exists for all time $t\in (-\infty, T)$,  
where $T < \infty$.   We will say that the ancient solution $g$ is 
of type I,  if it satisfies
$$\limsup_{t\to-\infty} \, ( |t| \, \max_{M_0} |\mbox{\em Rm}| \ \, (\cdot,t))< \infty,$$ 
(where $\mbox{\em Rm}$ is the Riemannian curvature of metric $g = v(\cdot,t) g_0$ and can be expressed in terms of $v$ and its first and second derivatives).
An ancient  solution which is not of type I,  will be called of type
II.

\end{defn}

\smallskip
\no Explicit examples of ancient solutions to the Yamabe  flow on $S^n$ are:

\smallskip
\noindent {\bf  Contracting spheres}:
They are  special solutions $v$   of \eqref{eq-YFS} which depend only on time $t$ and satisfy the ODE
$$\frac {d v^{\p}}{dt} =-c_n\, v.$$ They are given by 
 \begin{equation}
\label{eq-spheres}
v_S(p,t) = \left ( \mmm \, c_n\,  (T-t) \right )^{\frac {n-2}4}.
\end{equation}
and  represent
a sequence of round spheres shrinking to a point at time $t=T$. They are shrinking solitons and  type I
ancient solutions.  

\smallskip

\noindent {\bf King  solutions:} They were  discovered 
 by J.R. King \cite{K1}.  They can be expressed on $\R^n$ in closed from, namely $g= \bar v_K (\cdot,t) \, g_{_{\R^n}}$, where $\bar v_K$ is the radial function 
 \begin{equation}
\label{eq-king}
\bar v_{K}(r,t) = \left(\frac{A(t) }{1 + 2B(t) \, r^2 + r^4}\right)^{\frac{n-2}{4}}. 
\end{equation}
and the coefficients  $A(t)$ and $B(t)$ satisfy  a certain system of ODEs.  
The King solutions are {\em not
solitons} and  may be  visualized, as $t \to -\infty$,    as two Barenblatt self-similar solutions  ''glued'' together
to form a compact solution to the Yamabe  flow. They are type I ancient
solutions. 

\medskip
Let us make the analogy  with the Ricci flow on $S^2$. The two  explicit compact ancient solutions to the two dimensional Ricci flow are the contracting spheres   and the King-Rosenau solutions
\cite{K1}, \cite{K2}, \cite{R}. The latter ones are the analogues of the King solution (\ref{eq-king}) to the Yamabe flow. The difference is that the King-Rosenau solutions are type II ancient solutions to the Ricci flow  while
the King solution above is  of type I. 

It has been showed by Daskalopoulos, Hamilton and Sesum  \cite{DS1} that the spheres and the King-Rosenau solutions are the only compact ancient solutions to the two dimensional Ricci flow. The natural question to raise is whether the analogous statement holds  true for the Yamabe flow, that is, whether the contracting  spheres and the King solution are the only compact ancient solutions to the Yamabe flow. This occurs not to be the case as the following discussion shows.

\medskip 
In this article we will construct  ancient radially symmetric solutions of the Yamabe flow \eqref{eq-YFS}
on $S^n$ other than the contracting spheres \eqref{eq-spheres} and the King solutions
\eqref{eq-king}. Our new solutions, as $ t \to -\infty$, may be visualized as two spheres joint by a short neck. Their curvature operator changes sign and they are 
type II ancient solutions. 

\begin{remark}
Our construction can  be generalized to give ancient solutions which,   as $ t \to -\infty$, may be visualized 
as   a tower of $n$ spheres   joint by  short necks. We refer to  them as {\em moving towers of bubbles}.  
\end{remark}

\medskip

Before we present the ansatz of our construction we will perform a change of variables 
switching  to cylindrical coordinates.  Let  $g=\bar v^{\frac 4{n-2}} (\cdot,t) \, g_{_{\R^n}}$ be  a radially symmetric  solution of \eqref{eq-ufd} which becomes extinct at time $T$, namely $\bar v=\bar v(r,t)$ is a radial function on $\R^n$ that vanishes at $T$.  
One may introduce the  cylindrical change of variables
$$
 u(x,\tau )   = (T-t)^{-\frac 1{p-1}} r^{\frac 2{p-1}} \, \bar v(r,t)  , \quad r = e^x,\ t= T(1-e^{-\tau})   .
$$
In this language equation \eqref{eq-ufd}  becomes
\begin{equation}\label{eqn-vv}
(u^p)_\tau =  u_{xx} +  \alpha u^p -  \beta u, \quad \beta = \frac {(n-2)^2}4 ,\quad \alpha = \frac p{p-1} = \frac {n+2}4.
\end{equation}
From now on we will  denote the time $\tau$ by $t$. 
By  suitable scaling we can make the two constants $\alpha$ and $\beta$ in \eqref{eqn-vv} equal to 1, so that from now on
we will consider  the equation 
\begin{equation}\label{eqn-v}
(u^p)_t =  u_{xx} +  u^p -   u.
\end{equation}
The steady states of equation \eqref{eqn-v}, namely the solutions $w$  of the equation
\begin{equation}\label{eqn-w}
w_{xx} +  w^p -   w = 0, \qquad w(\pm \infty) = 0
\end{equation}
are given in closed form 
\begin{equation}\label{eqn-w1}
w(x) = \left ( \frac{k_n \, \la  \, e^{\gamma  x}}{1 +\la^2 \, e^{2\gamma \, x}} \right )^{\frac {n-2}2}
 = ( 2\, k_n\,\mbox{sech} (\gamma x+\log\la) )^{\frac {n-2}2}
 \end{equation}
 with 
 $$ \gamma = \frac 1{\sqrt{\beta}} = \frac 2{n-2} \qquad \mbox{and} \qquad   k_n = \left (\frac {4n}{n-2} \right )^{\frac 12}.$$
It is known that  $w(x)$ is the only even, positive  solution of \eqref{eqn-w}, given in cylindrical coordinates, after stereographic projection, geometrically representing the conformal metric for a sphere. Observe that 
\begin{equation}\label{eqn-w2}
w(x) = O(e^{- |x|}), \qquad \mbox{as}\,\,\  |x|\to \infty.   
\end{equation}

\medskip

We will construct   new evolving {\em ancient} compact metrics which look,  for  $t$ close to $-\infty$, 
like two spheres glued by a narrow neck. We choose our {\em ansatz}  for an ancient  solution  $u(x,t)$ of \eqref{eqn-v} to be of  the form
\begin{equation}
\label{eq-form100}
u(x,t) = (1+\eta(t))\,  z(x,t) + \psi(x,t)
\end{equation}
with
\begin{equation}
\label{eq-z100}
z(x,t) = w(x+\xi(t)) + w(x-\xi(t))
\end{equation}
for  suitable parameter functions $\eta(t), \xi(t)$. The perturbation function $\psi(x,t)$ will converge to zero,   as  
$t\to-\infty$,  in a suitable norm that  will be defined  below. 
More precisely,
$$
\xi(t) = \xi_0(t) + h(t) 
$$
for a suitable parameter function $h(t)$. Both parameter functions $h(t)$ and  $\eta(t)$ will decay 
in $|t|$,  as $t \to - \infty$.  Let 
$$\xi_0(t) := \frac{1}{2}\log(2b\,|t|)$$ be a solution to
$$\dot{\xi} + b\, e^{-2\xi} = 0, \qquad \mbox{with} \,\,\, b:= \frac{\int_0^{\infty} w^p e^{-x}\, dx}{p\, \int_{\mathbb{R}} w'^2 w^{p-1}\, dx},$$
which is the homogeneous part of the non-homogeneous equation (\ref{eq-xi}).
As we will explain below,  equation \eqref{eq-xi} is  derived  as a consequence of  adjusting  parameters $h(t)$ and $\eta(t)$ so that our solution $\psi$ satisfies   orthogonality conditions \eqref{eqn-orth1} and \eqref{eqn-orth2}.

\medskip
The main result in this article  states as follows.

\begin{theorem}
\label{thm-main}
Let $p:=(n+2)/(n-2)$  with  $n \geq 3$. There exists a  constant $t_0=t_0(n)$ and  a radially symmetric solution $u(x,t)$ to (\ref{eqn-v})
defined on   $R \times (-\infty, t_0]$, 
of the form (\ref{eq-form100})-(\ref{eq-z100}), where the functions    $\psi:=\psi(x,t)$, $\xi:= \frac{1}{2}\log(2b\,|t|)+h(t)$, with $b >0$, and $\eta:=\eta(t)$  tend to zero in the appropriate norms, 
as $t\to\infty$.  Moreover, $u$ defines a  radially symmetric ancient solution to  the Yamabe flow \eqref{eq-YF}
on $S^n$ which is of type II   (in the sense of Definition \ref{defn-ancient}) and its  Ricci curvature changes its sign. 
\end{theorem}

\medskip
Theorem \ref{thm-main} shows that the classification of ancient solutions to the compact Yamabe flow
on $S^n$  poses a  rather difficult,  even maybe impossible task. On the other hand, it gives a new way of constructing ancient solutions.  It shows how one may   glue two ancient solutions of a parabolic equation, in our case of equation  (\ref{eqn-v}),  to construct a new ancient solution to the same equation. This parabolic gluing becomes more and more apparent as $t\to -\infty$, since as $t\to +\infty$ it is known that our conformal factor approaches the one of the standard sphere. 

Gluing techniques relying on linearization and perturbation theory have been used in many elliptic settings, such as  gluing the manifolds of constant scalar curvature to produce another manifold of constant scalar curvature (\cite{MPP}, \cite{MPU}, \cite{S}),  gluing two constant mean  curvature surfaces to produce another 
constant mean  curvature surface (\cite{K}).  Gluing techniques have been used to construct new solutions to  elliptic semilinear equations  in \cite{DDM} and \cite{GJP}. Embedded  self similar solutions of the mean curvature flow have been constructed in \cite{N} by using gluing techniques. We use such techniques   here in the parabolic setting as well. We hope our way of constructing new ancient solutions to the Yamabe flow could be adopted to other geometric flows as well.

\medskip
We will next indicate the main steps in proving Theorem \ref{thm-main}.
\begin{itemize}
\item
We   first define  the  Banach space, which  our ancient solution $u$ to \eqref{eqn-v} belongs to  and its associated norm.  We also specify the spaces for our parameter functions  $\eta(t)$ and $h(t)$ and their associated norms. 
\item Using the ansatz  (\ref{eq-form100})-\eqref{eq-z100} for our solution $u$, we  show that the
perturbation   term $\psi$ is  a solution to the  equation
\begin{equation}\label{eqn-000}
pz^{p-1}\partial_t\psi = \psi_{xx} - \psi + p z^{p-1} \psi  + pz^{p-1}E(\psi)
\end{equation}
 where  $E(\psi)$ denotes our   error term and $z$ is given by \eqref{eq-z100}. It is well known that  $w$ and $w'$ are the eigenvectors 
 of the approximating linear operator $$L_0 \psi := - \frac 1{p w^{p-1}} \big (  \psi_{xx} - \psi + p w^{p-1} \psi \big )$$ 
  corresponding to the   eigenvalues $\lambda_{-1} < 0, \lambda_0=0$ of this operator,  respectively.   It is also well known that all the other eigenvalues
  of $L_0$ are positive. 
\item
In the first part of the article we  study the linear problem 
\begin{equation}\label{eqn-linear000}
pz^{p-1}\partial_t\psi = \psi_{xx} - \psi + p z^{p-1} \psi  + pz^{p-1}f.
\end{equation}
 Assuming certain orthogonality conditions  on $f$ with respect to  the eigenvectors  $w$ and $w'$ of $L_0$, we   establish the existence  of  an ancient solution to the linear problem \eqref{eqn-linear000}, satisfying the appropriate energy and  $L^2$ estimates. The latter means that  we can bound the  weighted $L^2$-norm of a solution  in terms of the weighted $L^2$-norm of the right hand side $f$. 
We also establish certain weighted   $W^{2,\sigma}$ estimates  for solutions to \eqref{eqn-linear000}. It follows  
that the  solution
 $\psi$  belongs  to the Banach space which is the intersection of these  $L^2$ and weighted   $W^{2,\sigma}$
  spaces.  
  
 We denote by $T$  the  linear operator between our defined Banach  spaces, so that $T(f)$ is  the  solution to the linear problem \eqref{eqn-linear000} satisfying the appropriate orthogonality conditions. 

\item 
In the second part of the article  we  study the nonlinear equation \eqref{eqn-000}.  We   apply our linear theory to the nonlinear equation to 
establish the existence of a  solution  $\psi$ to \eqref{eqn-000}, by solving the equation $T(E(\psi)) = \psi$. 
We  first show that  we can achieve this,   assuming that $E(\psi)$ satisfies our orthogonality assumptions with respect to $w$ and $w'$.  The main tool in this  proof is the  fixed point Theorem and subtle estimates of the error terms in our norms. 

\item
In the final part of our proof we  show how to  adjust  the parameters $\eta(t)$ and $h(t)$ so that 
the error term $E(\psi)$ in \eqref{eqn-000} indeed satisfies our orthogonality conditions. 
We  see that this  is equivalent to solving a certain nonlinear system of ODEs  for $\eta(t)$, $h(t)$. 
We  establish the existence of solutions to this system by the fixed point Theorem and subtle estimates.  

\end{itemize}

\section{The ansatz of the problem}

Following the discussion in the Introduction  we  look for an ancient solution $u$ to the equation 
\eqref{eqn-v}.  Since the long time existence for the Yamabe flow is well understood,
 it will be sufficient to construct a solution $u$ which is defined on $\R \times (-\infty,t_0]$
with $t_0$ sufficiently close to $-\infty$. Hence, from now on we will restrict our attention to 
equation 
\be
(u^p)_t =  u_{xx} - u + u^p, \qquad (x,t)\in  \R \times (-\infty,t_0]
\label{eq}
\ee
with exponent 
$p= \frac{n+2}{n-2} >1.$

\subsection{The set up of our construction} 
We seek  for a solution of \eqref{eqn-v} which is of the form 
$$u(x,t) = (1+\eta(t))\,  z(x,t) + \psi(x,t)$$
for a suitable parameter function $\eta(t)$, where
$$
z(x,t) =
\sum_{j=1}^2 w(x- \xi_j(t) ) = \sum_{j=1}^2 w_j
$$
and $\psi(x,t)  \to 0$ as  $t\to-\infty$ in a certain sense.  We recall that $w(x)$ is given by (\ref{eqn-w1}) and
solves the equation \eqref{eqn-w}. 
The functions $\xi_j(t)$ are given by
$$\xi_1(t)= - \xi(t) \qquad \mbox{and} \qquad \xi_2(t)=\xi(t)$$
where 
\begin{equation}\label{eqn-xit}
\xi(t) = \xi_0(t) + h(t), \qquad \xi_0(t) := \frac{1}{2}\log(2b\,|t|)
\end{equation}
for a suitable parameter function $h(t)$ and a suitable constant $b >0$.
Both  parameter functions $h(t)$ and  $\eta(t)$ will decay 
in $|t|$,  as $|t| \to  \infty$ and will be chosen in section \ref{sec-system}. 

\smallskip

Set  
\begin{equation}\label{eqn-wxt}
w_1:=w(x-\xi(t)), \qquad w_2:=w(x+\xi(t)) 
\end{equation}
and 
\begin{equation}\label{eqn-barz}
\bar z(x,t) = w'(x-\xi(t)) - w'(x+\xi(t))=\pp_x w_1 - \pp_x w_2
\end{equation}
Also set  
$$\tilde{z}(x,t) := (1+\eta(t))\,  z(x,t).$$
We notice that $z(\cdot,t)$ is an even function of $x$ and we impose the condition that 
$\psi(\cdot,t)$ is an even function of $x$ as well.  Equation \eqref{eqn-v} then becomes
$$
\pp_t \big(\tilde{z}+\psi)^p  =  \, ( \pp_x^2\psi  - \psi  +  \pp_x^2 \tilde{z}  - \tilde{z} )   
+   \big(\tilde{z}+\psi )^p.$$

\no Using that  $\pp_x^2w_{j}  - w_{j}= -w_j^p$ we obtain the equation 
$$
\pp_t \big(\tilde{z}+\psi )^p  =  \big (\, \pp_x^2\psi - \psi   
- (1+\eta(t)) \sum_{j=1}^2 w_j^p\, \big  )   +   \big(\tilde{z}+\psi )^p$$
which can be re-written as  
\begin{equation}\label{eqn-Psi}
pz^{p-1}\, \pp_t \psi  = \pp_{xx} \psi - \psi + pz^{p-1}\, \psi - z^{p-1} \, C( \psi,t)+  z^{p-1}\,   E(\psi)
\end{equation}
where  $C(\psi,t)$ is a correction term that will be chosen in \eqref{eqn-Cpsi22}. 

\no The error term $E(\psi)$ is given by 
\begin{equation}
\label{eq-Epsi} 
E(\psi) := z^{1-p} \, M + \underbrace{C(\psi,t) + z^{1-p}  \big [\,  (1-\partial_t) N(\psi)  - p\psi \, \pp_t  z^{p-1}
\, \big ]}_{Q(\psi)}
\end{equation}
where 
\be\label{eqn-M}
M := \tilde{z}^p - \big [\, (1+\eta(t))\, [ w^p(x+\xi(t)) + w^p(x-\xi(t))] \, \big ] - \pp_t \tilde{z}^p
\ee
is the error term that is independent of $\psi$,  and 
\be\label{eqn-N}
N(\psi) :=(\tilde{z}+\psi)^p -  \tilde{z}^p -  p\, \tilde{z}^{p-1} \psi + p\psi \, (\tilde{z}^{p-1} - z^{p-1}).
\ee

\medskip

\no Our goal is to construct an ancient  solution $\psi$ of the above equation  \eqref{eqn-Psi}  with the aid of the linear theory for equation \eqref{eqn-Psi}  that  will developed  in section \ref{sec-linear}.   The solution $\psi$ will be an even function in $x$ and it will
satisfy the orthogonality conditions
\begin{equation}\label{eqn-orth11}
\int_{-\infty}^\infty \psi (x-\xi(t),t) \, w'(x) \, w^{p-1} \, d x=0, \qquad a.e. \quad t < t_0
\end{equation}
and
\begin{equation}\label{eqn-orth22}
\int_{-\infty}^\infty \psi  (x-\xi(t),t) \, w( x) \, w^{p-1}  \, d  x=0, \qquad a.e. \quad t < t_0.
\end{equation}

\smallskip

\no  The {\em correction term}    $C(\psi,t)$ in equation \eqref{eqn-Psi}   will be chosen  in section \ref{sec-linear}, 
in such  a way so that the orthogonality conditions \eqref{eqn-orth11}-\eqref{eqn-orth22}  for $\psi$ are  being preserved by the equation
\eqref{eqn-Psi},  given that  $E(\psi)$ satisfies the same conditions. 

\smallskip

\no However, because in general the error  term $E(\psi)$ may not  satisfy the orthogonality conditions 
\eqref{eqn-orth11}-\eqref{eqn-orth22},   we will  first consider the auxiliary equation

\be\label{eqn-system}
\begin{split}
p z^{p-1}\, \pp_t \psi  = \pp_{xx} \psi - \psi &+ pz^{p-1}\, \psi - z^{p-1} C(\psi,t) \\
&+z^{p-1}\, [ \,  E(\psi) - (c_1(t) \, z + c_2(t) \, \bar  z)    \, ]
\end{split} 
\ee
where $c_1(t)$ and $c_2(t)$ are  chosen so that 

\be
\label{eqn-Epsi3}
\bar E(\psi):= E(\psi) - (c_1(t) \, z + c_2(t) \, \bar  z)  
\ee
\no satisfies the orthogonality conditions \eqref{eqn-orth11}-\eqref{eqn-orth22}, given that  $\psi$  satisfies them.
\smallskip

In section \ref{sec-system} we will choose  the parameter functions $h$ and $\eta$ so that $c_1(t) \equiv 0$ and
$c_2(t)  \equiv 0$. The parameter  functions $h$ and $\eta$ will decay in $t$, as $t \to -\infty$,  in certain norms that 
will be  defined in  Definition \ref{dfn-heta}.

\subsection{Norms}
\label{sec-norms}

We will  next introduce all the norms that will be used throughout the article. We will also fix 
the values of the various parameters. 
For a number $\tau < t_0-1$ we set $\Lambda_\tau = \R  \times [\tau,\tau+1]$.

We first  define the appropriate  $L^2$, $H^1$ and $H^2$ norms. 

\begin{definition}\label{defn1} [Local in time weighted $L^2$, $H^1$ and $H^2$ norms] 
Define 
$$\|\psi (\cdot,\tau) \|_{L^2} = \big (\inti  |\psi(\cdot,\tau)|^2 \, \zz^{p-1}  \, dx  \big )^{\frac 12}$$
and 
$$\|\psi \|_{L^2(\Lambda_\tau)} = \big (\iint_{\Lambda_\tau}  |\psi|^2 \, \zz^{p-1}  \, dx dt \big )^{\frac 12}$$
and 
$$
\|\psi \|_{H^1(\Lambda_\tau)} =  \|\psi \|_{L^2(\Lambda_\tau)}  + 
\| \zz^{-\frac{p-1}2} \, \psi_x  \|_{L^2(\Lambda_\tau)}$$
and
$$
\|\psi \|_{H^2(\Lambda_\tau)} = \|\psi_t \|_{L^2(\Lambda_\tau)} +  
\| \zz^{-(p-1)} \, (\psi_{xx}-\psi)  \|_{L^2(\Lambda_\tau)} +
 \| \zz^{-\frac{p-1}2} \, \psi_{xx} \|_{L^2(\Lambda_\tau)} + \|\psi \|_{H^1(\Lambda_\tau)}.$$
\end{definition}

%

\begin{definition}\label{defn2} [Global in time weighted $L^2$, $H^1$ and $H^2$ norms] For a given 
number $\nu \in [0,1)$
we define 
$$\|\psi \|_{L^2_{t_0}}^\nu  = \sup_{\tau \leq t_0-1}| \tau|^\nu \|\psi \|_{L^2(\Lambda_\tau)}$$
and
$$ \|\psi \|_{H^1_{t_0}}^\nu  = \sup_{\tau \leq t_0-1} |\tau|^\nu \|\psi \|_{H^1(\Lambda_\tau)}
$$
and
$$
\|\psi \|_{H^2_{t_0}}^\nu  = \sup_{\tau \leq t_0-1} |\tau|^\nu \|\psi \|_{H^2(\Lambda_\tau)}.
$$

Also, for any  $s < t_0-1$, we define 
$$\|\psi \|_{L^2_{s,t_0}}^\nu  = \sup_{s \leq \tau \leq t_0-1}| \tau|^\nu \|\psi \|_{L^2(\Lambda_\tau)}$$
and
$$ \|\psi \|_{H^1_{s,t_0}}^\nu  = \sup_{s \leq \tau \leq t_0-1} |\tau|^\nu \|\psi \|_{H^1(\Lambda_\tau)}
$$
and
$$
\|\psi \|_{H^2_{s, t_0}}^\nu  = \sup_{s \leq \tau \leq t_0-1} |\tau|^\nu \|\psi \|_{H^2(\Lambda_\tau)}.
$$

\no When $\nu=0$ we will omit the superscript $\nu$. 
\end{definition}

\ssk 
Set
$$ \beta :=\frac 2{n-2}=\frac{p-1}2.$$

\ssk

\no For a given number $\sigma \ge 2$, we next give  the definition of the weighted $W^{2,\sigma}$ norm and $\sigma$ will be chosen later in the text. To this end, we define the weight function
$\alpha_\sigma(x,t)$ by 
\begin{equation}\label{eqn-weight}
\alpha_\sigma(x,t)=
\begin{cases}
\zz^{n\beta-\sigma}(x,t) &\text{if  $|x|>\xi(t)$}\\
\zz^{(2\beta+\theta)\sigma }(x,t) \quad &\text{if  $|x|\leq \xi(t)$}
\end{cases}
\end{equation}
where $\theta $ is a small positive number which will be chosen sufficiently close to zero. 

\medskip

\begin{remark} 
(a) We will see in the sequel that the weight function in the outer region $|x| > \xi(t)$ is such that the solution
$\phi$ of \eqref{eqn-linear000}, or equivalently the solution $u$ of the nonlinear problem  corresponds to a smooth
solution, when lifted up to the sphere. 
However, it is necessary to change the weight function $\alpha_\sigma$ in the inner region $|x| \leq \xi(t)$
to incorporate the singularity, as $t \to -\infty$,  of the solution $u$ of our nonlinear problem in that region. 

(b) In the transition region  $x =\pm \xi(t) + O(1)$ the two weights are equivalent.

\end{remark}

\smallskip

\begin{definition} [Local in time weighted $W^{2,\sigma}$ norms] \label{defn-w2p1} For $\sigma \ge 2$ define 
$$\|\psi \|_{\sigma,\Lambda_\tau} = \big (\iint_{\Lambda_\tau}  |\psi|^\sigma \, \alpha_\sigma \, dx dt \big )^{\frac 1\sigma}$$
and
$$
\|\psi \|_{2,\sigma,\Lambda_\tau} =  \|\psi_t \|_{\sigma,\Lambda_\tau} + \|\psi \|_{\sigma,\Lambda_\tau}  +
\| \psi_x  \|_{\sigma,\Lambda_\tau} + \| \psi_{xx} \|_{\sigma,\Lambda_\tau}.$$

\end{definition}

\begin{definition} [Global in time weighted $W^{2,\sigma}$ norms] \label{defn-w2p2}  For a fixed number 
$\nu \in [0,1)$, $\sigma \ge 2$ and $\Lambda_\tau$ as above,  we  define 
$$\|\psi \|_{\sigma,t_0}^\nu  = \sup_{\tau \leq t_0-1} |\tau|^\nu \, \|\psi \|_{\sigma,\Lambda_\tau}$$
and
$$
\|\psi \|_{2,\sigma,t_0}^\nu = \sup_{\tau \leq t_0-1} |\tau|^\nu \, \|\psi \|_{2,\sigma,\Lambda_\tau}.$$

\end{definition}

\no We will next define a weighted $L^\infty$ norm and our global norm. 

\begin{definition}[Weighted $L^\infty$ norm] \label{defn-linfty} For a given  $\nu\in [0,1]$,  we  define the norm
$$\| \psi \|^\nu_{L^\infty_{t_0}} := \sup_{t \leq t_0} |\tau|^\nu  \| \psi(\cdot,\tau) \|_{L^\infty(\R)}$$
and the  weighted $L^\infty$ norm as 
$$
\| \psi \|_{\infty, t_0}^\nu := 
\| z^{-1} \, \psi \,  \chi_{\{|x| > \xi(t)\}}\|^\nu_{L^\infty_{t_0}} + \| \psi  \,  \chi_{\{|x| \leq \xi(t)\}}  \|^\nu_{L^\infty_{t_0}}. 
$$
\end{definition}

\no We  finally define the global norm for the perturbation term $\psi$.
\medskip

\begin{definition}[Global norm]\label{defn-norm}
For  $\nu \in [0,1)$ and $\sigma \ge 2$ we define 
the norms
$$\|\psi\|_{*,\sigma,t_0}^{\nu} := \|\psi\|_{L_{t_0}^2}^{\nu} + \|\psi\|_{\sigma,t_0}^{\nu}$$
and
$$\|\psi\|_{*,2,\sigma,t_0}^{\nu} := \|\psi\|_{H_{t_0}^2}^{\nu} + \|\psi\|_{2,\sigma,t_0}^{\nu} + \|\psi\|_{\infty,t_0}^{\nu}.$$
Also, for any $\tau < t_0$, we denote by
$$\| \psi\|_{*,\sigma,\Lambda_\tau}  := \|\psi \|_{L^2(\Lambda_\tau)} + \|\psi\|_{\sigma,\Lambda_\tau}.$$
\end{definition}

\no We will next define the  norms for the parameters $\eta(t)$ and $h(t)$. They are more or less determined by the choice of the global norm for $\psi$.

\begin{definition}[Weighted in time norms] \label{dfn-heta} For  $\mu \in [0,1)$  and $\sigma \ge 2$,  
and for any functions $\eta$ and $h$ defined on $(-\infty,t_0]$, 
we define   the  norms  
$$\| \eta \|^\mu_{\sigma,t_0} = \sup_{\tau \leq t_0-1} |\tau|^\mu \big ( \int_{\tau}^{\tau+1}  |\eta (t)|^\sigma \, dt \big )^{\frac 1\sigma}, \qquad \| \eta  \|^\mu_{\infty,t_0} = \sup_{\tau \leq t_0} (|\tau|^\mu |\eta (\tau)|),$$
$$\| \eta \|^\mu_{1,\sigma,t_0} := \| \eta \|^\mu_{\infty,t_0}  + \| \dot \eta  \|^\mu_{\sigma,t_0},$$
\smallskip
and
$$\|h\|^{\mu,1+\mu}_{1,\sigma,t_0} :=    \| h \|^\mu_{\infty,t_0} + \| \dot h \|^{1+\mu}_{\sigma,t_0}.$$ 

\end{definition}

%
%
%
\subsection{Outline of our construction} \label{sec-outline}

We will conclude this section by outlining the construction of the solution $u$. 

\begin{definition}\label{defn-X} We define $X$ to be  the  Banach space of all functions $\psi$   on $\R \times 
(-\infty,t_0]$ with $\| \psi\|_{*,2,\sigma,\nu} < \infty$ which also satisfy  the 
orthogonality conditions \eqref{eqn-orth11}-\eqref{eqn-orth22}.

\end{definition}

\smallskip
We  denote by $T$ the linear operator which assigns to any given $f$ with  $ \|f\|_{*,\sigma,t_0}^{\nu} < \infty$ the solution $\psi:=T(f)$ of the linear  auxiliary equation 
$$pz^{p-1}\partial_t\psi = \partial_{xx}\psi  - \psi + pz^{p-1}\psi - z^{p-1}\, C(\psi,t) + z^{p-1} f,$$
with orthogonality conditions (\ref{eqn-orth11})-(\ref{eqn-orth22}) being satisfied by $f$ and $\psi$ and with $C(\psi,t)$ given by \eqref{eqn-Cpsi22}.
The construction of such  function $\psi$ will be given in section \ref{sec-linear}. 

Going back to the nonlinear problem,   function $\psi$ is a   solution of  \equ{eqn-system} iff $\psi \in X $ solves the fixed point problem
\be\label{eqn-fp}
\psi = A(\psi)
\ee
where
$$
A(\psi) :=  T(\bar E(\psi) )
$$
and $\bar E(\psi)$ is as in \eqref{eqn-Epsi3}. 

\medskip

\no {\em Outline:} Given any parameter functions $(h,\eta)$ with $\|h\|^{\mu,1+\mu}_{1,\sigma,t_0} < \infty$
and $\| \eta \|^\mu_{\sigma,t_0} < \infty$ we will establish, in section \ref{sec-np},  the existence of a solution $\psi:=\Psi(h,\eta)$ of the
fixed point problem \eqref{eqn-fp}.  In the last section \ref{sec-system} we will choose  the parameter functions $h$ and $\eta$ so that $c_1(t) \equiv 0$ and
$c_2(t)  \equiv 0$.  We will conclude  that the solution $\psi$ of \eqref{eqn-fp} which is equivalent to \equ{eqn-system}
is actually a solution to \eqref{eqn-Psi}. Hence, $u:=(1+\eta) \, z + \psi$ will be the desired ancient solution to 
\eqref{eqn-v}. 

\subsection{Notation} We summarize now  the notation of parameters, functions and   norms 
 used  throughout the article. 

\begin{notation}[The choice of  the parameters $p, \beta, \sigma, \, \nu, \, \mu, b$ and $\theta$]\label{not-par}
i. For a given dimension $n \geq 3$, we recall that $p:= (n-2)/(n+2)$ and $\beta:=2/(n-2)=(p-1)/2.$ \\
\no ii.  \no In Theorem \ref{thm-main}, we have  $\sigma = n+ 2$.  We choose $\nu$ so  that $\frac 12 < \nu  < \min\{\nu_0,1\}$,  where  $\nu_0=\nu_0(n)$ is determined by the estimate of  Lemma \ref{lem-M}. 
We choose $0 < \mu < \min\{2\nu - 1, \gamma\}$, where $\gamma=\gamma(n) \in (0,1)$ is determined by Lemma \ref{lem-projection}. The constant $b=b(n) >0$ is defined in \eqref{eqn-ab}. \\ 
\no iii. The constant  $\theta$ in \eqref{eqn-weight}
is a small positive constant as determined in  the proof of Proposition \ref{prop-w2p}.  
The above constants are all universal depending only on the dimension $n$. 
\end{notation}

\smallskip
\begin{notation}[The choice of functions] \label{not-fun}
\no i. We denote by $w(x)$ the solution to  \eqref{eqn-w} given by \eqref{eqn-w1}.\\
\no ii. The function  $\xi_0(t)$ and the function $\xi(t)$  (for a parameter  function $h(t)$)
are  given in  \eqref{eqn-xit}. \\
\no iii. The functions $z(x,t)$ and $w_1(x,t), w_2(x,t)$  are defined in  \eqref{eq-z100} and \eqref{eqn-wxt} respectively.\\
\no iv.  Throughout the article,   $u(x,t)$ will denote an ancient  solution of the nonlinear
equation \eqref{eqn-v} of the form \eqref{eq-form100}  defined on $\R^n \times (-\infty,t_0]$,  where $t_0$ is a constant which will be chosen sufficiently close to $-\infty$. $\eta(t)$ is a parameter function defined on $(-\infty,t_0]$. The perturbation   function $\psi(x,t)$ satisfies equation \eqref{eqn-Psi} where $E(\psi)$ is a non-linear error term give by \eqref{eq-Epsi}-\eqref{eqn-N}.\\
\no  v. Only in section  \ref{sec-linear},  $\psi(x,t)$  will denote a  solution  to  equation \eqref{eqn-psi}, 
for a given $f$,  where the correction term $C(\psi,t)$ is given by \eqref{eqn-Cpsi22}.   Also, $\psi^s(x,t)$  will denote a  solution  of equation \eqref{eq-psi}.  
\end{notation}

\smallskip
\begin{notation}[The norms] \label{not-norms} \no i. For a given $\tau < t_0$,  the  norm $\|\psi (\cdot,\tau) \|_{L^2} $  is given in Definition \ref{defn1}.\\
\no ii. For any $\tau < t_0-1$ we set   $\Lambda_\tau := \R \times [\tau,\tau+1]$. 
 For a given function $\psi(x,t)$ defined on $\Lambda_\tau$, the norms $\|\psi \|_{L^2(\Lambda_\tau)}$, 
$\|\psi \|_{H^1(\Lambda_\tau)}$ and $ \|\psi \|_{H^1(\Lambda_\tau)}$ are given in Definition \ref{defn1}. \\
\no iii. The norms $\|\psi \|_{L^2_{t_0}}^\nu$, 
$ \|\psi \|_{H^1_{t_0}}^\nu$ and $ \|\psi \|_{H^2_{t_0}}^\nu$  are given in Definition \ref{defn2}. \\
\no iv. The norms $\|\psi \|_{\sigma,\Lambda_\tau}$,  $\|\psi \|_{2,\sigma,\Lambda_\tau}$ are given in Definition \ref{defn-w2p1}, while the norms $\|\psi \|_{\sigma,t_0}^\nu $, $\|\psi \|_{2,\sigma,t_0}^\nu $  are given in Definition \ref{defn-w2p2}. \\
\no v. The weighted  $L^\infty$ norm $\| \psi \|^\nu_{L^\infty_{t_0}}$ is given in Definition \ref{defn-linfty}. \\
\no vi. The global norms $\|f\|_{*,2, \sigma,t_0}^{\nu} $, $ \|w\|_{*,\sigma,t_0}^{\nu} $   are given in Definition \ref{defn-norm}. \\
\no vi. For given functions $h(t)$ and $\eta(t)$, the norms $\|h\|^{\mu,1+\mu}_{1,\sigma,t_0}$ and $\| \eta \|^\mu_{\sigma,t_0}$ are given in Definition \ref{dfn-heta}. 
\end{notation}

\section{The Linear Equation}\label{sec-linear}

Consider the linear equation
\begin{equation}\label{eqn-psi0}
p\zz^{p-1} \, \pp_t \psi = \pp_{xx} \psi - \psi  + p \zz^{p-1} \, \psi + \zz^{p-1} \, g
\end{equation}
defined on $-\infty <  t \leq  t_0$. The coefficient $\zz$ is given  by
\begin{equation}\label{eqn-wxt1}
\zz(x,t) = w(x-\xi(t)) + w(x+\xi(t)),
\end{equation}
where $\xi(t)$ is given by \eqref{eqn-xit} 
for a suitable function $h \in C^1((-\infty,t_0])$ and $b >0$.
Note that $\zz(\cdot,t)$ is even in $x$. We will also   impose that $g(\cdot,t)$ is an even function  in $x$ and we shall seek for a solution $\psi(\cdot,t)$  which is even in $x$. 
We will consider a class of functions $g$  defined for  $(x,t)\in \R\times (-\infty, t_0]$
that decay both in $x$ and $t$ at suitable rates, and will build a solution $\psi$ that defines a linear
operator of $g$ which shares the same decay rates.

\msk

Our goal is to  establish the existence of the solution $\psi$ of \eqref{eqn-psi0}  in  appropriate $L^2$ and $H^1$ spaces, defined in Definition \ref{defn2}.
We observe that in the region $-\infty < x < -\xi(t)$ and under the change of variables $\bar x :=x+\xi(t)$
the operator in \eqref{eqn-psi0}, namely 
$$L \psi := \frac 1{\zz^{p-1}} \big ( \psi_{xx}  - \psi + p\zz^{p-1} \psi \big )$$
can be approximated by the elliptic operator 
\begin{equation}\label{eqn-oper}
L_0\phi := \frac 1{w^{p-1}} \big ( \phi_{\bar x\bar x} - \phi + pw^{p-1} \phi \big ),
\end{equation}
with $\phi(\bar x,t) := \psi(x,t)$, since $z(x,t) = w(\bar x) + w(\bar x-2\xi(t)) \approx w(\bar x)$ in that region. 
Defining $\bar g(\bar x,t) :=g(x,t)$, the approximated
parabolic equation takes the form 
\begin{equation}\label{eqn-aeq}
p\, w^{p-1} \, \pp_t \phi = \pp_{\bar x \bar x} \phi - \phi_{\bar{x}}\, \dot{\xi} - \phi  + p w^{p-1} \, \phi + w^{p-1}  \, \bar g
\end{equation}
The region  $\xi(t) < x < +\infty$, under the change of variables $\bar x :=x-\xi(t)$ is treated similarly.
 
\ssk

We wish to construct an ancient solution $\psi$  of \eqref{eqn-psi0}, such that the  weighted $L^2$ norm
of $\psi$  is controlled by the weighted $L^2$ norm of 
the right hand side $g$. In order to do that  we  use the well known spectrum of the approximated 
operator $L_0$ defined above.   Consider the   eigenvalue problem
$$
L_0 \theta   +    \lambda \,  \theta =0, \qquad \theta\in S,
$$
\no on the weighted space $L^2( w^{p-1}dx)$.  It is standard that this problem  has an infinite sequence
 of simple eigenvalues
$$\la_{-1} < \la_0  = 0 < \la_1<\la_2<\cdots .$$

\no with an associated orthonormal basis of $L^2( w^{p-1}dx)$ constituted by eigenfunctions
$\theta_j$, $j=-1,0,1,\ldots$, where  $\theta_{-1}$ is a suitable multiple of $w$ and $\theta_0$ of $w'$.
Since we are seeking for a solution which is controlled by the weighted $L^2$ norm
of its right hand side, we need to  restrict ourselves to a subspace $S_0 \subset L^2( w^{p-1}dx)$ which
constitutes of functions $g(\cdot,t)$
on $\R \times (-\infty,t_0]$  that are even in $x$ and that also satisfy the orthogonality conditions
\begin{equation}\label{eqn-orth1}
\int_{-\infty}^\infty g(\bar x-\xi(t),t) \, w'(\bar x) \, w^{p-1} \, d\bar x=0, \qquad a.e. \quad t < t_0
\end{equation}
and
\begin{equation}\label{eqn-orth2}
\int_{-\infty}^\infty g (\bar x-\xi(t),t) \, w(\bar x) \, w^{p-1}  \, d\bar x=0, \qquad a.e. \quad t < t_0.
\end{equation}
Notice that since  $g$ is an even function in $x$, then the orthogonality conditions \eqref{eqn-orth1} and
\eqref{eqn-orth2} also imply the symmetric conditions 
\begin{equation}\label{eqn-orth10}
\int_{-\infty}^\infty g (\bar x+\xi(t),t) \, w'(\bar x) \, w^{p-1}  \, d\bar x=0, \qquad a.e. \quad t < t_0
\end{equation}
and
\begin{equation}\label{eqn-orth20}
\int_{-\infty}^\infty g (\bar x+\xi(t),t) \, w(\bar x) \, w^{p-1}  \, d\bar x=0, \qquad a.e. \quad t < t_0.
\end{equation}
This easily follows by changing the variables of integration and using that $w$ is an even function
of $\bar x$.


%

%


  
\medskip

We wish to establish the existence of an ancient solution of \eqref{eqn-psi0} 
on $\R \times (-\infty,t_0]$ which satisfies the estimate 
$$\sup_{\tau \leq t_0} |\tau|^\nu  \| \psi(\tau) \|_{L^2} \leq C \, \| g\|_{L^2_{t_0}}^\nu.$$ 
Such a solution can be easily constructed for the approximated equation \eqref{eqn-aeq}
if $\bar  g \in S_0$. Indeed one simply looks for a solution in the from 
$$
\phi(\bar x,t) = \sum_{j=1}^\infty  \phi_j(t)\theta_j( \bar x)
$$ 
where $\theta_j$, $j\geq 1$ are the  eigenfunctions,  corresponding to the positive eigenvalues
$\lambda_j$, $j \geq 1$ mentioned above. 
However, this cannot be done for equation \eqref{eqn-psi0} as its coefficients depend on time
and as a result the equation doesn't preserve the orthogonality conditions   \eqref{eqn-orth1} and \eqref{eqn-orth2}. 
In order to achieve our goal we need to consider the equation
\begin{equation}\label{eqn-psi}
p\zz^{p-1} \, \pp_t \psi = \pp_{xx} \psi - \psi  + p \zz^{p-1} \, \psi + \zz^{p-1} \, \big [ f - C(\psi,t) \big ],
\end{equation} 
where $g := f - C(\psi,t)$,
for a suitable correction  function  $C(\psi,t)$ of  the form
\be \label{eqn-Cpsi22}
C(\psi,t) = d_1(\psi,t)  \, \zz(x,t) +   d_2(\psi,t)  \, \bar \zz(x,t),
\ee
and where $f \in S_0$.
Recall that  $\zz(x,t) := w(x-\xi(t)) + w(x+\xi(t))$ and  that $\bar \zz$ is defined  by
\eqref{eqn-barz}.

\medskip

We will construct an ancient  solution of \eqref{eqn-psi} on $\R \times (-\infty, t_0]$,
by considering first the solution $\psi^s$ of the initial value problem 
\begin{equation}
\label{eq-psi}
\begin{cases}
p z^{p-1}\, \partial_t \psi^s = \psi^s_{xx} - \psi^s +p  z^{p-1}\, \psi^s + z^{p-1}\, (f - C(\psi^s,t))
 \qquad  & \mbox{on} \,\,\,  \R  \times [s,t_0]\cr
\psi^s(\cdot,0) = 0  \qquad & \mbox{on} \,\,\, \R 
\end{cases}
\end{equation}
and then pass  to the limit as $s \to -\infty$.  The existence of  $\psi^s$ will be shown in Lemma 
\ref{lem-es}. 

\smallskip

The coefficients $d_1(t)$ and $d_2(t)$ in \eqref{eqn-psi} are defined so  that $\psi^s(\cdot,t) \in S_0$
for all $t\in [s,t_0]$, given that $f\in S_0$. We will next determine the coefficients $d_1$ and $d_2$. To this end,  it is more
convenient to work with the function  
$$\phi^s(x,t) := \psi^s(x-\xi(t),t).$$
To simplify  notation we  omit for the moment the superscript $s$ and set $\phi=\phi^s$ and $\psi=\psi^s$.  
A direct computation shows that if $\psi$ is a solution  to \eqref{eqn-psi}, then the function $\phi$ satisfies
the equation
\begin{equation}\label{eqn-phi}
p \,  \pp_t \phi = L_0 \phi  + E(\phi)+   \big ( \bar f - d_1\, \bar w - d_2 \, \tilde  w \big ).  
\end{equation} 
Here we have used the following notation 
$$\bar w(x,t) := z(x-\xi(t),t)=w(x) + w(x-2\xi(t))$$
and
$$\tilde w := \bar z(x-\xi(t))= w'(x) - w'(x-2\xi(t))$$
and $\bar f(x,t) :=f(x-\xi(t),t)$, while $E(\phi)$ denotes the error term
$$E(\phi) := -\dot \xi(t) \,  \phi_x  + \big ( \frac{1}{\bar w^{p-1}}-\frac{1}{w^{p-1}} \big ) \, (\phi_{xx} - \phi).$$
We recall that $L_0$ is given by \eqref{eqn-oper}. 
Also recall that  $\theta_i$, $i=-1,0$  denote   the eigenfunctions (which are the multiples  of $w$, $w'$)  of  operator $L_0$,
corresponding to the eigenvalues $\lambda_{-1} <0$ and $\lambda_0=0$,  respectively.  
We have assumed that $\bar f$ is orthogonal to $\theta_i$, $i=-1,0$,  namely
$$\inti \bar f(x)\, \theta_i(x) \, w^{p-1}\, dx=0.$$
Since $\phi(\cdot, s)=0$ (remember that $\phi=\phi^s$ for the moment),  it follows from \eqref{eqn-phi} that the solution $\phi$ will remain orthogonal
to the eigenfunctions $\theta_i$, $i=-1,0$, if and only if the coefficients  $d_1(t)$ and $d_2(t)$  satisfy the system
of equations 
\begin{equation}\label{eqn-system000}
d_1(t) a_1^i(t)  + d_2(t) a_2^i(t)=E^i, \quad i=-1,0
\end{equation}
where 
$$a_1^i(t)=\inti \bar w(x,t) \, \theta_i(x)  w^{p-1} \, dx, \quad a_2^i(t)=\inti \tilde w(x,t) \, \theta_i(x) w^{p-1} \, dx.
$$
and
$$E^i:=\inti E(\phi)(x,t)\, \theta_i(x) w^{p-1} \, dx, \qquad i=-1,0.$$
Using that  $\bar w(x,t) := w(x) + w(x-2\xi(t))$ and  $\tilde w(x,t) := w'(x) - w'(x-2\xi(t))$ together with the orthogonality 
$$\inti w(x) \theta_0(x) w^{p-1}\, dx = \inti w'(x) \theta_{-1}(x) w^{p-1}\, dx =0,$$
we conclude that
$$a_1^{i}(t)=e_1^i + \inti w(x-2\xi(t)) \theta_i(x)\, w^{p-1}\, dx,  \quad a_2^{i}(t)=e_2^i - \inti w'(x-2\xi(t)) \theta_i(x)\, w^{p-1}\, dx$$
where 
$$e_1^{-1}= c_{-1}\, \inti w^{p+1}\, dx >0,  \quad e_2^0= c_0\, \inti w'(x)^2 w^{p-1}\, dx>0, \quad  e_1^{0}= e_2^{-1}=0.$$
It is easy to see that 
$$\inti w(x-2\xi(t)) \theta_i(x)\, w^{p-1}\, dx=O(|t|^{-1}), \quad \inti w'(x-2\xi(t)) \theta_i(x)\, w^{p-1}\, dx=O(|t|^{-1})
$$
as $t \to -\infty$. 
Hence,
$$a_1^{-1}(t)=e_1^{-1}+ O(|t|^{-1}),   \quad a_2^{0}(t)=e_2^0 + O(|t|^{-1}), \quad a_1^0=O(|t|^{-1}),
\quad a_2^{-1}=O(|t|^{-1}).$$
It follows that  the determinant $D$ of the coefficients of the system \eqref{eqn-system000} satisfies 
$$D:=a_1^{-1}a_2^0-a_1^0a_2^{-1} = e_1^{-1}e_2^0+ O(|t|^{-1}) >0, \qquad \mbox{as}\,\, t \to -\infty.$$
Solving the system \eqref{eqn-system000} gives
$$
d_1(t) = \frac{a_2^0(t)E^{-1}(t)-a_2^{-1}(t)E^0(t)}{D} = \frac{e^2_0E^{-1}(t)}{D} + O(|t|^{-1}) \quad
$$
and
$$
d_2(t) = \frac{a_1^{-1}(t)E^{-1}(t)-a_1^{0}(t)E^0(t)}{D} = \frac{e_1^{-1}E^0(t)}{D} + O(|t|^{-1}).$$

\begin{claim} We have
\begin{equation}\label{eqn-c1c2}
\left(\int_\tau^{\tau+1} (d_1^2 + d_2^2)\, dt \right)^{1/2}  \leq C \,\frac 1{\sqrt{|\tau|}} \,  \left \{ \left (  \iint_{\Lambda_\tau} 
 \big ( \frac {\phi_{xx} - \phi}{\bar w^{p-1}} \big )^2   \bar w^{p-1} \, dx \, dt\right  )^{\frac 12} + \|\phi\|_{L^{\infty}(\Lambda_\tau)}\right  \}.
\end{equation}

\end{claim}

\begin{proof}[Proof of Claim]
It is easy to see that
\begin{equation}\label{eqn-d1d2}
\left(\int_\tau^{\tau+1}\, (d_1^2 + d_2^2)\, dt\right)^{1/2} \le C\left(\int_\tau^{\tau+1}\big((E^{-1})^2 + (E^0)^2\big)\, dt\right)^{1/2} + O(|\tau|^{-1})
\end{equation}
so it is enough to estimate the integrals  $\int_\tau^{\tau+1} (E^i)^2\, dt$, for $i = \{-1, 0\}$ and $\tau \le t_0$.
We will only discuss the computation for $E^0$, as the computation for $E^{-1}$ is identical.  
We have
$$E^0 = -\dot \xi(t) \inti \,  \phi_x \,  w'(x) w^{p-1} \, dx + 
  \inti  c(x,t) \, (\phi_{xx} - \phi)\, w'(x)  \, dx=E^0_1+E^0_2$$
with 
$$c(x,t) =  \frac{w^{p-1}- \bar w^{p-1}}{\bar w^{p-1}}=\big (\frac w{\bar w} \big )^{p-1} -1.$$
Clearly we have
\begin{equation}\label{eqn-E1}
\begin{split}
\left(\int_\tau^{\tau+1} |E^0_1|^2\, dt \right)^{1/2} &\leq C \,\left(\int_\tau^{\tau+1}|\dot{\xi}|^2\big(\inti\phi_x w'(x)w^{p-1}\, dx\big)^2\, dt\right)^{1/2} \\
&= C\left(\int_\tau^{\tau+1}|\dot{\xi}|^2\, \big(\inti \phi\,  (w'(x)w^{p-1})_x\, dx\big)^2\, dt \right )^{1/2} \\
&\le C\, \|\phi\|_{L^{\infty}(\Lambda_\tau)} \left(\int_\tau^{\tau+1}|\dot{\xi}|^2\, dt \right)^{1/2}\\
&\le \frac{C}{|\tau|}\, \|\phi\|_{L^{\infty}(\Lambda_\tau)}.
\end{split}
\end{equation}
For the second term, we have
\begin{equation}\label{eqn-c1}
\inti  c(x,t) \, (\phi_{xx} - \phi)\, w'(x)  \, dx 
 \leq I(t)^{\frac 12}   \, \left (  \inti  \frac {(\phi_{xx} - \phi)^2}{\bar w^{p-1}} \, dx  \right )^{\frac 12}
\end{equation}
where
$$I(t) := \inti c^2(x,t) \, |w'(x)|^2 \bar w^{p-1}\, dx .$$

\no Recall that $\xi$ given by \eqref{eqn-xit} satisfies $\xi(t) = \frac 12 \log |t| + O(1)$,
as $t \to -\infty$. On $x < \xi(t)$ we have 
$w \leq \bar w \leq 2w$, hence 
$$\frac 12 \leq \frac w{\bar w} \leq 1.$$ 
It follows that 
$$c^2(x,t)=(1- (\frac w{\bar w})^{p-1})^2 \leq C(p)  \, (1-\frac w{\bar w})^2.$$

\no We conclude that
\begin{equation*}
\begin{split}
I_1:&=\int_{-\infty}^{\xi(t)}  c^2(x,t)  |w'(x)|^2 \bar w^{p-1}\, dx  \\
&\leq C  \int_{-\infty}^{\xi(t)} \big (\frac{w(x-2\xi)}{\bar w(x,t)}\big )^2  |w'(x)|^2 \bar w^{p-1}\, dx \leq C \int_{-\infty}^{\xi(t)}  w(x-2\xi)^2 \bar w^{p-1}\, dx\\
&\leq C\, |t|^{-2}  \left  ( \int_{-\infty}^0  e^{2x}  e^{(p-1)x}\, dx + \int_0^{\xi(t)} e^{2x} e^{-(p-1)x}\, dx \right  )\\
&\leq C\, |t|^{-2} \big (C_1 + C_2 |t|^{\frac{3-p}2}) \leq C\, \max\{|t|^{-2}, |t|^{-\frac{1+p}2}\}.
\end{split}
\end{equation*}
On $x > \xi(t)$, using the bound   $c^2 \leq 1$ and $|w'(x)|^2 \leq C\, |t|^{-1}$ we have 
\begin{equation*}
\begin{split}
I_2:&=\int_{\xi(t)}^\infty c^2(x,t)  |w'(x)|^2 \bar w^{p-1} \, dx  \leq C \int_{\xi(t)}^\infty   |w'(x)|^2 \bar w^{p-1}\, dx \\
&\leq C |t|^{-1} \, \int_{\xi(t)}^\infty   \bar w^{p-1}\, dx \leq C\, |t|^{-1}. 
\end{split}
\end{equation*}
Since $p >1$, combining the above gives us the estimate
$$I (t) =I_1+I_2 \leq C\, |t|^{-1}.$$
Using  the last estimate in  \eqref{eqn-c1} yields the bound 

\begin{equation}\label{eqn-E2}
|E^0_2(t)| \leq C \,\frac 1{\sqrt{|t|}} \,  \left (  \inti 
  \frac {(\phi_{xx} - \phi)^2}{\bar w^{p-1}} \, dx  \right )^{\frac 12}.
\end{equation}

\no Combining \eqref{eqn-E1} and \eqref{eqn-E2} gives us the bound
$$\left(\int_\tau^{\tau+1} |E^0(t)|^2\, dt\right)^{1/2} \leq C \,\frac 1{\sqrt{|\tau|}} \,  \left \{ \left (  \iint_{\Lambda_\tau} 
 \big ( \frac {\phi_{xx} - \phi}{\bar w^{p-1}} \big )^2   \bar w^{p-1} \, dx \, dt\right  )^{\frac 12} + \|\phi\|_{L^{\infty}(\Lambda_\tau)}\right  \}$$
and the same bound holds for $E^{-1}(t)$. 
By \eqref{eqn-d1d2} it follows that \eqref{eqn-c1c2} holds.
\end{proof}

\no Using  \eqref{eqn-c1c2}  we can easily   estimate the $L^2(\Lambda_\tau)$ norm of the term 
$$C(\psi,t):= d_1(t)\, z + d_2(t) \, \bar z$$

\no by the $H^2(\Lambda_\tau)$ norm of  function $\psi$,  as 

\begin{equation}\label{eqn-cpsi}
\| C(\psi,t)  \|_{L^2(\Lambda_\tau)}  \leq C \,\frac 1{\sqrt{|\tau|}} \, \big (\|\psi  \|_{H^2(\Lambda_\tau)} + \|\psi\|_{L^{\infty}(\Lambda_{\tau})} \big).
\end{equation}

\medskip

The main result in  this section is the following proposition. 

\begin{prop}\label{prop-l2e} For  given numbers $p >1$, $b >0$ and $0 \leq \mu \leq 1$   
and a given function $\xi:=\frac 12 \log (2b|t|) + h$ on $(-\infty, t_0]$  with
$\| h \|_{1,\sigma, t_0}^{\mu, 1+\mu} \leq 1$,  consider the equation \eqref{eqn-psi}
with coefficient $z$ given by \eqref{eqn-wxt1}. Then, for any $\nu \in [0,1]$,  there is a number  $t_0 <0$,
depending on $ \mu, \nu, b$ and $p$  and such that  for any even function  $f$ on $\R \times (-\infty,t_0]$
with
$\|f \|_{L^2_{t_0}}^\nu  < \infty$,   satisfying the 
orthogonality conditions \eqref{eqn-orth1} and \eqref{eqn-orth2} there exists an ancient  solution $\psi$ of \eqref{eqn-psi} on $-\infty \leq  t \leq t_0$
 also satisfying the orthogonality conditions \eqref{eqn-orth1} and \eqref{eqn-orth2},  and for which
\begin{equation}\label{eqn-basic0}
\sup_{\tau \leq t_0} |\tau|^\nu  \| \psi(\tau) \|_{L^2} + \| \psi \|_{H^2_{t_0}}^\nu \leq C \, \| f \|_{L^2_{t_0}}^\nu.
\end{equation}
The  constant $C$ depends  only on $b, \mu$, $\nu$ and $p$. 

\end{prop}

As we already discussed,   the ancient solution $\psi$ will be constructed as the
limit of the solutions $\psi^s$ to \eqref{eq-psi}, as $s \to -\infty$. 
The existence of the solutions  $\psi^s$ is given by  the next Lemma. 

\begin{lem}\label{lem-es} Under the assumptions of Proposition \ref{prop-l2e}, 
there exists a  number $t_0<0$ depending on $b,  \mu, \nu$ and $p$ and a  solution $\psi^s$ of the initial value problem \eqref{eq-psi}
also satisfying the orthogonality conditions \eqref{eqn-orth1} and \eqref{eqn-orth2}. 
In addition,  
\begin{equation}\label{eqn-basic0s}
\sup_{\tau \in [s,t_0]} |\tau|^\nu  \| \psi^s(\tau) \|_{L^2} + \| \psi^s\|_{H^2_{s,t_0}}^\nu \leq C \, \| f \|_{L^2_{s,t_0}}^\nu
\end{equation}
where $C$ depends only on $b,  \nu, \mu$ and $p$.  
\end{lem} 

\begin{remark} [Dependence  on  function  $\xi$] 
For the remaining of Section \ref{sec-linear}, we  will fix $b >0$, $\mu \in (0,1)$
and function $\xi:=\frac 12 \log (2b|t|) + h$ with   $\| h \|_{1,\sigma, t_0}^{\mu, 1+\mu} \leq 1$ and we will only discuss the dependence
of the various constants on $s$ and $t_0$, while assuming that may also depend on $b$, $\mu$ and $\nu$.

\end{remark}

\subsection{A priori estimates} 
\label{sec-apriori}

We will establish in this section apriori $L^2$ and $H^2$ energy estimates for the solutions $\psi^s$ of \eqref{eq-psi} that are independent on $s$. 
 We begin by proving an energy estimate (independent of $s$) for solutions of the initial value problem
 \begin{equation}
\label{eq-psi222}
\begin{cases}
p z^{p-1}\, \partial_t \psi^s = \psi^s_{xx} - \psi^s + p  z^{p-1}\, \psi^s + z^{p-1}\, g
 \qquad  & \mbox{on} \,\,\,  \R  \times [s,t_0]\cr
\psi^s(\cdot,s) = 0  \qquad & \mbox{on} \,\,\, \R 
\end{cases}
\end{equation}
Energy estimates for solutions of equation \eqref{eq-psi} will easily follow by Lemma \ref{lem-energy1} and estimate \eqref{eqn-cpsi}.

\begin{lem}[Energy $H^2$ and   $L^{\infty}$ estimate for equation \eqref{eqn-psi0}]
\label{lem-energy1}
Let $\psi^s(x,t)$ be a  solution of \eqref{eq-psi222}.   Then,
for any $\nu \in [0,1)$  there exists a uniform in $s$ constant $C$ so that for $|t_0|$ sufficiently large we have 
\begin{equation}
\label{eqn-l224}
\sup_{\tau \in [s,t_0]} |\tau|^{\nu} \| \psi \|_{L^{\infty}(\Lambda_{\tau})} + \| \psi \|_{H^2_{s,t_0}}^\nu   \le C \big (  \| \psi \|^{\nu}_{L^2_{s, t_0}}+ \|g\|^{\nu}_{L^2_{s, t_0}}\big ).
\end{equation} 
\end{lem}

\begin{proof}

To simplify the notation, we will denote $\psi^s$ by $\psi$. In what follows we will perform
various integration by parts in space without worrying about the boundary terms at
infinity. This can be easily  justified by considering approximating  solutions  $\psi^s_R$ of the
Dirichlet problem on expanding cylinders  $Q_R:=[-R,R]\times [s,t_0]$, establish 
the  a priori estimates  on $\psi^s_R$, independent of both, $s$ and $R$, and then pass to the limit  of $\psi_R^s$ as $R \to \infty$ (our solution
$\psi^s$ in Lemma \ref{lem-es} will be constructed  that way). 

 If we  multiply the equation (\ref{eqn-psi0}) by $\psi$ and integrate it over $\mathbb{R}$, we  obtain,
\begin{equation}
\label{eqn-ddt}
\frac{p}{2}\frac{d}{dt}  \inti \psi^2\, \zz^{p-1} dx = \inti  \big (\psi \, \psi_{xx} - \psi^2  +   p  [1+ (p-1) \frac{\zz_t}{z} ] \, \psi^2 \, \zz^{p-1} +  g\,  \psi \zz^{p-1} \big ) \, dx.
\end{equation}
If we integrate by parts the first term  on the right hand side, use the bound $|\zz_t| / z \le C\,  |\dot{\xi}|$ and apply Cauchy-Schwarz, we obtain
\begin{equation*}
\begin{split}
\frac{p}{2}\frac{d}{dt}  \inti \psi^2(\cdot,t)\,  \zz^{p-1} dx &+  \inti ( \psi_x^2 + \psi^2)  \, dx\\
&\leq C\, \left(\inti  (\psi^2 +  g^2)\, \zz^{p-1} \, dx + |\dot{\xi}|\, \inti \psi^2\, z^{p-1}\, dx\right).
\end{split}
\end{equation*}
For any number $\tau \in [s,t_0-1] $, set $\eta(t)=t-\tau$ so that $0 \leq \eta(t) \leq 1$
on $[\tau,\tau +1]$. 
 Then, for any $t \in [\tau,\tau+1]$, we have 
 \begin{equation*} 
\begin{split}
 \frac{d}{dt} \left  (\eta(t) \inti \psi^2(\cdot,t) \, \zz^{p-1} dx \right ) &+ \eta(t)\inti (\psi_x^2 + \psi^2)\, dx \\&\leq  C\, \left (\inti  (\psi^2+g^2)  \, \zz^{p-1} \, dx +  |\dot{\xi}|\inti \psi^2 z^{p-1}\, dx\right ).
 \end{split}
\end{equation*}
Integrating this inequality in time  on $[\tau,\tau+1]$ while   applying the Cauchy-Schwarz inequality to the last term and  using that $\eta(t) \le 1$ and  $\left(\int_{\tau-1}^{\tau}|\dot{\xi}|^2\, dt\right)^{1/2} \le C/|\tau|$, we obtain
\begin{equation}
\begin{split}
\label{eqn-H1}
\inti \psi^2  &(\cdot,\tau+1) \, \zz^{p-1} dx + \iint_{\Lambda_\tau} \eta(t)\, (\psi_x^2 + \psi^2) \, dx\, dt  \\
&\leq  C\, \left (\|\psi\|^2_{L^2(\Lambda_{\tau})} + \|g\|^2_{L^2(\Lambda_{\tau})} + \frac{1}{|\tau|}\, \sup_{t \in [\tau, \tau+1]}
(\inti  \psi^2 z^{p-1}\, dx )^{1/2}\|\psi\|_{L^2(\Lambda_{\tau})}\right ). 
\end{split}
\end{equation}

If we now multiply (\ref{eqn-psi0}) by $\psi_t(x,t)$, integrate by parts over $\mathbb{R}$, use the bound 
$|\zz_t|/z \le C\, |\dot{\xi}|$  and apply Cauchy-Schwartz to the last term,  we obtain 
\begin{equation}\label{eqn-psit000}
\begin{split}
\frac p2\,  \inti \psi_t^2\, \zz^{p-1} dx  &+ \frac{1}{2}\frac{d}{dt}\left (  \inti (\psi_x^2 + \psi^2 - p\zz^{p-1} \psi^2)\, dx\right  ) \\& \le  C\, \left (\inti \big ( \psi^2 + g^2)  \, \zz^{p-1}  \, dx  + |\dot{\xi}|\inti \psi^2 z^{p-1}\, dx \right ).
\end{split}
\end{equation}
Multiplying the last inequality   by the cut off function $\eta(t)$ introduced above,  integrating in time and using (\ref{eqn-H1}), we obtain (similarly as above) the  bound
\begin{equation}
\label{eqn-estH2}
\begin{split}
\iint_{\Lambda_\tau} \eta(t)\,  \psi_t^2\, &z^{p-1} \, dx \,  dt + \inti (\psi_x^2 + \psi^2 - pz^{p-1}\psi^2)(\cdot,\tau+1)
\, dx \nonumber \\
&\le C\, \left (\|\psi\|^2_{L^2(\Lambda_{\tau})} + \|g\|^2_{L^2(\Lambda_{\tau})} + \frac{1}{|\tau|}  
\sup_{t \in [\tau, \tau+1]}
(\inti  \psi^2 z^{p-1}\, dx )^{1/2}
\,\|\psi\|_{L^2(\Lambda_{\tau})} \right ).
\end{split}
\end{equation}

Furthermore,  \eqref{eqn-estH2}, \eqref{eqn-H1},  the Sobolev embedding theorem in one dimension and  the interpolation inequality,  yield  the $L^\infty$ estimate
\begin{equation*}
\begin{split}
\|\psi (\cdot,\tau+1) \|_{L^{\infty}(\R)} &\le C\left(\inti (\psi_x^2 + \psi^2)(\cdot,\tau+1)\, dx\right)^{1/2}  \\
&\le C\left (\|\psi\|_{L^2(\Lambda_{\tau})} + \|g\|_{L^2(\Lambda_{\tau})} + \frac{1}{|\tau|} \|\psi\|_{L^{\infty}(\Lambda_{\tau})}\right )
\end{split}
\end{equation*}
We next multiply the last inequality  by $|\tau+1|^{\nu}$, for any $\nu \ge 0$. Since $ s \leq \tau \le t_0-1$, by choosing $|t_0|$ sufficiently large  we obtain
$$|\tau+1|^{\nu}\|\psi\|_{L^{\infty}(\Lambda_{\tau + 1})} \le C(\|\psi\|^{\nu}_{L^2_{s,t_0}} + \|g\|^{\nu}_{L^2_{s,t_0}}) + \frac{1}{2}\sup_{\tau\in [s,t_0]}|\tau|^{\nu}\|\psi\|_{L^\infty(\Lambda_{\tau})}.$$
\no Since $\tau + 1 \le t_0$ is arbitrary, we obtain
$$\sup_{\tau\in [s,t_0]}\, |\tau|^{\nu} \|\psi\|_{L^{\infty}(\Lambda_{\tau})} \le C(\|\psi\|^{\nu}_{L^2_{s,t_0}} + \|g\|^{\nu}_{L^2_{s,t_0}}).$$
Since $\tau \in [s,t_0-1]$ is arbitrary, by choosing $|t_0|$ sufficiently large,  we conclude
\begin{equation}
\label{eq-Linfty100}
\sup_{\tau\in [s,t_0]} |\tau|^{\nu} \|\psi\|_{L^{\infty}(\Lambda_{\tau})} \le C(\|\psi\|^{\nu}_{L^2_{s,t_0}} + \|g\|^{\nu}_{L^2_{s,t_0}}).
\end{equation}
In addition, integrating \eqref{eqn-psit000} on $[\tau,\tau+1]$ and using the previous estimates yields    the  bound
\begin{equation}
\label{eqn-psit111}
\|\psi_t\|^{\nu}_{L^2_{s,t_0}} 
 \le C\, \big (\|\psi\|^{\nu}_{L^2_{s,t_0}} + \|g\|^{\nu}_{L^2_{s,t_0}} \big ).
\end{equation}
Finally, from    \eqref{eqn-psit111},  (\ref{eq-Linfty100}) and the equation \eqref{eqn-psi0} we obtain 
$$ \| \zz^{-(p-1)} \, (\psi_{xx}-\psi)  \|_{L^2_{s,t_0}}^\nu \leq C \big (  \| \psi\|^{\nu}_{L^2_{s, t_0}}+ \|g\|^{\nu}_{L^2_{s, t_0}}\big ).$$
Combining the above estimates gives us  the bound \eqref{eqn-l224}. 

\end{proof}

We will proceed next to showing the same estimate as above for solutions of \eqref{eq-psi}.

\begin{lem}[Energy $H^2$ and   $L^{\infty}$ estimate for equation \eqref{eq-psi}]
\label{lem-energy2}
Let $\psi^s(x,t)$ be a  the  solution to (\ref{eq-psi}). Then there exists a $t_0 < 0$ so that 
for any $\nu \in [0,1)$ we have 
\begin{equation}\label{eqn-energy2}
\sup_{\tau \in [s,t_0]} |\tau|^{\nu} \| \psi \|_{L^{\infty}(\Lambda_{\tau})} + \| \psi \|_{H^2_{s,t_0}}^\nu   \le C 
\big (  \| \psi \|^{\nu}_{L^2_{s, t_0}}+ \|f\|^{\nu}_{L^2_{s, t_0}} \big ).
\end{equation}
\end{lem}

\begin{proof}
If we   apply the  estimate from Lemma \ref{lem-energy1}, with $g=f + C(\psi,t)$, we obtain

$$\sup_{\tau \in [s,t_0]} |\tau|^{\nu} \| \psi \|_{L^{\infty}(\Lambda_{\tau})} + \| \psi \|_{H^2_{s,t_0}}^\nu   \le C \big (  \| \psi \|^{\nu}_{L^2_{s, t_0}}+ \|C(\psi,t)\|^{\nu}_{L^2_{s, t_0}} +   \|f\|^{\nu}_{L^2_{s, t_0}}\big ).$$

\no On the other hand, it follows from \eqref{eqn-cpsi}, that 

\begin{equation}\label{eqn-cpsi1}
 \|C(\psi,t)\|^{\nu}_{L^2_{s, t_0}} \leq \frac{C}{\sqrt{|t_0|} }\, \| \psi \|_{H^2_{s,t_0}}^\nu
 \end{equation}

\no and the desired estimate follows by choosing $t_0$   so that 
$\frac{C}{\sqrt{|t_0|}} \le 1/2.$ 
\end{proof}

\begin{corollary}\label{cor-energy2}
[Estimation of the error term] Under the assumptions of Lemma \ref{lem-es},  there exist uniform constants $t_0 < 0$ and $C > 0$ so that  for any $\nu \in [0,1)$,  we have 
\begin{equation}\label{eqn-cpsi5}
 \|C(\psi,t)\|^{\nu}_{L^2_{s, t_0}} \leq \frac{C}{\sqrt{|t_0|} }  \big (  \| \psi \|^{\nu}_{L^2_{s, t_0}}+ \|f\|^{\nu}_{L^2_{s, t_0}} \big ).
 \end{equation}
\end{corollary}

\begin{proof}
It  readily follows from combining \eqref{eqn-cpsi1}
and \eqref{eqn-energy2}. 
\end{proof}

We will next establish \eqref{eqn-basic0s} as an a'priori estimate. 


\begin{prop}\label{prop-apriori} 
There exist uniform constants $C< \infty$, $t_0 <0$,   such that if 
  $\psi^s(x,t)$ is a solution of  (\ref{eq-psi}) with $s < t_0/2$,  which also  satisfies the  orthogonality conditions \eqref{eqn-orth1} and \eqref{eqn-orth2}, then 
 \begin{equation}\label{eqn-basic-s}
\sup_{\tau \in [s,t_0]} |\tau|^{\nu} \| \psi^s (\cdot,\tau) \|_{L^2}  \leq C \, \|f \|_{L^2_{s,t_0}}.
\end{equation}
It follows that \eqref{eqn-basic0s} holds. 
\end{prop}

\begin{proof} 
We will first establish the estimate \eqref{eqn-basic-s}. 
We begin by observing that under the assumptions of the Proposition,  it will be sufficient to establish the bound
\begin{equation}\label{eqn-l22}
\sup_{s \leq \tau  \leq t}  \| \psi^s (\cdot,\tau) \|_{L^2}  \leq C  \,
\sup_{s \leq \tau  \leq t}  \|f \|_{L^2(\Lambda_\tau)}, 
\end{equation}
for all $t$ such that $s \le t/2 \le t_0/2$, where we recall that $\Lambda_\tau=\R \times [\tau,\tau+1]$. 
Indeed, if  \eqref{eqn-l22} holds, then for any $t$ satisfying  $ s <  t/2 \leq  t_0/2$,  we have  
\begin{equation*}
\begin{split}
|t|^\nu  \| \psi^s (\cdot,t) \|_{L^2} &\leq C  \,
|t|^\nu \sup_{s \leq \tau  \leq t} \|f \|_{L^2(\Lambda_\tau)} \leq 
 \sup_{s \leq \tau  \leq t} |\tau|^\nu  \|f \|_{L^2(\Lambda_\tau)}\\
& \leq \sup_{s \leq \tau  \leq t_0} |\tau|^\nu \|f \|_{L^2(\Lambda_\tau)} = C \, \|f \|_{L^2_{s,t_0}} 
\end{split}
\end{equation*}
which readily shows that \eqref{eqn-basic-s} holds.

To establish the validity of \eqref{eqn-l22} we  argue by contradiction.  If \eqref{eqn-l22} doesn't hold,  then there
exist decreasing sequences  $ \bar t_k \to -\infty$ and $s_k < \bar t_k/2$,
$s_k \to -\infty$  and solutions   
$\psi_k$ of equation
\begin{equation}\label{eqn-psik}
p\zz^{p-1} \, \pp_t \psi_k = \pp_{xx} \psi_k - \psi_k  + p \zz^{p-1} \, \psi_k + \zz^{p-1} \, [ f_k - C(\psi_k,t)] 
\end{equation}
defined on $\R \times [s_k ,\bar t_k]$, so that  
\begin{equation}\label{eqn-l2k}
\sup_{s_k \leq \tau  \leq \bar t_k}  \| \psi_k (\cdot,\tau) \|_{L^2}   \geq k  \, 
\sup_{s_k \leq \tau  \leq \bar t_k}  \|f_k\|_{L^2(\Lambda_\tau)}.  \end{equation}
We may assume, without loss of generality, that 
\begin{equation}\label{eqn-sup1}
\sup_{s_k \leq \tau  \leq \bar t_k}  \inti \psi^2_k (x,\tau) \, \zz^{p-1}\, dx=2
\end{equation}
otherwise we would perform the rescaling of the solutions  for (\ref{eqn-sup1}) to hold.
Then, by \eqref{eqn-l2k}, we have
\begin{equation}\label{eqn-l2k2}
\| f_k \|_{L^2_{s_k,\bar t_k}} \leq \frac Ck. 
\end{equation}
Because of \eqref{eqn-sup1}, we can pick $t_k \in  [s_k ,\bar t_k]$ such that 
\begin{equation}\label{eqn-sup12}
\frac 32 \leq  \inti \psi^2_k (x,t_k) \, \zz^{p-1} \, dx \leq  2.
\end{equation}
Also, passing to a subsequence if necessary, we may assume that $t_k$ is decreasing.

\begin{claim}\label{claim-sigma}  We have 
$$\liminf_{k \to \infty} \, (t_k - s_k)= +\infty.$$
\end{claim}

\begin{proof}[Proof of Claim] 
We will apply  \eqref{eqn-l224} with $\psi=\psi_k$,
$g=g_k:=f_k-C(\psi_k,t)$ and for $\nu=0$.   To estimate  the right hand side of \eqref{eqn-l224},
we use   \eqref{eqn-cpsi5}, \eqref{eqn-l2k2} and \eqref{eqn-sup1}  to obtain for all $s_k \leq t \leq \bar t_k$  the bound 
$$\| g_k ( \cdot, t) \|_{L^2} \leq   \frac{C}{\sqrt{|\bar t_k|} }  
\big (  \| \psi \|_{L^2_{s, \bar t_k}}+ \frac 1k \big ) + 
\frac Ck \leq C \, \big ( \frac 1{\sqrt{|\bar t_k|}} + \frac 1k \big ).$$
Hence, for all $\tau \in [s_k, \bar t_k]$ we have 
\begin{equation}\label{eqn-ode2}
\int_{s_k}^{\tau} \inti g^2_k\,   \zz^{p-1}  \, dx  \, dt \leq  (\tau - s_k)  \, \| g_k \|_{L^2_{\bar t_k}}^2  \leq  C \, \big ( \frac 1{\sqrt{|\bar t_k|}} + \frac 1k \big )^2 (\tau-s_k)
\end{equation}
Set $\alpha(\tau) = \int_{s_k}^{\tau}  \inti  \psi^2 \, \zz^{p-1}\, dx dt$. It follows from  \eqref{eqn-l224} and the above discussion that $\alpha(\tau)$ satisfies the differential inequality
$$\alpha'(\tau) \leq C \, \alpha(\tau) + \mu_k \, (\tau-s_k)$$
with $
\mu_k = C \,( \frac 1{\sqrt{|\bar t_k|}} + \frac 1k )^2$
and  $\alpha(s_k)=0$. 
Solving this differential inequality gives
$$\alpha(\tau) \leq \frac{\mu_k}{C^2} \big ( e^{C(\tau-s_k)} - [1+C (\tau-s_k)]\big )
\leq \frac{\mu_k}{C^2} \,  e^{C(\tau-s_k)}$$ which combined with \eqref{eqn-l224} and  \eqref{eqn-ode2} gives   the bound
$$\inti \psi^2 (\cdot,\tau) \, \zz^{p-1} dx \leq  C_1\, \mu_k \,  e^{C(\tau-s_k)},$$
for all $\tau\in [s_k,\bar{t}_k]$, where $C_1$ is a different, but still uniform constant. 
Hence  
$$\inti \psi^2 (\cdot,\tau) \, \zz^{p-1} dx \leq  1$$
as long as 
$e^{C(\tau-s_k)} < 1/(C_1\mu_k).$
Since $\inti \psi^2 (\cdot,t_k)  \, \zz^{p-1} dx \geq 3/2$, we must have 
$$e^{C(t_k-s_k)}  \geq 1/(C_1\mu_k).$$
Since $\lim_{k \to \infty} \mu_k =0$, this 
 readily implies  the claim. 
\end{proof}

\ssk

Set
$$\bar \psi_k(x,t) = \psi_k(x,t+t_k) \quad \mbox{and} \quad \bar f_k(x,t) = f_k(x,t+t_k)$$
and observe that each $\bar \psi_k$ satisfies the equation
\begin{equation}\label{eqn-psik2}
p \zz^{p-1}_k \, \pp_t \bar \psi_k = \pp_{xx} \bar  \psi_k - \bar \psi_k  + p \zz^{p-1}_k \, 
\bar \psi_k + \zz^{p-1}_k \, [\bar f_k - C_k(\bar \psi_k,t))]
\end{equation}
on $\R \times [\bar s_k , 0]$, with $\bar s_k:= s_k - t_k$  and
$$z_k(x,t) := z(x,t+t_k), \qquad \bar z_k(x,t) = \bar z(x,t+t_k)$$
and
$$C_k(\bar \psi_k,t) = d_1 (\bar \psi_k,t+t_k)\, z_k(x, t) + d_2(\bar \psi_k,t+t_k)\, \bar z_k(x, t)$$
where  $d_i(\bar \psi_k,t)$ are defined in terms of $\psi_k$ as before. Notice that because
of the previous claim, $\bar s_k \leq -\sigma$, for a uniform constant $\sigma$. 
It follows from   \eqref{eqn-energy2}  (with $\nu=0$) and \eqref{eqn-sup1} that $\bar \psi_k$ satisfy the bound 
\begin{equation}\label{eqn-energy3}
\| \bar \psi_k \|_{L^\infty(\R \times [\bar s_k,0])} + \| \bar \psi_k \|_{H^2_{\bar s_k,0}}    \leq C
\end{equation}
for a uniform in $k$  constant $C$. 

The inequality
\eqref{eqn-sup12} says   that 
\begin{equation}\label{eqn-sup11}
\frac 32 \leq  \inti \bar \psi^2_k (x,0) \, \zz^{p-1}_k  \, dx \leq  2.
\end{equation}
If we integrate \eqref{eqn-ddt} in time on $[t_k-\delta,t_k]$ and use  \eqref{eqn-sup11},
we conclude, after a straightforward calculation, the bounds   
\begin{equation}\label{eqn-sup11111}
1 \leq \inf_{\tau \in [-\delta,0]}   \inti \bar \psi^2_k (x,\tau) \, \zz^{p-1}_k \, dx \leq  2.
\end{equation}
for a uniform in $k$ small constant $\delta >0$. 

\begin{claim} There exists a universal large constant $M >0$ for which 
\begin{equation}\label{eqn-claim3}
 \sup_{\tau \in [-\delta,0]} \int_{-\xi(\tau+t_k)+M}^{\xi(\tau+t_k)-M} \bar \psi_k^2(x,\tau ) \, \zz_k^{p-1}\, dx < \frac 12.
 \end{equation}

\end{claim}

\begin{proof}[Proof of Claim] We recall that by \eqref{eqn-xit}, $\xi(\tau+t_k) = \frac 12 \log (2b|\tau+t_k|)  + O(1).$
By symmetry ($\bar \psi_k$ is an even function)  we only need to show that
$$  \sup_{\tau \in [-\delta,0]} \int_0^{\xi(\tau+t_k) -M} \bar \psi_k^2(x,\tau) \, \zz_k^{p-1}\, dx < \frac 14.$$
Also, since for $x>0$ and $\tau \in [-\delta,0]$, 
$$\zz_k(x,\tau)= w(x-\xi(\tau+t_k)) + w(x+\xi(\tau+t_k)) \leq 2\, w(x-\xi(\tau+t_k)) $$
it will be enough  to establish the inequality  
$$ \sup_{\tau \in [-\delta,0]}  \int_0^{\xi(\tau+t_k) -M} \bar \psi_k^2(x,\tau) \, w^{p-1}(x-\xi(\tau+t_k)) \, dx < \frac 18.$$
Using the $L^\infty$ bound in \eqref{eqn-energy3}, we 
 conclude that for every $\tau \in [-\delta,0]$ we have
$$ \int_0^{\xi(\tau +t_k) -M} \bar \psi_k^2(x,\tau) \, w^{p-1}(x-\xi(\tau+t_k)) \, dx \leq  C \,
\int_0^{\xi(\tau+t_k) -M}  w^{p-1}(x-\xi(\tau+t_k)) \, dx$$
for a uniform constant $C$. 
Finally, we have
$$\int_0^{\xi(\tau+t_k) -M}  w^{p-1}(x-\xi(\tau+t_k)) \, dx = \int_{-\xi(\tau+t_k)}^{-M}  w^{p-1}(x) \, dx$$
where $w$ is given by \eqref{eqn-w1}. It follows that there exists a uniform constant $M$ such that
$$ C \,
\int_0^{\xi(\tau+t_k) -M}  w^{p-1}(x-\xi(\tau+t_k)) \, dx = C \, \int_{-\xi(\tau+t_k)}^{-M}  w^{p-1}(x) \, dx <\frac 18$$
for all $\tau \in [-\delta,0]$ finishing the proof of the claim.
\end{proof}

We will now conclude the proof of the Proposition. By \eqref{eqn-sup11111}, 
\eqref{eqn-claim3} and the symmetry of $\bar \psi_k$, we have 
\begin{equation}\label{eqn-bb}
\inf_{\tau \in [-\delta,0]}  \int_{-\infty}^{\xi(\tau+t_k)+M}  \bar \psi_k^2(x,\tau) \, \zz_k^{p-1}\, dx \geq  \frac 14. 
\end{equation}

\no  We wish to pass to the limit along a sub-sequence  
$k_l \to \infty$. However, in order that we see something non-trivial at the limit,
we will need to perform a new change of variables defining  
$$\phi_k(x,t) := \bar \psi_k(x-\xi(t+t_k),t), \qquad t \leq 0.$$
It follows that each $\phi_k$ satisfies the equation 
\begin{equation}\label{eqn-phik3}
\begin{split}
p w^{p-1}_k \, \pp_t \phi_k &= \pp_{xx}   \phi_k- \phi_k  + p  w^{p-1}_k \,  \phi_k  - \dot \xi(t+t_k) \, \pp_x   
\phi_k + w^{p-1}_k \, g_k
\end{split}
\end{equation}
on $-\bar s_k <  t \leq 0$ with $g_k(x,t) := \bar f_k(x-\xi(t+t_k),t) - C_k(\phi_k,t)$ and 
$$w_k(x,t) :=  \zz_k (x-\xi(t+t_k),t)= w(x) + w(x-2\xi(t+t_k)).$$
Moreover, \eqref{eqn-l2k2} and \eqref{eqn-cpsi5} imply the bounds
\begin{equation}\label{eqn-l2k5}
 \|g_k\|_{L^2_{s, t_0}}  \leq C \, \big ( \frac 1k + \frac 1{\sqrt{|t_k|} } \big )
\end{equation}
and by \eqref{eqn-energy3} and the inequality $w_k \geq w$, 
\begin{equation}\label{eqn-energy31}
\|  \phi_k \|_{L^\infty(\R \times [\bar s_k,0])} + \| \phi_k \|_{H^2_{\bar s_k,0}}    \leq C.
\end{equation}

\no In addition, \eqref{eqn-bb} implies the following uniform bound
\begin{equation}\label{eqn-bb2}
\inf_{\tau \in [-\delta,0]} \int_{-\infty}^M    \phi_k^2(x,\tau) \, w_k^{p-1}\, dx \geq  \frac 14. 
\end{equation}

\msk

Set $Q=(-\infty,x_0] \times [\tau_0,0] $, where $x_0 >0$ is an arbitrary number  and $\tau_0$ is any number 
such that  $\bar s_k < \tau_0 < 0$, for all $k$ (recall that $\bar s_k \leq - \sigma$, $\forall k$
by Claim \ref{claim-sigma}).
It follows from  the energy bound \eqref{eqn-energy31} that  passing to a subsequence,  still denoted by $\phi_k$,  we have
$\phi_k \to \phi $ in  $L^2_w(Q)$ and $\phi_k \to \phi $ weakly in
$H^1_w(Q)$.  Passing to the limit in \eqref{eqn-phik3} while using \eqref{eqn-l2k5} and
the bound $$ \int_{\tau_0}^0 \dot \xi^2(t+t_k) \, dt = O(\frac 1{|t_k|^2})$$ we conclude that $\phi$ is a weak solution of
\begin{equation}\label{eqn-phi22}
p w^{p-1} \, \pp_t \phi = \pp_{xx}   \phi - \phi  + p  w^{p-1} \, 
\phi
\end{equation}
on $\R \times [\tau_0,0)$. Standard regularity theory  shows that $\phi$ is actually
a smooth solution.  In addition,   $\phi$ satisfies the orthogonality conditions 
\begin{equation}\label{eqn-orth11p}
\int_{-\infty}^\infty \phi (x,t) \, w'(x) \, w^{p-1} (x) \, dx=0, \qquad a.e. \quad t < t_0
\end{equation}
and
\begin{equation}\label{eqn-orth21p}
\int_{-\infty}^\infty \phi  (x,t) \, w(x) \, w^{p-1}(x) \, dx=0, \qquad a.e. \quad t < t_0.
\end{equation}
Moreover,   from \eqref{eqn-energy3}
 we have the following uniform estimate
\begin{equation}\label{eqn-sup4}
\sup_{\tau  \leq 0}  \int^{\tau+1}_\tau \inti  \phi^2(x,t)  \,  w^{p-1} \, dx\, dt \leq 2.
\end{equation}
Also, passing to the limit in \eqref{eqn-bb2} we conclude that
$$
\int_{-\delta}^0  \int_{-\infty}^M   \phi^2(x,t)  \, w^{p-1}\, dx\, dt  \geq  \frac \delta4 >0 
$$
which shows that our limit $\phi$ is non-trivial. 
From Claim \ref{claim-sigma} we have  $\liminf_{k \to \infty} \bar s_k = -\infty$. Hence,  we may assume,
passing to a subsequence,  that $\bar s_k \to  -\infty$.
It follows, that the limit $\phi$ is an ancient solution of equation \eqref{eqn-phi22}, i.e
defined on $\R \times (-\infty,0]$ which satisfies  
the orthogonality conditions \eqref{eqn-orth11p} and \eqref{eqn-orth21p}. 
 
Set $\alpha(t) = \frac 12 \, \| \phi (\cdot, t) \|_{L^2(w^{p-1}\, dx)}$ and observe that since $\phi$ is orthogonal 
 to the two eigenfunctions of the operator $L_0$ (defined in  \eqref{eqn-oper})  corresponding to its only  two nonnegative eigenvalues $\lambda_{-1}$ 
 and $\lambda_0$, we have 
$$ \alpha'(t)  \leq -  \lambda  \, \alpha(t), \qquad t \leq 0$$
for some $\lambda  >0$, 
implying that 
$\alpha(t) \geq \alpha(0) \, e^{ \lambda  |t|}$
which contradicts  \eqref{eqn-sup4}.  This finishes the  proof of the Proposition. \end{proof}


\subsection{The proofs of Lemma \ref{lem-es} and Proposition \ref{prop-l2e}}

Based on the a priori estimates of  the previous section, we will give now the proofs of Lemma \ref{lem-es}
and Proposition \ref{prop-l2e}.

\begin{proof}[Proof of Lemma \ref{lem-es}]
It will be sufficient  to establish the existence of a solution $\psi^s$ to \eqref{eq-psi}. Indeed, given the existence of solution $\psi^s$, by the fact that the forcing term $f$ satisfies orthogonality conditions \eqref{eqn-orth1} and \eqref{eqn-orth2}  we already know that $\psi^s(\cdot,t)$ will continue to satisfy the  conditions \eqref{eqn-orth1} and \eqref{eqn-orth2} for $t \ge s$. Then the estimate \eqref{eqn-basic0s} follows by Proposition \ref{prop-apriori}.

The strategy for establishing the existence of $\psi^s$ is as follows. Fix an  $s < t_0-1$. We  first  establish the existence of a solution $\psi^s$ to  the initial value problem 
\eqref{eq-psi222}
on $\mathbb{R}\times [s,s+\tau_0]$, for a given function $g$ with $\|g\|_{L^2_{t_0}}^{\nu} < \infty$, where $\tau_0$ is a uniform constant, independent of $s$, to be chosen below.  Then we solve the non-local problem
\eqref{eq-psi} on  $\R  \times [s,s+\tau_0]$. At the end we show how to extend such a solution in time up to 
$t_0$, to obtain a solution of \eqref{eq-psi}. 

We first claim that given an $s < t_0-1$ and a  function $g$ with $\|g\|_{L^2_{t_0}}^{\nu} < \infty$, there exists a 
solution $\psi^s$ of \eqref{eq-psi222}, on $\mathbb{R}\times [s,s+\tau_0]$, for some $\tau_0$ to be chosen below.
The solution  $\psi^s$ will be constructed as the limit, as $R \to \infty$,  of solutions $\psi_R^s $ to the Dirichlet  problems
\begin{equation}
\label{eq-psi-sR}
\begin{cases}
p \, z^{p-1}\, \partial_t \psi_R^s  = (\psi_R^s)_{xx} - \psi_R^s + p \, z^{p-1}\, \psi_R^s
 + z^{p-1}\, g \qquad  & \mbox{on}\,\,\, Q_{R,s} \cr
\psi_R^s(\cdot,s) = 0  \qquad  & \mbox{on}\,\,\, \partial_p Q_{R,s} 
 \end{cases}
\end{equation}
on    $Q_{R,s}:= [-R,R] \times [s,s+\tau_0]$. Since our  weight $z^{p-1}$ is bounded from above and below away from zero  on  $Q_{R,s}$,  by  standard parabolic theory there exists a solution $\psi_R^s$  to the same Dirichlet problem on $\hat Q_{R,s} := [-R,R] 
\times [s,s+\tau_R]$, for some $\tau_R > 0$  which will be taken to satisfy   $\tau_R \leq 1$.   Similarly 
 as in the proof of Lemma \ref{lem-energy1},  $\psi:=\psi_R^s$ satisfies the following estimate
 \begin{equation}
 \label{eq-interm}
\frac{d}{dt}\int_{-R}^R \psi^2 z^{p-1}\, dx + \int_{-R}^R \big( \psi^2 + \psi_x^2 \big )\, dx \le C_1 \, \left
( \int_{-R}^R g^2\,  z^{p-1}\, dx + (\dot{\xi} + 1)\, \int_{-R}^R \psi^2 z^{p-1}\, dx \right)
\end{equation}
for a universal  constant $C_1$. Before we integrate  \eqref{eq-interm} in time,  we observe that 
 \begin{equation}
 \label{eqn-last12}
\int_s^{s+\tau_R} \dot{\xi}  \int_{-R}^R \psi^2 z^{p-1}\, dx\, dt  \le
\left(\int_s^{s+\tau_R} \dot{\xi}^2\, dt \right)^{1/2}  \sup_{[s,s+\tau_R]} \int_{-R}^R \psi^2 z^{p-1}\, dx 
\leq  \epsilon  \sup_{[s,s+\tau_R]} \int_{-R}^R \psi^2 z^{p-1}\, dx
\ee
if $s$ is chosen sufficiently close to $-\infty$. The last inequality follows from the fact that $\dot \xi(t) = \frac 1{2|t|} + \dot h(t)$
and $\| h \|_{1,\sigma,t_0}^{\mu, 1+\mu} \leq 1$ by assumption. 

Set $\tau_0 := \min\{ 1/ (3C_1), 1\}$,  with $C_1$ being the  constant from (\ref{eq-interm})
and take  $\tau_R \leq \tau_0$. Integrating  in time (\ref{eq-interm}),  while using the Dirichlet boundary condition in 
\eqref{eq-psi-sR},  the Cauchy-Schwartz inequality and \eqref{eqn-last12}  with $\epsilon$ chosen sufficiently small, 
we obtain 
\be
\begin{split}
\frac 23\, \sup_{[s,s+\tau_R]} \int_{-R}^R \psi^2 z^{p-1} dx  &+ \iint_{\hat Q_{R,s}} \big ( \psi_x^2 + \frac 12 \psi^2 \big ) dx\, dt\\ &\le C \iint_{\hat Q_{R,s}}  g^2 z^{p-1} dx dt+ C_1  \iint_{\hat Q_{R,s}} \psi^2 z^{p-1} dx \, dt \nonumber \\
&\le C\, \iint_{\hat Q_{R,s}} g^2 z^{p-1}\, dx\, dt +
\frac 13 \sup_{[s,s+\tau_R]} \int_{-R}^R \psi^2\, dx\, dt \nonumber 
\end{split}
\ee
since $C_1 \, |\tau_R| \leq C_1 \, \tau_0 \leq 1/3$ by our choice of $C_1$. 
It follows that $\psi:=\psi_R^s$ satisfies 
\begin{equation}
\label{eq-tildepsi}
\sup_{[s,s+\tau)} \int_{-R}^R \psi^2 z^{p-1}  \, dx + \iint_{Q_{R,s}}  \big(\psi_x^2 + \frac 12 \psi^2 \big)\, dx \, dt 
\le C_0 \iint_{Q_{R,s}}    g^2 z^{p-1}\, dx\, dt
\end{equation}
for a uniform constant $C_0$. Similarly to deriving the energy estimate in Lemma \ref{lem-energy1}, using (\ref{eq-tildepsi}) and the fact that $(\psi_R^s)_x(\cdot,s) = 0$, we find that  $\psi:=\psi_R^s$ also satisfies 
\begin{equation}
\label{eq-psi-H2}
\iint_{\hat Q_{R,s}}  \psi_t^2 z^{p-1}\, dx  dt + \frac 12 \sup_{[s,s+\tau_R]} \int_{-R}^R \big( \psi^2 + \psi_x^2 \big)\, dx  \le C_0 \, \iint_{\hat Q_{R,s}}   g^2  \,   z^{p-1}\, dx dt
\end{equation}
where $C_0$ is a constant, possibly larger than the constant in \eqref{eq-tildepsi},  but still  independent of $R, s$. 
The right hand sides in both inequalities  (\ref{eq-tildepsi}) and (\ref{eq-psi-H2}) are bounded by a  constant  that is independent of $R$, namely
$C_0\, \int_s^{s+1} \int_{-\infty}^{\infty} g^2 z^{p-1}\, dx dt$.
Hence, by   standard linear parabolic  theory  the  solution $\psi_R^s$ will  exist at least for $s \le t \le \tau_0$, namely on $Q_{R,s}$. 
Take a sequence $R_j \to +\infty$ and set $\Lambda_{s,\tau_0} := \mathbb{R}\times [s,s+\tau_0]$. Since the equation in  \eqref{eq-psi222} is non-degenerate on any compact subset $K$  of
$\Lambda_{s,\tau_0}$, the uniform estimates \eqref{eq-tildepsi}-\eqref{eq-psi-H2} and  standard arguments imply
that a subsequence of solutions $\psi_{R_j}^s$ converges in $C^\infty(K)$ 
to  a smooth solution $\psi^s$  of \eqref{eq-psi222}. The limiting smooth solution $\psi^s$ is defined on 
$\Lambda_{s,\tau_0}$.

The next step is to show that we can solve a nonlocal problem (\ref{eq-psi}) on $\Lambda_{s,\tau_0}$. We will do that via contraction map arguments. Define a set
$$X^s  := \{\, \psi \, |\, \|\psi \|_{H^2(\Lambda_s)} < \infty\, \}.$$
We  consider  the  operator $A^s: X^s  \to X^s $ given  by 
$$A^s(\psi) := T^s( f - C(\psi,t)),$$ where $T^s(g)$ denotes  a   solution to \eqref{eq-psi222}
as constructed above. We will  show that the map $A^s$ defines a  contraction map and we will  apply the fixed point theorem to it. 

To this end, set  $c := C_0\, \|f\|_{L^2(\Lambda_s)}$ and  $X_c^s := \{\psi  \in X^s \, | \, \, \|\psi \|_{H^2(\Lambda_s)}  < 2c \}$, where the constant $C_0$ is taken from \eqref{eq-tildepsi}-(\ref{eq-psi-H2}). We claim that $A^s(X_c^s) \subset X_c^s$. To show this claim, let $\psi  \in X_c^s$. The estimates (\ref{eq-tildepsi}), (\ref{eq-psi-H2}), the estimate 
\eqref{eqn-cpsi} for $C(\psi,t)$ 
 and the Sobolev embedding yield to
\begin{equation}
\label{eq-preserve}
\begin{split}
\|A^s(\psi)\|_{H^2(\Lambda_s)} &= \|T^s(f - C(\psi,t))\|_{H^2(\Lambda_s)} \le  C_0\, \|f - C(\psi,t)\|_{L^2(\Lambda_s)} \\
& \le C_0\, \big( \|f\|_{L^2(\Lambda_s)} + \|C(\psi,t)\|_{L^2(\Lambda_s)} \big) = c + C_0\, \|C(\psi,t)\|_{L^2(\Lambda_s)}
\\ & \le c + \frac{C}{\sqrt{|s|}}\, \|\psi\|_{L^2(\Lambda_s)} < 2c,
\end{split}
\end{equation}
if  $|s|$ sufficiently large (which holds if   $t_0$ is chosen sufficiently close to $-\infty$). Next we show that $A^s$ defines a contraction map. Indeed, since $C(\psi,t)$ is  linear in $\psi$, we have
\begin{equation}
\label{eq-contr1000}
\begin{split}
\|A^s(\psi_1) &- A^s(\psi_2) \|_{H^2(\Lambda_s)} = \| T^s \big (C(\psi_1,t) - C(\psi_2,t) \big )\|_{H^2(\Lambda_s)} \\
& \le
C_0\, \|C(\psi_1,t) - C(\psi_2,t)\|_{L^2(\Lambda_s)} = C_0\, \|C(\psi_1 - \psi_2, t)\|_{L^2(\Lambda_s)} 
\\ &\le
\frac{C}{\sqrt{|s|}}\, \| \psi_1 - \psi_2 \|_{H^2(\Lambda_s)}
\le \frac 12 \, \| \psi_1 - \psi_2 \|_{H^2(\Lambda_s)}. 
\end{split}
\end{equation}
By (\ref{eq-preserve}) and (\ref{eq-contr1000}), the fixed point Theorem implies that there exists a $\psi^s \in X^s$ so that $A^s(\psi^s) = \psi^s$, meaning that  the equation (\ref{eq-psi}) has a solution $\psi^s$, defined
on $\Lambda_{s,\tau_0}$. 

We claim that $\psi^s(\cdot,t)$  can be extended  to a solution on $\mathbb{R} \times [s,t_0]$, still satisfying our orthogonality conditions and apriori estimates. To this end, assume that our solution $\psi^s(\cdot,t)$ exists for $s \le t < T$, where $T < t_0$ is the maximal time of   existence. Since $\psi^s(\cdot,t)$ satisfies  orthogonality conditions \eqref{eqn-orth1} and \eqref{eqn-orth2} for $t \in [s,T]$, by Proposition \ref{prop-apriori} ,
\begin{equation}
\label{eq-11}
\sup_{t\in [s,T)} \, |\tau|^{\nu}\, \| \psi^s(\cdot,t) \|_{L^2(\mathbb{R})} \le C\, \| f \|_{L^2_{s,T}}
\end{equation}
and
\begin{equation}
\label{eq-12}
\sup_{t\in [s,T)}\, |\tau|^{\nu} \|\psi^s(\cdot,t)\|_{L^{\infty}(\mathbb{R})} + \|\psi\|_{H^2_{s,T}} \le C\, \|f\|_{L^2_{s,T}}
\end{equation}
where $C$ is a uniform constant.
Since $\| f \|_{L^2_{s,T}} \le \|f\|_{L^2_{s,t_0}} \le C$, it follows  that $\psi^s$ can be extended past time $T$, unless $T = t_0$. Moreover, (\ref{eqn-basic0s}) is  satisfied as well and $\psi^s$ also satisfies the orthogonality conditions.
\end{proof}

Having Lemma \ref{lem-es} we are able to conclude the proof of Proposition \ref{prop-l2e}.

\begin{proof}[Proof of Proposition \ref{prop-l2e}]
Having Lemma \ref{lem-es} we are able to conclude the proof of the Proposition.  Take a sequence of $s_j \to -\infty$. By Lemma \ref{lem-es}, for every $s_j$ there is a solution $\psi^{s_j}$ to \eqref{eq-psi} such that $\psi^{s_j}(\cdot,s_j) = 0$ and it satisfies the  uniform estimate \eqref{eqn-basic0s}, independent of $s_j$. Moreover, our equation \eqref{eq-psi} is non degenerate on every compact subset $K \subset \mathbb{R}\times (-\infty,t_0)$. Therefore on  $K$  we can apply standard parabolic theory  to get higher order derivative estimates for our sequence of solutions $\psi^{s_j}$, which are independent of $s_j$  but may depend on $K$. Let $j\to \infty$. By the Arzela-Ascoli theorem and a standard diagonalization argument  we conclude that a subsequence $\{\psi^{s_j}\}$ converges, as $j\to\infty$,  to a smooth function $\psi$  defined on $\mathbb{R}\times (-\infty, t_0)$. Moreover, $\psi$ satisfies orthogonality conditions \eqref{eqn-orth1} and \eqref{eqn-orth2} and the estimate \eqref{eqn-basic0}. The latter follows from taking the limit as $j\to\infty$ in \eqref{eq-11} and \eqref{eq-12}, both satisfied by $\psi^{s_j}$, and the fact that the constants on the right hand side are independent of $j$.
\end{proof}


\subsection{$W^{2,\sigma}$ estimates}

We will next derive weighted $W^{2,\sigma}$ estimates for the linear equation \eqref{eqn-psi}. We recall the 
$W^{2,\sigma}$ norm is given by Definitions \ref{defn-w2p1} and \ref{defn-w2p2}. We have  the following global  estimate.
\begin{prop}\label{prop-w2p}  Let  $\psi$ be a  solution of \eqref{eqn-psi} as in  Proposition 
\ref{prop-l2e}. If   $\|f\|_{\sigma,t_0}^\nu < \infty$,  for some $\sigma >2$ and $\nu \in [0,1]$,  
then
\begin{equation}\label{eqn-w2p}
\|\psi \|_{2,\sigma,t_0}^\nu \leq C\, \big (  \|f\|_{L^2_{t_0}}^\nu + \|f\|_{\sigma,t_0}^\nu \big ).
\end{equation}

\end{prop}

The proof of Proposition \ref{prop-w2p} will follow from a similar a'priori  estimate for solutions
of \eqref{eqn-psi0}. 

\begin{lem}\label{lem-w2p}  Let  $\psi$ be an even solution of  equation \eqref{eqn-psi0}
with $g$ a given even function that  satisfies $\|g\|_{\sigma,t_0}^\nu  + \| g \|_{L^2_{t_0}}^\nu< \infty$,
for some $\sigma >2$ and $\nu \in [0,1)$. 
Then, we have 
\be\label{eqn-lemw2p}
\|\psi \|_{2,\sigma,t_0}^\nu \leq C\, \big (  \|g\|_{L^2_{t_0}}^\nu + \|g\|_{\sigma,t_0}^\nu \big ).
\ee

\end{lem}

\medskip

Before we give the proof of  Lemma \ref{lem-w2p}, we will  prove Proposition \ref{prop-w2p} using  Lemma \ref{lem-w2p}. 
\begin{proof}[Proof of Proposition \ref{prop-w2p}]  Assume  that $\psi$ is a solution
of equation \eqref{eqn-psi}, as in the statement of the Proposition. It follows from   Lemma 
\ref{lem-w2p} that 
$$\|\psi \|_{2,\sigma,t_0}^\nu \leq C\, \big (  \|f\|_{L^2_{t_0}}^\nu + \|f\|_{\sigma,t_0}^\nu 
+  \|C(\psi,t)\|_{L^2_{t_0}}^\nu + \|C(\psi,t)\|_{\sigma,t_0}^\nu  \big ).$$
It follows from \eqref{eqn-cpsi1} and the estimate in Proposition \ref{prop-l2e}  that 
\begin{equation}\label{eqn-cp1}
\|C(\psi,t)\|_{L^2_{t_0}}^\nu  \leq \frac {C}{\sqrt{|t_0|}} \, \|f\|_{L^2_{t_0}}^\nu.
\end{equation}
In addition,  it can be shown, similarly as in the proof of \eqref{eqn-cpsi1},  that
\begin{equation}\label{eqn-cp2}
\|C(\psi,t)\|_{\sigma,t_0}^\nu  \leq \frac {C}{\sqrt{|t_0|}} \, \|\psi\|_{2,\sigma,t_0}^\nu.
\end{equation}
Combining the last three estimates, readily yields  the estimate of the Proposition, provided 
that $|t_0|$ is chosen sufficiently large. 
\end{proof}

Before we proceed with the proof of Lemma \ref{lem-w2p}, let us  summarize the estimate we have  for $C(\psi,t)$, using Propositions \ref{prop-l2e} and \ref{prop-w2p}.
\begin{corollary}\label{cor-cp3} Under the assumptions of  Proposition \ref{prop-w2p} we have
\begin{equation}\label{eqn-cp3}
\|C(\psi,t)\|_{*,\sigma,t_0}^\nu \leq \frac {C}{\sqrt{|\tau_0|}} \, \big ( \|\psi\|_{H^2_{t_0}}^\nu + \|\psi\|_{2,\sigma,t_0}^\nu \big ).
\end{equation}
It follows that 
\begin{equation}\label{eqn-cp0}
\|C(\psi,t)\|_{*,\sigma,t_0}^\nu \leq 
\frac {C}{\sqrt{|\tau_0|}} \, \|f\|_{*,\sigma,t_0}^\nu.
\end{equation}
for a universal constant $C$. 
\end{corollary}

\begin{proof} The estimate \eqref{eqn-cp3} readily follows by combining \eqref{eqn-cpsi1}
and \eqref{eqn-cp2}.
The estimate \eqref{eqn-cp0} follows from  \eqref{eqn-cp3}  and the bounds in Propositions \ref{prop-l2e} and \ref{prop-w2p}. 
\end{proof}

\medskip
We will now proceed  to the proof of Lemma \ref{lem-w2p}. 

\begin{proof}[Proof of Lemma \ref{lem-w2p}]
We first observe that since both $\zz$ and $\psi$ are even functions in $x$, we will only
need to establish the Lemma on $-\infty < x \leq 0$. We first perform a translation
in space, setting
$$\phi(x,t) = \psi (x-\xi_0(t),t), \qquad -\infty < x < \xi_0(t),$$
where $\xi_0(t) = \frac{1}{2}\log (2b|t|)$. 
It follows that $\phi$ satisfies the equation 
\begin{equation}\label{eqn-phi0}
p \bar \zz^{p-1} \, \pp_t \phi = \pp_{xx} \phi - p\dot{\xi}_0\, \bar \zz^{p-1} \, \pp_x \phi  - \phi  + p \bar \zz^{p-1} \, \phi + \bar \zz^{p-1} \, 
\bar g
\end{equation}
with
$$\bar \zz(x,t) := w(x+\xi(t) - \xi_0(t))+ w(x-\xi(t)-\xi_0(t))$$
and $\bar g(x,t) :=g(x-\xi_0(t),t)$. We observe that on the interval of consideration  $-\infty < x \leq \xi_0(t)$, we have $w(x-\xi(t)-\xi_0(t)) \leq w(x+\xi(t)-\xi_0(t))$, hence 
\begin{equation}\label{eqn-Ww}
w(x+\xi(t)-\xi_0(t)) \leq \zz(x,t) \leq 2 \, w(x+\xi(t)-\xi_0(t)), \qquad \mbox{on} \,\,\, -\infty < x \leq \xi_0(t).
\end{equation}
If we divide equation  \equ{eqn-phi0} by $\bar z^{p-1}$ and perform the change of variables

\begin{equation}\label{eqn-tildephi}
\phi(x,t) = e^x \ttt \phi ( r, t), \qquad r=e^{\beta x},
\end{equation}

\no we conclude, after a simple calculation, that  the new function $ \tilde \phi(r,t)$ satisfies  the equation

\be
\ttt \phi_t  = \al(r,t) \,  \Delta \ttt \phi - \beta \,  \dot{\xi}_0 \, r \ttt \phi_r 
+ (1-\dot{\xi}_0)\,  \ttt \phi  + \ttt{g}(r,t)
\label{eq2}\ee
with
$$\al(r,t)=p^{-1} \beta^2 U^{1-p}(r,t)$$ 
and
$$U(r,t)=e^{-x}\, ( w(x+\xi(t)-\xi_0(t)) + w(x-\xi(t)-\xi_0(t))) , \qquad r=e^{\beta x}.
$$
 To obtain \eqref{eq2}, we compute directly that 
$$
\phi_{xx} -\phi  =  \beta^2 e^{(1+ 2\beta)x} (\ttt \phi_{rr} + \frac{n-1}r \ttt \phi_r )
$$
and
$$
p  \bar \zz^{p-1} \phi_t =     p  e^x \ttt \phi_t  \,  ( e^x U )^{p-1} =
p  e^{(1+ 2\beta)x}  \,  U^{p-1} \ttt \phi_t
$$
and similarly
$$
p \bar \zz^{p-1} \bar g=   p e^{(1+ 2\beta)x} U^{p-1}\ttt g.
$$
Combining the last three equalities we readily conclude  \equ{eq2}. In addition, it  follows from \eqref{eqn-w1} and \eqref{eqn-Ww} that 
\be\label{eqn-Ur}
   \left ( \frac{ 2\, k_n\, e^{\beta(\xi_0(t)-\xi(t))}    }{ e^{2\beta  (\xi_0(t)-\xi(t))} +  r^2}\right )^{\frac 1\beta} 
   \leq U(r,t)  \leq 2\,     \left ( \frac{ 2\, k_n \, e^{\beta (\xi_0(t)-\xi(t))}   }{ e^{2\beta (\xi_0(t)-\xi(t))} +  r^2}\right )^{\frac 1\beta}.
\ee
Observe that $|\xi_0(t)-\xi(t)| = |h(t)| \le C\, |t|^{-\mu}$, since $\| h \|^{\mu, 1+\mu}_{1, \sigma, t_0} < \infty$
by assumption. This together with (\ref{eqn-Ur}), the fact that $p - 1 = \frac{4}{n-2} = 2\beta$,  imply the estimate for the ellipticity coefficient $\alpha(r,t)$,
$$d_1(1/2 + r^2)^2 \le \alpha(r,t) \le d_2 (1+r^2)^2,$$
for $d_1$ and  $d_2$ universal positive constants.

\smallskip 
We fix $\tau \leq t_0$. We will next establish sharp $W^{2,\sigma}$ estimates for equation \eqref{eq2} 
on $B_{R(\tau)} \times [\tau-2,\tau]$,  where $R(\tau) := e^{\beta \xi_0(\tau)}$ is a large number. 
\no Let $\bar Q := B_2 \times [\tau-2,\tau]$ and $Q := B_1 \times [\tau-1,\tau]$. By the  standard parabolic 
$W^{2,\sigma}$,  we have
\be \label{eqn-w2p1}  \| \ttt \phi \|_{W^{2,\sigma}(Q)} \leq C \, \big (  
\| \ttt \phi \|_{L^\sigma(\bar Q)} + \| \ttt g \|_{L^\sigma(\bar Q)} \big ).
\ee
Translating this estimate   back to the original coordinates and in terms of $\psi$ gives us the desired weighted $W^{2,\sigma}$ bound on the exterior region, namely 
\be \label{eqn-w2p11} \|\psi \|_{2,\sigma,E_\tau} 
\leq C \, \big (   \|\psi \|_{\sigma, \bar E_\tau}  
+ \| g \|_{\sigma,\bar E_\tau}    \big ) 
\ee
where $E_\tau= (-\infty,-\xi_0(\tau))  \times [\tau-1,\tau]$ and $\bar E_\tau = (-\infty,-\xi_0(\tau) + \frac {\ln 2}\beta)   \times [\tau-2,\tau]$. 
\medskip

We will next obtain a weighted $W^{2,\sigma}$ estimate on $B_{R(\tau)}  \setminus  B_1$.
To this end, we will assume that $R(\tau)=2^{k_0}$ for a large constant $k_0=k_0(\tau)$ and we will 
derive the estimate on the annuli 
$$\{ 2^k < |x| < 2^{k+1} \}  \times [\tau-1,\tau], \,\,\,   \mbox{for any}\,\,  k=0, \cdots, k_0-1.$$
  Set $\rho=2^k$,  $D_\rho = \{\rho  < r < 2\rho \}  \times [\tau-1,\tau]$ and by $\bar D_\rho =  \{\frac \rho 2 < r < 4\rho \}  \times [\tau-2,\tau]$. Then, on $\bar D_\rho$ we have
$$ \la \rho^4 \leq  \alpha(r,t) \leq \Lambda \rho^4$$
for $\la >0$ and $\Lambda < \infty$ universal constants. 
We will then divide the time interval $[\tau-1,\tau]$ into subintervals of length $1/\rho^2$ and
in each of them we scale our solution $\ttt \phi$ to make the equation \eqref{eq2} strictly
parabolic.
Let us denote by $[s-1/{\rho^2},s]$ one such sub-interval
and consider the cylindrical regions   $D_{\rho}^s=  \{\rho  < r < 2\rho \}  \times [s-1/{\rho^2},s]$ and $\bar D_{\rho}^s=  \{\frac \rho2  < r < 4\rho \}  \times [s-2/{\rho^2},s].$
It follows that  the rescaled solution 
$$\phi_\rho (r,t) := \ttt \phi(\rho \, r, s + \rho^{-2} t), \qquad (r,t) \in \bar D := \{ \frac 12 < r <  4
\} \times [\tau-2,\tau]$$
satisfies the equation  
\be
\pp_t  \phi_\rho  = \frac{ \alpha(r,t)}{\rho^4}  \,   \Delta  \phi_\rho  - \beta \dot \xi_0(t)\, \frac {r}{\rho^3} \pp_r \phi_\rho + 
\frac 1{\rho^2} (1- \dot \xi_0(t))  \phi_\rho  +  \frac 1{\rho^2} f_\rho(r,t)
\label{eq3}\ee
on $\bar D$ with $g_\rho(r,t) := \ttt g(\rho r, s + \rho^{-2} t)$.  Moreover, 
$$ \la  \leq  \frac{\alpha(r,t)}{\rho^4}  \leq \Lambda.$$
Recall also that $\dot \xi_0 (t) = 1/t$, and in particular, $|\dot{\xi}_0|$ is bounded. 
Hence, by the standard $W^{2,\sigma}$ estimates on equation  \eqref{eq3},  we have
$$\| \phi_\rho \|_{L^{2,\si}(D)}  \leq C \big ( \, \| \phi_\rho \|_{L^\si(\bar D)} + \rho^{-2} \,  
\|g_\rho\|_{L^\si(\bar D)} \big )$$
with $D := \{ 1 < r <  2 \} \times [\tau-1,\tau]$ and $\bar D:=\{ 1/2 < r <  4 \} \times [\tau-2,\tau]$.
This readily yields  the bound
\begin{equation*}
\begin{split}
 \| \ttt \phi_t \|_{L^\si(D_{\rho}^s)} +  \rho^4 \| D^2 \ttt \phi \|_{L^\si(D_{\rho}^s)}
+ \rho^3 \|  &D \ttt \phi \|_{L^\si(D_{\rho}^s)} + \rho^2 \,  \| \ttt \phi \|_{L^\si(D_{\rho}^s)}\\ &\leq 
C \big ( \rho^2 \,  \| \ttt \phi \|_{L^\si(\bar D_{\rho}^s)} +  
\| \ttt g\|_{L^\si(\bar D_{\rho}^s)} \big ).
\end{split}
\end{equation*}
By repeating the above estimate on all time sub-intervals we finally conclude 
\be\label{eqn-w2p2}
\begin{split}
 \| \ttt \phi_t \|_{L^\si(D_{\rho})} +  \rho^4 \| D^2 \ttt \phi \|_{L^\si(D_{\rho})}
+ \rho^3 \| &D \ttt \phi \|_{L^\si(D_{\rho})} + \rho^2 \,  \| \ttt \phi \|_{L^\si(D_{\rho})}\\ &\leq 
C \big ( \rho^2 \,  \| \ttt \phi \|_{L^\si(\bar D_{\rho})} +  \|\ttt g\|_{L^\si(\bar D_{\rho})} \big ).
\end{split}
\end{equation}

Because the first terms on the left hand side of \eqref{eqn-w2p2} have a growth in
$\rho$, we will need to weight the $L^\sigma$ norms by a power $r^\la$,
for some appropriate $\lambda <0$ to be chosen in the sequel. To this end, we define for any function $\tilde h$ the norm
$$\| \tilde h \|_{L^\si_\lambda(A)} = \big ( \iint_A    | \ttt h |^\sigma \, r^{ \lambda  + n-1}\, dr dt  \big )^{\frac 1\sigma}$$
and observe that \eqref{eqn-w2p2} readily implies the following estimate in the new norms
\be\label{eqn-w2p3}
\begin{split}
 \| \ttt \phi_t \|_{L^\si_\lambda(D_{\rho})} +   \| r^4 \, D^2 \ttt \phi \|_{L^\si_\lambda(D_{\rho})}
&+  \| r^3 D \ttt \phi \|_{L^\si_\lambda(D_{\rho})} +  \| r^2 \ttt \phi \|_{L^\si_\lambda
(D_\rho)}\\ &\leq 
C \big (   \| r^2 \ttt \phi \|_{L^\si_\lambda(\bar D_{\rho})}   +  \|\ttt g\|_{L^\si_\lambda(\bar D_\tau)} \big ).
\end{split}
\end{equation}

\ssk

We will next use the above local  estimate to establish an estimate on the entire
inner region. To this end, set $D_\tau= \{ 1 < r < R(\tau) \} \times [\tau-1,\tau]$ and
$\bar D_\tau= \{ 1/2  < r < 2R(\tau) \} \times [\tau-2,\tau]$,  where $R(\tau) :=e^{2\beta \xi_0(\tau)}$, as before. Applying \eqref{eqn-w2p3}
for all $\rho=2^k$, $k=0,1,\cdots k_0$, where $R(\tau)=2^{k_0}$, we obtain the bound
\be\label{eqn-w2p31}
\begin{split}
 \| \ttt \phi_t \|_{L^\si_\lambda(D_\tau)} +   \| r^4 \, D^2 \ttt \phi \|_{L^\si_\lambda(D_\tau)}
+  \| r^3 D \ttt \phi &\|_{L^\si_\lambda(D_\tau)} +  \| r^2 \ttt \phi \|_{L^\si_\lambda
(D_\tau)}\\ &\leq 
C \big (   \| r^2 \ttt \phi \|_{L^\si_\lambda(\bar D_\tau)}   +  \|\ttt g \|_{L^\si_\lambda(\bar D_\tau)} \big ).
\end{split}
\end{equation}

\ssk

\no Before we find the appropriate $\la$, we will express \eqref{eqn-w2p31}
back in terms of the 
functions $\phi(x,t)=e^x \ttt \phi(r,t)$ and $f(x,t)=e^x \ttt f(r,t)$ through the change of variables $r=e^{\beta x}$. Let  
$I_\tau := [0, \xi_0(\tau)] \times [\tau-1,\tau]$ and
$\bar I(\tau) := [- (\ln 2)/{\beta}, \xi_0(\tau)+ (\ln 4)/\beta ]\times [\tau-2,\tau]$
denote the images of the sets $D_\tau$ and $\bar D_\tau$ under this change of variables.
We will use the  formula 
${\partial}/{\partial r} =(\beta r)^{-1}  \,   {\partial}/{\partial x}$ and the bounds
\begin{equation}\label{eqn-wxw}
c e^{-\beta x} \leq w^\beta(x) \leq  C\, e^{-\beta x}
\end{equation}
which hold in the region of consideration. 
 A  direct calculation  
shows that \eqref{eqn-w2p31} implies the bound 
\be\label{eqn-w2p5}
\begin{split}
 \| \pp_t \phi  \|_{L^\si_\lambda(I_\tau)} +   \|  \pp^2_x  \phi \|_{L^\si_\lambda(I_\tau)}
+ & \| \pp_x   \phi \|_{L^\si_\lambda(I_\tau)} \\ &\leq 
C \big (   \| e^{2\beta x} \phi \|_{L^\si_\lambda(\bar I_\tau)}   +  
\| \bar g\|_{L^\si_\lambda(\bar I_\tau)} \big )
\end{split}
\end{equation}
where, for any function $h$ and any $I \subset [0,\infty)  \times (-\infty,0]$,  we denote 
$$\| h \|_{L^\si_\lambda(I)} := \big ( \iint_I   |  h |^\sigma \, e^{( \lambda  + n)\beta x}\, dx dt  \big )^{\frac 1\sigma}.$$

We next observe that the same arguments as in the proof of Lemma
\ref{lem-energy1} give us a  global $L^\infty $ bound on the solution $\psi$ of \eqref{eqn-psi},
namely  
$$\| \psi \|_{L^\infty(\R \times (-\infty, t_0])} \leq C \, \big (   \|\psi\|_{L^2_{t_0}} +\|g\|_{L^2_{t_0}} \big ).$$
which gives a similar bound for $\phi$, namely  
$$\| \phi \|_{L^\infty(\R \times (-\infty, t_0])} \leq C \, \big (  \|\phi\|_{L^2_{t_0}} + \|\bar g\|_{L^2_{t_0}})$$
were
$$\|\bar g\|_{L^2_{t_0}} := \sup_{\tau \leq t_0} \big 
(\iint_{\Lambda_\tau}  \bar g^2 \, \bar \zz^{p-1} \, dx \, dt \big )^{\frac 12}.$$
\medskip

\no Using this bound we obtain 
$$
\| e^{2\beta x} \phi \|_{L^\si_\lambda(\bar I_\tau)} \leq C\, \big(\, \|\bar g\|_{L^2_{t_0}} + \|\phi\|_{L_{t_0}^2} \,\big) \,
\left ( \iint_{\bar I_\tau} e^{(2 \sigma + \lambda + n)\beta x}\, dx dt \right )^{\frac 1\sigma}.$$
The last integral is bounded uniformly in $\tau$ if $\la$ is chosen so that 
$$2 \sigma + \lambda + n <0.$$
Choose $\lambda= - (2\sigma +n + \theta)$, with   $\theta >0$ any small universal constant. 
With this choice of $\lambda$ and for any 
function $h$ we have 
$$\| h \|_{L^\si_\lambda(I)} = \big ( \iint_I   |  h |^\sigma \, e^{( \lambda  + n)\beta x}\, dx dt  \big )^{\frac 1\sigma}=\big ( \iint_I   |  h |^\sigma \, e^{-(2\sigma +\theta) \beta x}\, dx dt  \big )^{\frac 1\sigma}.$$
With such a choice of $\lambda$, combining this last estimate with \eqref{eqn-w2p5},  yields to the bound
$$
 \|  \phi \|_{2,\sigma,I_\tau} \leq 
C \big (   \|\phi \|_{L^2_{t_0}} +   \|\bar g\|_{L^2_{t_0}} +
\| \bar g\|_{L^\si_\lambda(\bar I_\tau)} \big ).
$$
This readily gives the desired $W^{2,\sigma}$ estimate on $\psi$ in the intermediate
region, which combined with \eqref{eqn-w2p11} yields to  \eqref{eqn-lemw2p},
finishing the proof of the lemma.
\end{proof}

We recall next the weighted $L^\infty$ norm and our global norm  given in Definitions \ref{defn-linfty}  and \ref{defn-norm} respectively.  It is clear that $$\| \psi \|^\nu_{L^\infty_{t_0}} \leq C\, \| \psi \|_{\infty, t_0}^\nu$$ since 
$z^{-1} \geq c >0$ for a universal constant $c$.
\no The following  $L^\infty$ estimate follows as a consequence of estimates \eqref{eqn-basic0}
and \eqref{eqn-w2p}. To derive it, as we will see below we need to take $\sigma > n+1$, so lets define from now on $\sigma := n+2$. We have

\begin{corollary} \label{coro-linfty} Under the assumptions of Proposition \ref{prop-w2p}, if $\sigma =n+2$  then the solution $\psi$ satisfies the
following estimate
\begin{equation}\label{eqn-linfty1}
\| \psi \|_{\infty, t_0}^\nu
\leq  C\, \big (  \|f\|_{L^2_{t_0}}^\nu + \|f\|_{\sigma,t_0}^\nu \big ).
\end{equation}

\end{corollary} 

\begin{proof} The bound on $\| \psi  \|_{L^\infty_{t_0}}^\nu$ readily follows from \eqref{eqn-basic0} and
Sobolev embedding. For the bound on $\| z^{-1} \, \psi  \,  \chi_{\{|x| \geq \xi(t)\}}  \|_{L^\infty_{t_0}}^\nu$,
by symmetry we may  restrict  ourselves to the region $\{ -\infty < x < - \xi(t) +  \beta^{-1}\ln 2 \}$, and set  
$\phi(x,t) :=\psi(x-\xi_0(t),t)$, where $-\infty < x < \beta^{-1}\ln 2$.  As in the proof of Lemma \ref{lem-w2p}, it follows
that $\phi$ satisfies the equation \eqref{eqn-phi0} with $\bar g:= g(x-\xi_0(t),t)$ and $g := f - C(\psi,t)$. 
Hence, $\tilde \phi(r,t)$ given by \eqref{eqn-tildephi} satisfies the equation \eqref{eq2} which is
now strictly  parabolic in the region of consideration $0 \leq r < 2$, $t \leq t_0$. 
Let $Q= B_{e^{\epsilon}} \times [\tau-1,\tau]$ and $\bar Q= B_2 \times [\tau-2,\tau]$, $\tau \leq  t_0$ and $e^{\epsilon} < 2$. Standard parabolic estimates imply the bound
$$\| \tilde \phi \|_{L^\infty(Q)} \leq C \, \big ( \| \tilde \phi \|_{L^\sigma(\bar Q)} + \| \tilde g \|_{L^\sigma(\bar Q)} \big )$$ 
since  $\sigma >  n+1$.  
Expressing everything in the original variables, using that $|\xi(t) - \xi_0(t)| = |h(t)|  \le C\, |t|^{-\mu} < \epsilon$, for $|t|$ sufficiently large, we conclude
$$\| z^{-1} \, \psi  \,  \chi_{\{|x| \geq \xi(t)\}}  \|_{L^\infty_{t_0}}^\nu \leq \| z^{-1} \, \psi  \,  \chi_{\{|x| \geq \xi_0(t) - \epsilon\}}  \|_{L^\infty_{t_0}}^\nu \leq \big ( \| \psi \|_{\sigma,t_0}^\nu + \|g\|_{\sigma,t_0}^\nu \big ).$$
Since  $g=f-C(\psi,t)$,  where 
$C(\psi,t)$ satisfies the bound \eqref{eqn-cp0}, 
the desired bound readily follows from \eqref{eqn-w2p}.
\end{proof}

\medskip

We will finally summarize the results in this section in one result. This will play a crucial
role in the construction of the solution of our nonlinear problem. 
We have shown the following result.

\begin{prop}\label{prop-linear} Let $\mu, \nu \in [0,1)$ and $\sigma = n+2$, $n \geq 3$,  be fixed constants.  
Then, there is a number  $t_0 <0$ such that  for any even function  $f$ on $\R \times (-\infty,t_0]$
with
$\|f \|_{*,\sigma}^\nu < \infty$,   satisfying the 
orthogonality conditions \eqref{eqn-orth1} and \eqref{eqn-orth2} and a function $h$ on $(-\infty, t_0]$  with
$\| h \|_{1,\sigma, t_0}^{\mu, 1+\mu} < \infty$, 
there exists an ancient  solution $\psi = T(f)$ of \eqref{eqn-psi} on $-\infty \leq  t \leq t_0$, also satisfying the orthogonality conditions \eqref{eqn-orth1} and \eqref{eqn-orth2},  and the estimate 
\begin{equation}\label{eqn-basic02}
\|\psi \|_{*,2,\sigma,t_0}^\nu \leq C \, \|f \|_{*,\sigma,t_0}^\nu.
\end{equation}
The  constant $C$ depends  only on dimension $n$, $ \nu$ and $\mu$. 
\end{prop}

\section{The nonlinear problem}
\label{sec-np}

Let   $X$ be the  Banach space defined as in Definition \ref{defn-X} and $T$ and $A$ the operators as introduced 
in Section \ref{sec-outline}.  
In addition, for a given $\mu \in (0,1)$ and  $t_0< 0$, we set  
\begin{equation}\label{defn-K}
K:= \{\, (h,\eta) : (-\infty,t_0] \to \R  \,\,  |\quad  \|h\|^{\mu,1+\mu}_{1,\sigma,t_0}  \le 1 \,\,  \mbox{and} \,\, 
\| \eta \|^1_{1,\sigma,t_0} \leq C_0\, \}
\end{equation}
where $C_0$ is the universal constant  given by \eqref{eqn-defnco}.  
We say that $h \in K$ or  $\eta \in K$ if $(h,0) \in K$ or  $(0,\eta) \in K$ respectively.  Moreover, define
\begin{equation}
\label{eq-lambda}
\Lambda := \{\psi\in X \, |\, \| \psi\|_{*,2,\sigma,\nu} \le 1\}.
\end{equation}
\begin{remark} \label{rem-small}  If $(h,\eta) \in K$, then 
\begin{equation}\label{eqn-small}
\| h\|_{L^\infty(-\infty,t_0])} \leq |t_0|^{-\mu} \quad \mbox{and} \quad \| \eta\|_{L^\infty(-\infty,t_0])} 
\leq C_0\,  |t_0|^{-1}.
\end{equation}
In particular, by choosing $|t_0|$ sufficiently large we may assume that both $h$ and $\eta$ have  small 
$L^\infty$ norms. In addition, 
\begin{equation}\label{eqn-xi0}
 \xi(t) := \frac 12 \log (2b\,|t|) + h(t) = \frac 12 \log (2b\,|t|) + o(1), \qquad \mbox{ as} \,\, |t| \to -\infty. 
\end{equation} 
 \end{remark}

The main goal in this section is to prove the following Proposition.

\begin{prop}\label{prop-fp} Let $\sigma=n+2$. There exist   numbers    $\nu \in (\frac 12,1)$ and $t_0 <0$,   depending 
only on dimension $n$,  such that for any given pair of functions $(h,\eta)$ in $K$,  there is a solution $\psi = 
\Psi(h,\eta)$ of \eqref{eqn-fp} which  satisfies
the orthogonality conditions \eqref{eqn-orth11}-\eqref{eqn-orth22}. Moreover, the following estimates hold
\begin{equation}\label{eqn-contr1}
\| \Psi(h^1,\eta) - \Psi(h^2,\eta) \|^{\nu}_{*,2,\sigma,t_0}  \leq C \, |t_0|^{-\mu} \, \| h^1 - h^2 \|_{1,\sigma,t_0}^{\mu,\mu+1}
\end{equation}
and
\begin{equation}\label{eqn-contr2}
\| \Psi(h,\eta^1) - \Psi(h,\eta^2) \|^{\nu}_{*,2,\sigma,t_0}  \leq C \, |t_0|^{-1+\nu} \,  \| \eta^1 - \eta^2 \|_{1,\sigma,t_0}^1
\end{equation}
for any $(h^i,\eta) \in K$ and $(h,\eta^i) \in K$, $i=1,2$ and $\mu < \min\{2\nu-1, \gamma\}$, where $\gamma \in (0,1)$ is a positive number determined by Lemma \ref{lem-projection} and $C$ is a universal constant.  
\end{prop}

We will find a solution of \equ{eqn-fp} by the contraction mapping principle. To this end, we need suitable estimates on the operator 
$E(\psi)$.
They are given in  the following section, after which we will proceed with the proof of Proposition \ref{prop-fp}.

\subsection{The estimation of the Error term}

We will next estimate the error term $E(\psi)$ in the $\| \cdot \|_{*,\sigma,t_0}^\nu$ norm
and also establish its Lipchitz property with respect to $\psi$ as well as $h$ and $\eta$.  
We will begin by estimating the error term $M$ in  \eqref{eq-Epsi}.

\begin{lem}\label{lem-M} Let $\sigma =n+2$. There exist  numbers $\nu=\nu_0(n)  \in (\frac 12, 1]$ and $t_0 <0$, depending on dimension $n$, such that for any $\nu\in (1/2, \nu_0]$ and $\mu < \min\{2\nu - 1, \gamma\}$ and any $(h, \eta) \in K$ (where the set $K$ is defined with respect to this particular $\mu$), we have  
$$\| z^{1-p} \, M \|_{*,\sigma,t_0}^{\bnu} \leq C$$
for a universal constant $C$.
\end{lem}

\begin{proof} Throughout  the proof $C$ will denote various universal constants. Since  all the functions involved, including $M$, are even in $x$ it will be sufficient
to restrict  our computation to the region $x >0$. Notice  in that region
$$w_2(x,t) :=w(x+\xi(t)) \leq w(x-\xi(t)) = :w_1(x,t).$$
 
\no We write $M=M_1+M_2$,  where
$$M_1=\tilde{z}^p  - (1+\eta)^p \, (w_1^p+w_2^p) , \quad M_2 = \big [(1+\eta)^p - (1+\eta)\big ]\,  (w_1^p+w_2^p)- \pp_t \tilde{z}^p$$
and  set
$$\bar M_1=z^{1-p} \, M_1 \quad \mbox{and} \quad \bar M_2 = z^{1-p}\, M_2.$$
 
\no We have 
\begin{equation*}
\begin{split}
0 \leq \bar M_1  &= z^{1-p} \, (1+\eta)^p\,\big (  [ (w_1+w_2)^p -  w_1^p] - w_2^p \big )\\
& \leq p (1+\eta)^p w_2  \int_0^1  \frac{(w_1+sw_2)^{p-1}}{z^{p-1}} \, ds 
\end{split}
\end{equation*}
hence, using that $w_2 \leq w_1 \leq z$, obtain the bound
$$
|\bar  M_1| \leq C w_2. 
$$
For the moment take $\sigma \geq 2$ to be any constant and $q>0$. By the  last bound  and the estimate $z \leq 2 w_1$ which holds on $x >0$, we compute  
\begin{equation*}
\begin{split}
\big ( \int_0^{\xi(t)}  \big .  \big .  {\bar M_1}^\sigma \, z^{q \sigma} \, dx   \big )^{\frac 1\sigma} &
\leq  C \big ( \int_0^{\xi(t)}  w^\sigma(x+\xi(t))  \, w^{q \sigma} 
(x-\xi(t))  \, dx \big )^{\frac 1\sigma}\\ & \leq   C \,\big ( \int_0^{\xi(t)}  e^{-\sigma (x+\xi(t))}   \, 
e^{q \sigma(x-\xi(t))}   \, dx \big )^{\frac 1\sigma} \\&= C\, e^{-(1+q)\, \xi(t)} \big ( \int_0^{\xi(t)}  
e^{(q -1)\sigma x}   \, dx \big )^{\frac 1\sigma}.
\end{split}
\end{equation*}
Recalling \eqref{eqn-xi0}  we then conclude that
$$\big ( \int_0^{\xi(t)}  \bar M_1^\sigma \, z^{q \sigma} \, dx   \big)^{\frac 1\sigma} \leq C \, \alpha(|t|)$$
where
$\alpha(|t|) = |t|^{-\frac{1+q}2}$, if $q < 1$, $\alpha(|t|) =  |t|^{-1}$, if $q >1$, and  
$\alpha(|t|) = (\ln |t|)^{1/\sigma} \, |t|^{-1}$, if  
$q =1.$

\medskip

\no On the other hand, recalling that $\beta :=2/(n-2)=(p-1)/2$, we have
\begin{equation*}
\begin{split}
\big  ( \int_{\xi(t)}^\infty     \bar M_1^\sigma \, z^{n\beta -\sigma} \, dx   \big  )^{\frac 1\sigma} &\leq  C \big( \int_{\xi(t)}^\infty
   w^\sigma(x+\xi(t))  \, w^{n\beta -\sigma} 
(x-\xi(t))  \, dx \big )^{\frac 1\sigma}\\ & \leq C \,\big ( \int_{\xi(t)}^\infty  e^{-\sigma (x+\xi(t))}   \, 
e^{-(n\beta -\sigma)(x-\xi(t))}   \, dx \big )^{\frac 1\sigma} \\&= C e^{-2\xi(t)} 
\big( \int_{\xi(t)}^\infty   \, 
e^{-(n\beta)(x-\xi (t))}   \, dx \big)^{\frac 1\sigma} \, 
\leq C\, |t|^{- 1}. 
\end{split}
\end{equation*}

\no First, we combine  the above estimates  when $q=\beta$ and $\sigma=2$,  to  conclude that
on $\Lambda_\tau = \R \times [\tau,\tau+1]$, $\tau < t_0-1$,  we have 
$
\| \bar M_1 \|_{L^2(\Lambda_\tau)}  \leq C \big  |\tau|^{-\bnu}, 
$
with  $\bnu:=\bnu(\beta) > 1/2$ for all $\beta$. It follows that
$\| \bar M_1 \|_{L^2_{t_0}}^{\bnu} \leq C.$ 
 Also,   if  we combine the above estimates for   $\sigma = n+2$  and   $q=2\beta + \theta$, with $\theta$ a small 
universal positive constant as in the proof of Proposition \ref{prop-w2p}, we obtain (recall the Definitions  \eqref{eqn-w2p1} and \eqref{eqn-w2p2}) that
$
\| \bar M_1 \|_{\sigma, \Lambda_\tau}  \leq C \big  |\tau|^{-\bnu}, 
$
where $\bnu = \bnu(\beta) > 1/2$.
By choosing $t_0 < -1$, we obtain
$\| \bar M_1 \|_{\sigma,t_0}^{\bnu} \leq C$.  We conclude that
$$\| \bar M_1 \|_{*,\sigma,t_0}^{\bnu} \leq C.$$ 

\medskip
We will now estimate the term involving $\bar M_2$. Since $w_2 \leq w_1 \leq z$, we have
$$
\abs{\bar M_2} \le C(p)\, \big(|\eta|\, w_1 + |\dot{\xi}(t)| \,  \big  |\, w'(x+\xi(t)) - w'(x-\xi(t) \big) | + |\dot{\eta}(t)| \, z  \big)
$$
hence, using that $|w'(x)| \leq C\, w(x)$, we obtain 
$$
\abs{\bar M_2} \le C \, (|\eta| +|\dot \eta | + |\dot \xi|)\, z. 
$$
It follows that  for any $q >0$,  $\sigma \geq 2$ and $\tau < t_0-1$, we have 
\begin{equation*}
\begin{split}
\big (\iint_{\Lambda_\tau}  |\bar M_2|^\sigma  \,\big [  z^{q\sigma} \chi_ {\{|x| < \xi(t)\} } &+
 z^{n\beta -\sigma}\, \chi_ {\{|x| \geq  \xi(t)\} }  \big ] \, dx\, dt \big)^{\frac 1\sigma} \\&
 \le C\big (\int_{\tau}^{\tau+1} (|\eta|^{\sigma} + |\dot{\xi}|^{\sigma} + |\dot{\eta}|^{\sigma})\, dt\big)^{\frac 1\sigma}.
 \end{split}
\end{equation*}

\no By  Definition \ref{dfn-heta}  the right hand side of the last estimate is bounded
by $C\, |\tau|^{-1}$, if we assume that $\|\eta \|_{\sigma,t_0}^1 \leq C_0$ and 
$\|h \|_{\sigma,t_0}^{\mu,1+\mu} \leq 1$, with $\mu >0$. Hence, arguing as before,  we easily conclude the bound 
$$\| \bar M_2 \|_{*,\sigma,t_0}^{\bnu} \leq C.$$
This finishes the proof of the Lemma. 
\end{proof}

\medskip

The following Corollary follows  immediately  from  Lemma \ref{lem-M} by choosing $\nu=\nu(n)$ to be any number
in $(\frac 12,\nu_0)$ and $t_0  <0$ so that   $C\, |t_0|^{-(\nu_0-\nu)} < \frac 12$. Let also $\mu < \min\{2\nu - 1, \gamma\}$, where $\gamma \in (0,1)$ is a positive number determined by Lemma \ref{lem-projection}.

\begin{coro}\label{cor-M} Let $\sigma =n+2$ and let  $\nu=\nu(n)  \in (\frac 12, 1)$, $t_0 <0$, $\mu$ be as above. There exist uniform constants $t_0 < 0$ and $C>0$,  depending on dimension $n$, such that
for any $(h,\eta) \in K$,  we have 
$$\| z^{1-p} \, M \|_{*,\sigma,t_0}^{\nu} \leq C\, |t_0|^{-(\nu_0-\nu)}.$$

\end{coro} 

\smallskip
For the remaining of the section we will fix the parameters $\sigma, \mu$ and $\nu$ as in Corollary \ref{cor-M}.  
We will next establish   an $L^\infty$ bound on $\psi/z$ which will be used very frequently in  the rest of the article. 

\begin{claim}\label{claim-linfty} 
For any function $\psi$ on $\R \times (-\infty,t_0]$, we have
\begin{equation}\label{eqn-linfty11}
 \| \frac{\psi}z    \|_{L^\infty_{t_0}} 
\leq  C\, |t_0|^{\frac 12-\nu} \, \| \psi \|_{*,2,\sigma,t_0}^\nu  
\end{equation}
for a universal constant $C$. 
Hence, under the assumption $\| \psi \|_{*,2,\sigma,t_0}^\nu \leq1$,  
\begin{equation}\label{eqn-linfty}
 \| \frac{\psi}z    \|_{L^\infty_{t_0}} 
\leq  C\, |t_0|^{\frac 12-\nu}   
\end{equation}
and it  can be made sufficiently small by taking $|t_0|$ sufficiently 
large. 
\end{claim}

\begin{proof}[Proof of Claim]   Recalling  the definitions of our norms from section \ref{sec-norms}, we have
\begin{equation}\label{eqn-linfty00}
\| \frac {\psi}{z} \chi_{\{ |x| \geq \xi(t)\} }   \|_{L^\infty_{t_0}} 
\leq C\, |t_0|^{-\nu} \, \| \psi \|_{*,2,\sigma,t_0}^\nu.
\end{equation}
 On the other hand for any $\tau \leq  t_0-1$,  on $\Lambda_\tau=\R \times [\tau,\tau+1]$ we have
 $$\| \frac {\psi } z \chi_{\{ |x| \leq \xi(t)\} }  \|_{L^\infty(\Lambda_\tau)} \leq  
 |\tau|^{-\nu} \, \| \psi  \|_{*,2,\sigma,t_0}^\nu \,  \|  z^{-1}   \|_{L^\infty(\Lambda_\tau \cap \{ |x| \leq \xi(t)\})}.$$
To estimate   $\|  z^{-1}  \|_{L^\infty(\Lambda_\tau \cap \{ |x| \leq \xi(t)\})}$ we observe that
 $$\min_{\{ |x| \leq \xi(t)\}} z(x,t) \geq \min_{\{0 <  x  \leq \xi(t)\}} w(x-\xi(t)) =w(\xi(t)) \geq C\, 
 |t|^{-\frac 12}$$
 since \eqref{eqn-xit} holds. 
 Hence, 
 $ \|  z^{-1}   \|_{L^\infty(\Lambda_\tau \cap \{ |x| \leq \xi(t)\})} \leq C\,  |\tau|^{\frac 12}.$
It follows that 
 $ \| z^{-1} \psi   \|_{L^\infty(\Lambda_\tau \cap \{ |x| \leq \xi(t)\})}   \leq C \,  |\tau|^{\frac 12- \nu} \,  \| \psi
 \|_{*,2,\sigma,t_0}^\nu$
 implying the bound  
 \begin{equation}\label{eqn-linfty10}
\| \frac {\psi}{z} \chi_{\{ |x| \leq \xi(t)\} }   \|_{L^\infty_{t_0}} 
\leq C\, |t_0|^{\frac 12-\nu} \, \| \psi \|_{*,2,\sigma,t_0}^\nu.
\end{equation}
The claim now follows from the two estimates  \eqref{eqn-linfty00} and \eqref{eqn-linfty10} and our assumption 
$\|  \psi\|_{*,2,\sigma,t_0}^\nu \leq 1$. 
 
 \end{proof}

We will next estimate the  norm of the term $(1-\pp_t) N(\psi)$ in \eqref{eq-Epsi}.

\begin{lem}\label{lem-N}  There exist  uniform constants $t_0 < 0$ and $C >0$  depending on dimension $n$ such that
for any functions $(h,\eta) \in K$ and  $\psi \in \Lambda$,  we have 
$$\| z^{1-p} \, (1-\pp_t) N(\psi) \|_{*,\sigma,t_0}^\nu \leq C\, \, |t_0|^{\frac 12-\nu}\, 
\| \psi \|_{*,2,\sigma,t_0}^\nu.$$
\end{lem}

\begin{proof} 
We write $N(\psi)=N_1(\psi) + N_2(\psi)$ where 
$$N_1(\psi)=(\tilde{z}+\psi)^p -  \tilde{z}^p -  p\, \tilde{z}^{p-1} \psi, \quad N_2(\psi)= 
p\psi z^{p-1}[(1+\eta)^{p-1} - 1].$$
To estimate $\|z^{1-p} \, (1-\pp_t) \,  N_1(\psi)\|_{*,\sigma,t_0}^\nu$, we begin by observing
that 
\begin{equation}\label{eqn-N4}
z^{1-p} \, N_1(\psi)  = (1+\eta)^{p-1}\, \tilde z  \big [ \big ( 1 + \frac{\psi}{\tilde{z}} \big )^p -1 - p \frac{\psi}{\tilde{z}} \big ] 
=p (1+\eta)^{p-1} \psi  A(\psi)
\end{equation}
where 
$$ A(\psi) = \int_0^1\big(\big( 1 + s  \frac{\psi}{\tilde{z}}   \big )^{p-1} - 1\big)\, ds.$$
By \eqref{eqn-linfty} we have
\begin{equation}\label{eqn-apsi}
|A(\psi)| \leq C(p) \, \frac{|\psi|}{\tilde{z}}   \leq C \, \frac{|\psi|}{z}
\end{equation}
and therefore we conclude the bound
$$
\| z^{1-p} \,N_1(\psi)\|_{*,\sigma,t_0}^\nu    \leq C\, \| \psi\|_{*,\sigma,t_0}  \, 
\| \frac{\psi}{z}   \|_{L^\infty_{t_0}} \leq C\, |t_0|^{\frac 12-\nu} \, \|  \psi\|_{*,2,\sigma,t_0}.
$$

\medskip

\no It remains to estimate $z^{1-p} \,\, \pp_t N_1(\psi)$. Differentiating  \eqref{eqn-N4} in $t$ we get
\begin{equation}\label{eqn-N7}
\pp_t N_1(\psi) = \underbrace{p \, \pp_t \tilde z^{p-1}  \, \psi  A(\psi)}_{N_{11}} + 
\underbrace{p \, \tilde z^{p-1}  \, \pp_t \psi\,   A(\psi)}_{N_{12}}  + \underbrace{
p  \, \tilde z^{p-1}  \, \psi  \pp_t A(\psi)}_{N_{13}}.
\end{equation}
Using \eqref{eqn-apsi} we obtain,  similarly  as before,  the bounds 

$$|z^{1-p} \, N_{11}| \leq C\, ( |\dot \xi| + | \dot \eta | )  \, |\psi| \,  \frac {| \psi | }{z}, \qquad |z^{1-p} \, N_{12}| \leq C\,  |\pp_t \psi | \,  \frac {|\psi | }{z}.$$
To estimate the term $|z^{1-p} \, N_{13} (\psi) |$, we first observe that since 
${|\psi |} /{\tilde z} < 1/2$ by \eqref{eqn-linfty} and $|t_0| >>1$, we have  
\begin{equation}\label{eqn-apsit}
|\pp_t A(\psi)| \leq C \frac{| \tilde{z} \pp_t \psi  - \tilde{z}_t \psi|}{\tilde{z}^2} \int_0^1 \big(1 + s\frac{\psi}{\tilde{z}}\big)^{p-2} s \, ds \leq C\, \frac{ |\pp_t \psi| + ( |\dot \xi| + |\dot \eta |)\, |\psi|}{z}.
\end{equation}
Hence
\begin{equation}\label{eqn-N17}
|z^{1-p} \, N_{13} (\psi) |   \leq C\, \frac{|\psi|}z \,  \big ( \, |\pp_t \psi| + ( |\dot \xi| + |\dot \eta |) \, |\psi| \, \big ).
\ee
Combining the above estimates with \eqref{eqn-linfty11} gives 
$$\|z^{1-p} \,\pp_t N_1 (\psi) \|_{*,\sigma,t_0}^\nu  \leq C\, |t_0|^{\frac 12-\nu} \,
  \|\psi\|_{*,2,\sigma,t_0}^\nu   ( \|\psi\|_{*,2,\sigma,t_0}^\nu  + \| ( |\dot \xi| + |\dot \eta |)\, \psi \|_{*,\sigma,t_0}^\nu  ).$$
However, a  direct computation shows that 
\begin{equation}\label{eqn-good16}
\|  ( |\dot \xi| + |\dot \eta |)\, \psi  \|_{*,\sigma,t_0}^\nu \leq C\, |t_0|^{-1} \big ( \| \dot \xi \|_{\sigma,t_0}^1   + 
\| \dot \eta \|_{\sigma,t_0}^1 \big ) \, \|\psi  \|_{*,2,\sigma,t_0}^\nu
\end{equation}
where  $\| \dot \eta \|_{\sigma,t_0}^1 \leq \|\eta \|_{1,\sigma,t_0}^1 \leq 1$ and   $ \| \dot \xi \|_{\sigma,t_0}^1 \leq \frac 12 + 
\| \dot h \|_{\sigma,t_0}^1 \leq \frac 12 + 
\|  h \|_{1,\sigma,t_0}^{\mu,1+\mu}  \leq 2$. 
Hence,
$$ \|  z^{1-p} \, \pp_t  N_1(\psi)\|_{*,\sigma,t_0}^\nu    \leq C\, |t_0|^{\frac 12 -\nu} \,  \|  \psi\|_{*,2,\sigma,t_0}^\nu.
$$
Since  $\|\eta\|_{*,1,\sigma,t_0}^1 \leq C_0$,  a computation along the lines of the previous estimate
also shows that 
\begin{equation}\label{eqn-N00}
\|z^{1-p} \, (1-\pp_t) \,  N_2(\psi)\|_{*,\sigma,t_0}^\nu \leq C\, |t_0|^{-1} \,   \| \psi \|_{*,2,\sigma,t_0}^\nu.
\end{equation}
The proof of the Lemma is now complete. 
\end{proof}

\begin{lem}\label{lem-E1} 
 There exist  $t_0 < 0$ and $C >0$, depending on dimension $n$, such that
for any functions $(h,\eta) \in K$ and $\psi \in \Lambda$,    we have 
\begin{equation}\label{eqn-bepsi1}
\| \bar E(\psi) \|_{*,\sigma,t_0}^\nu \leq C\, \left ( |t_0|^{-(\nu_0-\nu)} +  |t_0|^{\frac 12-\nu}\, 
\| \psi \|_{*,2,\sigma,t_0}^\nu \right ).
\end{equation}
\end{lem}

\begin{proof} Let $Q(\psi)$ be as in \eqref{eq-Epsi}. The estimate of the error term $N(\psi)$ given in Lemma \ref{lem-N}, 
the estimate of the correction term $C(\psi,t)$ given in Corollary  \ref{cor-cp3}, and the  bound 
$$\|z^{1-p}\psi  \pp_t z^{p-1} \|_{*,\sigma,t_0}^\nu \leq C\,  \|  | \dot \xi | \, \psi  \|_{*,\sigma,t_0}^\nu \leq 
  C\, |t_0|^{-1}  \,  \|\psi  \|_{*,2,\sigma,t_0}^\nu,$$
(where we have used the $L^{\infty}$ bound on $\psi$ given by Lemma \ref{lem-energy1}) yield
\begin{equation}\label{eqn-Qpsi}
 \| Q(\psi) \|_{*,\sigma,t_0}^\nu \leq C\,  |t_0|^{\frac 12-\nu}\, \| \psi \|_{*,2,\sigma,t_0}^\nu.\end{equation}
This combined with the estimate in Lemma \ref{lem-M} easily imply \eqref{eqn-bepsi1}. 
\end{proof}

We will next show the Lipschitz property of $E(\psi)$ with respect to $\psi$. 

\begin{lem}\label{lem-E2} 
There exist $t_0 < 0$ and $C >0$, depending on dimension $n$, such that
for any functions $(h,\eta) \in K$ and $\psi^1, \psi^2  \in \Lambda$,    we have   
\begin{equation}\label{eqn-bepsi2}
\| E(\psi^1) -  E(\psi^2) \|_{*,\sigma,t_0}^\nu \leq   C\, \, |t_0|^{\frac 12-\nu}\, 
\| \psi^1 - \psi^2 \|_{*,2,\sigma,t_0}^\nu.
\end{equation}
\end{lem}

\begin{proof} We begin by observing that  the bound
$$
\|C(\psi^1,t) - C(\psi^2,t)\|_{*,\sigma,t_0}^{\nu} \le \frac{C}{\sqrt{|t_0|}}\|\psi^1 - \psi^2\|_{*,2,\sigma,t_0},
$$
follows similarly as the bound in Corollary \ref{cor-cp3}. 

All the other estimates    are similar to those in Lemma \ref{lem-N}, so we will omit   most of the
details. Using the notation in the proof of Lemma \ref{lem-N}, let us look at  the  estimate of the term  
$$|z^{1-p} \, \big (N_{13}(\psi^1) - N_{13}(\psi^2) \big ) | \leq  
C \, \big ( \, |\psi^1 - \psi^2 |   \, |\pp_t A(\psi^1)| + |\psi^2|\,  |\pp_t A(\psi^1) - \pp_t A(\psi^2) | \, \big).$$
By \eqref{eqn-apsit} we have
$$|\psi^1 - \psi^2 |   \, |\pp_t A(\psi^1)| \leq C\, \frac{|\psi^1-\psi^2|}{z} \,(\, |\pp_t \psi^1| + ( |\dot \xi| + |\dot \eta |)\, |\psi^1| \, )$$
hence, by \eqref{eqn-linfty} applied to $\psi^1-\psi^2$, \eqref{eqn-good16}  
  and the assumed  bounds on $h, \eta$ and $\psi$, we   conclude
$$\| \, |\psi^1 - \psi^2 |   \, |\pp_t A(\psi^1)| \,  \|_{*,\sigma,t_0}^\nu 
\leq C\, |t_0|^{\frac 12 -\nu}\, 
 \|\psi^1 - \psi^2\|_{*,2,\sigma,t_0}^\nu.$$ 
To estimate the last term we set $I(\psi) = \int_0^1 \big(1 + s \tilde z^{-1}\, \psi\big)^{p-2} s \, ds$
so that 
$$ |\psi^2| \, |\pp_t A(\psi^1) - \pp_t A(\psi^2) | \leq \Lambda_1 + \Lambda_2$$
where
$$\Lambda_1 = |\psi^2| \,   \frac{|\tilde z\, \pp_t \psi^1 - \pp_t \tilde{z} \, \psi^1|}{\tilde{z}^2} \,
| I(\psi^1) - I(\psi^2)| $$
and
$$\Lambda_2 =|\psi^2| \,  \frac{ |\tilde z\, \pp_t (\psi^1-\psi^2) - \pp_t \tilde{z}\,  (\psi^1-\psi^2)|}{\tilde z^2}  \, |I(\psi^2)|.$$
The bound $\| \Lambda_2 \|_{*,\sigma,t_0}^\nu  \leq C\, |t_0|^{\frac 12 -\nu}\, 
 \|\psi^1 - \psi^2\|_{*,2,\sigma,t_0}^\nu$ follows by similar arguments  as before. 
 For the other term we have
 $$\| \Lambda_1 \|_{*,\sigma,t_0}^\nu \leq C \,\| \frac{\psi^2}{z} \|_{L^\infty_{t_0} }  \| |\pp_t \psi^1| + ( |\dot \xi| + |\dot \eta |)\, |\psi^1| \|_{*,\sigma,t_0}^\nu \, \| I(\psi^1) - I(\psi^2) \|_{L^\infty_{t_0}}$$
 where $\| \tilde z^{-1} \psi^2 \|_{L^\infty_{t_0} } \leq C\, |t_0|^{\frac 12 - \nu}$ by \eqref{eqn-linfty} and $\| |\pp_t \psi^1| + ( |\dot \xi| + |\dot \eta |)\, |\psi^1| \|_{*,\sigma,t_0}^\nu \leq C$
 by 
 \eqref{eqn-good16} and the assumed bounds on $\psi$, $h$ and $\eta$. On the other hand, applying \eqref{eqn-linfty11} to $\psi^1 - \psi^2$ we obtain 
 $$ \| I(\psi^1) - I(\psi^2) \|_{L^\infty_{t_0}} \leq C\, \|  \frac{\psi^1 - \psi^2}{\tilde z}  \|_{L^\infty_{t_0}}
 \leq  C\, |t_0|^{\frac 12 - \nu} \, \| \psi^1 - \psi^2 \|_{*,2,\sigma,t_0}. $$
Combining the above gives us  the  bound $\| \Lambda_2 \|_{*,\sigma,t_0}^\nu  \leq C\, |t_0|^{\frac 12 -\nu}\, 
 \|\psi^1 - \psi^2\|_{*,2,\sigma,t_0}^\nu$.   All other bounds can be obtained similarly.
\end{proof}

We will now show the  Lipschitz property of the error term $M$ with respect to $h$ and $\eta$.

\begin{lem}\label{lem-Mh2}  There exist $t_0 < 0$ and $C >0$, depending on dimension $n$, such that
for any functions $h,h_i$, $\eta, \eta_i  \in K$, $ i=1,2$,  we have 
\begin{equation}\label{eqn-Mh2}
\| (z^1)^{1-p}  \big (  M(h^1,\eta) - M(h^2,\eta)  \big )\|_{*,\sigma,t_0}^\nu \leq   C\, \, |t_0|^{-\mu}\, 
\| h^1 - h^2 \|_{1,\sigma,t_0}^{\mu,\mu+1}.
\end{equation}
and
\begin{equation}\label{eqn-Meta2}
\| z^{1-p}  \big ( M(h,\eta^1) - M(h,\eta^2)\big ) \|_{*,\sigma,t_0}^\nu \leq   C\, \, |t_0|^{\nu-1}\, 
\| \eta^1 - \eta^2 \|_{1,\sigma,t_0}^1.
\end{equation}
\end{lem}

\begin{proof} The estimates follow by direct (yet tedious)  calculation, along the lines of the proof of Lemma \ref{lem-M}.  Set
 $$z^i(x,t) := \underbrace{w(x+\xi_i(t))}_{w_2^i} + \underbrace{w(x-\xi_i(t))}_{w_1^i}.$$  For the reason of dealing with even functions we restrict ourselves to the region $x \ge 0$, where $w_2^i \le w_1^i \le z^i$. 
 Using the notation of Lemma \ref{lem-M}, we have
\begin{equation*}
\begin{split}
|(z^1)^{1-p} \, &[\, M_1(h^1,\eta) - M_1(h^2,\eta)]\, | \\
&= (z^1)^{1-p} (1+\eta)^p \big [  w_2^1\int_0^1(w_1^1 + sw_2^1)^{p-1}\, ds   - w_2^2\int_0^1 (w^2_1 + sw^2_2)^{p-1}\, ds \big ]\\
& \quad  + (z^1)^{1-p} (1+\eta)^p \big [ -  (w_2^1)^p+(w_2^2)^p \, \big  ] \\
&= (z^1)^{1-p}(1+\eta)^p\big [ (w_2^1-w_2^2)\, \int_0^1(w_1^1 + sw_2^1)^{p-1}\, ds  + ( ( w_2^2)^p-(w_2^1)^p ) ] \\
&\quad +(z^1)^{1-p} (1+\eta)^p \,  w_2^2\int_0^1\big((w_1^1 + sw_2^1)^{p-1} - (w_1^2 + sw_2^2)^{p-1}
\big)\, ds.
\end{split}
\end{equation*}
From the bound 
$$\frac{|w_2^1 - w_2^2|}{w_2^2}  \leq C\, |h_1(t) - h_2(t)|, \qquad i=1,2$$
we conclude the estimate
$$|(z^1)^{1-p} \, [\, M_1(h^1,\eta) - M_1(h^2,\eta)]\, | \leq C\, w_2^2 \, |h_1 - h_2|.$$

\no Having the above estimate, a similar computation as in Lemma \ref{lem-M} implies 
$$\|(z^1)^{1-p}\big(M_1(h^1,\eta) - M_1(h^2,\eta)\big)\|^{\nu}_{*,\sigma,t_0} \le C|t|^{-\mu}\, \|h^1 - h^2\|_{\infty,t_0}^{\mu}.$$
For the $M_2$ term we have
\begin{equation*}
\begin{split}
(z^1)^{1-p}|M_2(h^1,\eta) - M_2(h^2,\eta)| &\le C|\eta|\big(((w_1^1)^p - (w_1^2)^p) + ((w_2^1)^p - (w_2^2)^p)\big) \\
&\le C|\eta||h^1 - h^2| w(x-\xi_1)
\end{split}
\end{equation*}
implying
$$\|(z^1)^{1-p}\big(M_2(h^1,\eta) - M_2(h^2,\eta))\|^{\nu}_{*,\sigma,t_0} \le C|t|^{\nu - 1 -\mu}\|h^1 - h^2\|_{\infty,t_0}^{\mu}.$$
To conclude \eqref{eqn-Mh2} we observe that $\nu <1$ and $\|h^1 - h^2\|_{\infty,t_0}^{\mu}
\leq \| h^1 - h^2 \|_{1,\sigma,t_0}^{\mu,\mu+1}$. 
\smallskip

\no Finally,to show the  Lipschitz property of $M$ in $\eta$, we use the bounds
$$|(z^1)^{1-p}\big(M(h,\eta^1) - M(h,\eta^2)\big)| \le C\,  |\eta^1 - \eta^2|\,  w_1$$
and $|\eta_1 - \eta_2| \leq |t_0|^{-1} \, \|\eta^1 - \eta^2\|^1_{\infty,t_0} \leq |t_0|^{-1} \, \|\eta^1 - \eta^2\|^1_{1,\sigma,t_0}$,  
implying that
$$\|(z^1)^{1-p}\big(M(h,\eta^1) - M(h,\eta^2)\big)\|^{\nu}_{*,\sigma,t_0} \le C|t|^{\nu- 1}\, \|\eta^1 - \eta^2\|^1_{1,\sigma,t_0}.$$  
\end{proof}

We will show the Lipschitz property of the error term 
$$\tilde E(\psi)(h,\eta) :=  (1-\partial_t) N(\psi)  - p\psi \, \pp_t  z^{p-1} - z^{p-1} \big ( c_1(t) \, z +  c_2(t) \, \bar z \big )$$  with respect to $h$ and $\eta$.

\begin{lem}
\label{lem-E3}
There exist $t_0 < 0$ and $C > 0$, depending on dimension $n$, so that for any functions $\psi\in \Lambda$, $h, h_i, \eta, \eta_i \in K$, $i = 1, 2$, we have
\begin{equation}
\label{eq-contr-Epsi}
\|(z^1)^{1-p}\big ( \tilde E(\psi)(h^1,\eta) - \tilde E(\psi)(h^2,\eta) \big ) \|^{\nu}_{*,\sigma,t_0} \le C\, |t_0|^{ - \mu}\, \|h^1 - h^2\|^{\mu,\mu+1}_{1,\sigma,t_0}
\end{equation}
and
\begin{equation}
\label{eq-contr-Epsi2}
\|z^{1-p} \big ( \tilde E(\psi)(h,\eta^1) - \tilde E(\psi)(h,\eta^2) \big )\|^{\nu}_{*,\sigma,t_0} \le C\, |t_0|^{-1}\|\eta^1 - \eta^2\|^1_{1,\sigma,t_0} .
\end{equation}
\end{lem}

\begin{proof}
The estimates again follow from a direct (yet tedious) calculation, along the lines of the proof of Lemma \ref{lem-N}. For example, regarding the term  $N_{13}$,  as in the proof of Lemma \ref{lem-N},
we have
\begin{equation*}
\begin{split}
& |(z^1)^{1-p} N_{13}(\psi)(h^1,\eta) - (z^2)^{1-p} N_{13}(\psi)(h^2,\eta)|  \\
&\le C|\psi|\left|\int_0^1\left(\big (1 + s\frac{\psi}{\tilde{z}^1}\big)^{p-2} \,  \frac{(\psi_t \tilde{z}^1 - \psi(\tilde{z}^1)_t)}{(\tilde{z}^1)^2} - \big (1 + s\frac{\psi}{\tilde{z}^2}\big)^{p-2} \frac{(\psi_t \tilde{z}^2 - \psi(\tilde{z}^2)_t)}{(\tilde{z}^2)^2}\right) s\, ds\right| \\
&\le C\, |\psi| \frac{|\psi_t\tilde{z}^1 - \psi(\tilde{z}^1)_t|}{(\tilde{z}^1)^2}\, \int_0^1 \left |\big(1 + s\frac{\psi}{\tilde{z}^1}\big)^{p-2} - \big(1+s\frac{\psi}{\tilde{z}^2}\big)^{p-2}\right |\,s\,  ds\\
&\quad +C\, |\psi| \, \left| \psi_t \big(\frac{1}{\tilde{z}^2} - \frac{1}{\tilde{z}^1}\big) - \psi\big (\frac{(\tilde{z}^2)_t}{(\tilde{z}^2)^2} - \frac{(\tilde{z}^1)_t}{(\tilde{z}^1)^2} \big)\right|\int_0^1\left(1+s\frac{\psi}{\tilde{z}^2}\right)^{p-2} s\, ds \\
&\le C\, |\psi|\,  |h^1 - h^2| \left(\frac{|\psi_t|}{\tilde{z}^1} + \frac{|\psi|}{\tilde{z}^1}(|\dot{\xi}_1| + |\dot{\xi}_2| + |\dot{\eta}|)\right)
\end{split}
\end{equation*}
\noindent Now it easily follows (similarly as in the proof of Lemma \ref{lem-N}) that
for $\psi \in \Lambda$, we have 
\begin{equation*}
\begin{split}
\|(z^1)^{1-p} &N_{13}(\psi)(h^1,\eta) - (z^2)^{1-p} N_{13}(\psi)(h^2,\eta)\|^{\nu}_{*,\sigma,t_0} \\&\le C\, |t_0|^{\frac 12 - \nu - \mu}\|h^1 - h^2\|_{\infty,t_0}^{\mu}\|\psi\|^{\nu}_{*,\sigma,t_0} \leq C\, |t_0|^{\frac 12 - \nu - \mu}\, \|h^1 - h^2\|^{\mu,\mu+1}_{1,\sigma,t_0}.
\end{split}
\end{equation*}

\no All other terms in $E(\psi)$ can be estimated similarly and the estimate (\ref{eq-contr-Epsi}) follows. 

Let us now  look at the contraction of $E(\psi)$ in $\eta$. For example, we have
\begin{equation*}
\begin{split}
& |z^{1-p}\big(N_{13}(\psi)(h,\eta^1) - N_{13}(\psi)(h, \eta^2)\big)| \\
&\le C\, |\psi|\, \left |  (1+\eta^1)^{p-1}\left ( \frac{\psi_t}{(1+\eta^1)z} - \psi\, \frac{(1+\eta^1)\dot{z} +  \dot{\eta}^1 z}{(1+\eta^1)^2 z^2}\right ) \int_0^1\big  (1+s\frac{\psi}{(1+\eta^1)z}\big)^{p-2} s \, ds \right . \\
&\quad -  \left . (1+\eta^2)^{p-1}\left ( \frac{\psi_t}{(1+\eta^2)z} - \psi \, \frac{(1+\eta^2)\dot{z} + \dot{\eta}_2 z}{(1+\eta^2)^2 z^2}\right ) \int_0^1\big (1+s\frac{\psi}{(1+\eta^2)z}\big )^{p-2} s \, ds\,  \right | \\
&\le  C\, |\psi||\eta^1 - \eta^2|\left(\frac{|\psi_t| + |\psi| + |\psi|(|\dot{\xi}| + |\dot{\eta}_1|)}{z}\right) + C\, \frac{|\psi|}{z}\, (|\dot{\eta}_1 - \dot{\eta}_2| + |\dot{\xi}||\eta^1 - \eta^2|).
\end{split}
\end{equation*}
\no This easily implies (as in the proof of Lemma \ref{lem-N}) the bound
$$\|z^{1-p}\big(N_{13}(\psi)(h,\eta^1) - N_{13}(\psi)(h,\eta^2)\big)\|^{\nu}_{*,\sigma,t_0}  \le C|t_0|^{-\frac{1}{2}-\nu}\|\eta^1 - \eta^2\|^1_{1,\sigma,t_0}\|\psi\|^{\nu}_{*,2,\sigma,t_0}.$$
\no Furthermore
$$|z^{1-p}(N_{12}(\psi)(h,\eta_1) - N_{12}(h,\eta_2))| \le C|\psi_t||\eta^1 - \eta^2|\big(1+\frac{|\psi|}{z}\big)$$
implying that
$$\|z^{1-p}(N_{12}(\psi)(h,\eta^1) - N_{12}(h,\eta^2))\|^{\nu}_{*,\sigma,t_0} \le C|t_0|^{-1}\|\eta^1 - \eta_2\|^1_{1,\sigma,t_0} \|\psi\|^{\nu}_{*,2,\sigma,t_0}.$$ 

\no The  other terms in $E(\psi)$ can be estimated similarly and the estimate (\ref{eq-contr-Epsi2}) follows by recalling that $\|\psi\|^{\nu}_{*,2,\sigma,t_0} \leq 1$
for all $\psi \in \Lambda$.
\end{proof}

\subsection{Proof of Proposition \ref{prop-fp}} 

In this section we   give the proof of Proposition \ref{prop-fp} which claims the existence of a solution to the auxiliary equation (\ref{eqn-fp}) with the desired properties.
We fix $(h,\eta) \in K$ and we show that $A: X \cap \Lambda \to X \cap \Lambda$, given by (\ref{eqn-fp}) defines a contraction and then
use the fixed point Theorem. 

\smallskip

\no {\bf (a)} {\em There exists a universal constant $t_0 <0$ for which  
$A(X \cap \Lambda ) \subset X\cap \Lambda$. } 
\smallskip

\no Indeed, assume that $\psi \in X \cap \Lambda$.  By Proposition \ref{prop-linear}   we have
$$
 \|A(\psi)\|^{\nu}_{*,2,\sigma,t_0} = \|T(\bar E(\psi)\|^{\nu}_{*,2,\sigma,t_0} \le C\, 
 \|\bar E(\psi)\|_{*,\sigma,t_0}^\nu
$$
given that $\bar E(\psi)$ satisfies the orthogonality conditions \eqref{eqn-orth11} and \eqref{eqn-orth22}.
In addition,  it is easy to see that for the second term on the right hand side in 
\eqref{eqn-Epsi3},  we have
$$
\|c_1(t)z + c_2(t)\bar{z}\|^{\nu}_{*,\sigma,t_0} \le C\, \|E(\psi)\|_{*,\sigma,t_0}^{\nu}. 
$$
Hence,
$$
 \|A(\psi)\|^{\nu}_{*,2,\sigma,t_0} = \|T(\bar E(\psi)\|^{\nu}_{*,2,\sigma,t_0} \le C\, 
 \|E(\psi)\|_{*,\sigma,t_0}^\nu.
$$
Combining the last  estimate with \eqref{eqn-bepsi1} and the bound 
$ \|\psi\|^{\nu}_{*2,\sigma,t_0} \le 1$, we get 
$ \|A(\psi)\|^{\nu}_{*,2,\sigma,t_0}  \leq 1$,   
if $|t_0|$ is chosen sufficiently large. 

\smallskip

\no {\bf (b)} {\em There exists a universal constant $t_0 <0$ for which $A: X\cap \Lambda \to 
X\cap \Lambda$ 
defines a contraction map.}

\smallskip
\no For any  $\psi^1, \psi^2 \in X \cap \Lambda$, Proposition \ref{prop-linear} implies the bound 
$$
\|A(\psi^1) - A(\psi^2) \|^{\nu}_{*,2,\sigma,t_0} =  \|T(\bar E(\psi^1) - \bar E(\psi^2))\|^{\nu}_{*,2,\sigma,t_0} \le C\|\bar E(\psi^1) - \bar E(\psi^2)\|^{\nu}_{*,\sigma,t_0}.$$
Similarly as above we have
$$\|\bar E(\psi^1) - \bar E(\psi^2)\|^{\nu}_{*,\sigma,t_0} \leq C \, \|E(\psi^1) - E(\psi^2)\|^{\nu}_{*,\sigma,t_0}.$$
The last two estimates with  \eqref{eqn-bepsi2} yield  the contraction bound 
$$\|A(\psi^1) - A(\psi^2)\|^{\nu}_{*,2,\sigma,t_0} \leq q\,    \|\psi^1 - \psi^2\|_{*,2,\sigma,t_0}^\nu$$
with $q <1$,  provided that $|t_0|$  is chosen sufficiently large.  

\smallskip

The  above discussion and the Fixed Point  Theorem readily imply  the existence of  
a unique fixed point $\psi=\Psi(h,\eta) \in X\cap \Lambda$  of the map $A$. 

\medskip

We will now continue with the proofs of \eqref{eqn-contr1} and \eqref{eqn-contr2}. 

\smallskip 

\no {\bf (c)} {\em There exists  a $t_0 <0$, such that for any $(h^1,\eta),(h^2,\eta) \in K$,  \eqref{eqn-contr1} holds.}

\smallskip
\no Since $\eta$ is fixed we will omit denoting the dependence on $\eta$. For simplicity we set $\psi^1=\Psi(h^1,\eta)$ and $\psi^2=\Psi(h^2,\eta)$. The estimate will be obtained
by applying the estimate \eqref{eqn-basic0}. However, because each $\psi^i$ satisfies the orthogonality conditions \eqref{eqn-orth11} and \eqref{eqn-orth22} with $\xi(t)=\xi^i(t) := \frac 12 \log (2b\,|t|) + h^i(t)$, the difference $\psi^1-\psi^2$ doesn't satisfy 
an exact orthogonality condition. To overcome this technical difficulty we will consider instead the difference
$Y:= \psi^1 - \bar \psi^2$, where  
$$ \bar \psi^2 (x,t)= \psi^2 (x,t) - \lambda_1(t) \, w'(x-\xi^1(t)) - \lambda_2(t) \, w(x-\xi^1(t))$$
with 
$$\la_1(t) = \int \psi^2 (x-\xi^1(t),t)  w'(x)  w^{p-1} d x, \,\,\, 
\la_2(t) = \int  \psi^2 (x-\xi^1(t),t)  w(x)  w^{p-1} d x.$$
Clearly,  $Y$ satisfies the orthogonality conditions \eqref{eqn-orth11} and \eqref{eqn-orth22} with 
$\xi(t)=\xi^1(t)$. Denote by $L_t^1$ the operator 
$$L^1_t Y:= p(z^1)^{p-1} \, \pp_t Y - \big [ \pp_{xx} Y - Y + p(z^1)^{p-1} \, Y  \big ].$$
Since each  of the $\psi^i$ satisfies equation \eqref{eqn-fp}, it follows that 
$Y:= \psi^1 - \bar \psi^2$ satisfies
\begin{equation*}
\begin{split}
L^1_t Y  = & M(h^1) - M(h^2) +  (z^1)^{p-1} \big ( \hat E(\psi^1,h^1)-  \hat E(\psi^2,h^2) )  \\& - L_t^1 (\psi^2 - \bar \psi^2)+ \big ( (z^2)^{p-1} - (z^1)^{p-1} \big) \, (1-\pp_t)\psi^2
\end{split}
\end{equation*}
where for $i=1,2$, we denote by $M^i:=M(h^i)$ and by 
$$\hat  E(\psi^i,h^i) := (z^1)^{1-p} \big [ (1-\partial_t) N(\psi^i)  - p\psi^i \, \pp_t  (z^i)^{p-1} - 
(z^i)^{p-1}\big ( c_1^i(t) \, z^i +  c_2^i(t) \, \bar z^i \big) \big ]
$$
with  $M(h^i)$, $N(\psi^i)$ and $c_1^i, c_2^i$  defined in \eqref{eqn-M}, \eqref{eqn-N} and \eqref{eqn-Epsi3}
respectively.  

\smallskip

We next observe that the estimate \eqref{eqn-basic0} holds for any even solution $Y$  of equation 
$L^1_tY = (z^1)^{p-1} \, f$, as long as the solution  $Y$  itself, and not necessarily $f$,  satisfies the 
orthogonality conditions \eqref{eqn-orth11} and \eqref{eqn-orth22}. Indeed, the a priori estimate 
$$ \| Y \|_{*,2,\sigma,t_0}^\nu \leq C \, \big ( \| Y \|_{*,\sigma,t_0}^\nu +  \| f \|_{*, \sigma,t_0}^\nu \big )$$
holds for any solution $Y$ and the bound $\| Y \|_{*,\sigma,t_0}^\nu \leq C \, \| f \|_{*, \sigma,t_0}^\nu$,
based on the contradiction argument given in Proposition \ref{prop-apriori} can be shown to hold for
any even solution $Y$ that satisfies \eqref{eqn-orth11} and \eqref{eqn-orth22}. Hence, we have 
\be\label{eqn-YY}
 \| Y \|_{*,2,\sigma,t_0}^\nu \leq C \,   \| (z^1)^{1-p}L^1_t Y \|_{*, \sigma,t_0}^\nu.
 \ee

\begin{claim} We have 
$$ \| (z^1)^{1-p}\, L^1_t Y \|_{*, \sigma,t_0}^\nu \leq  C\,  |t_0|^{\frac 12 -\nu}  \| Y \|_{*, 2, \sigma,t_0}^\nu
 + C\, |t_0|^{-\mu} \,   \| h^1 - h^2 \|_{1,\sigma,t_0}^{\mu,\mu+1}.$$
\end{claim}

\begin{proof}[Proof of Claim] By \eqref{eqn-Mh2}, we have 
$$\| (z^1)^{1-p}  ( M(h^1) - M(h^2) ) \|_{*,\sigma,t_0}^\nu \leq   C\, |t_0|^{-\mu}\, 
\| h^1 - h^2 \|_{1,\sigma,t_0}^{\mu,\mu+1}.$$
Also, by combining \eqref{eqn-bepsi2} and \eqref{eq-contr-Epsi} we have
$$ \|   \hat E(\psi^1,h^1)-  \hat E(\psi^2,h^2)   \|_{*,\sigma,t_0}^\nu \leq   
C |t_0|^{\frac 12 -\nu} \, \| \psi^1 - \psi^2 \|_{*, 2, \sigma,t_0}^\nu + C  |t_0|^{-\mu}\, 
\| h^1 - h^2 \|_{1,\sigma,t_0}^{\mu,\mu+1}.$$
Also, since  $\| (1-\pp_t)\psi^2 \|_{*,\sigma,t_0}^\nu \leq 2 \| \psi^2 \|_{*,2,\sigma,t_0}^\nu \leq 2$, 
$$\| \frac{ (z^2)^{p-1} - (z^1)^{p-1} }{(z^1)^{p-1}}  \, (1-\pp_t)\psi^2 \|_{*,\sigma,t_0}^\nu \leq 
C\, |t_0|^{-\mu}\, \| h^1 - h^2 \|_{1,\sigma,t_0}^{\mu,\mu+1}.$$
For the term, $L_t^1(\psi^2 - \bar{\psi}^2)$, we observe that since both $w(x)$ and $w'(x)$ are eigenfunctions of the operator 
$L_0$ given in \eqref{eqn-oper}, and $w(x)$ and all its derivatives are bounded in $\R$, we have  
$$ (z^1)^{1-p} \, |L_t^1 (\psi^2 - \bar \psi^2) | \leq C  \,   
 \sum_{i=1}^2 (|\la_i| + |\dot \la_i| + |\lambda_i| |\dot \xi^1 |).$$
Let us now estimate $|\la_i(t)|$ and $|\dot \la_i(t)|$. Using the orthogonality condition \eqref{eqn-orth11}
satisfied by $\psi^2$ (with $\xi=\xi^2$) we have
$$|\la_1(t)| = \big  | \int_{\R}  \big ( \psi^2(x-\xi^1) - \psi^2(x-\xi^2)\big  ) w'(x)\, w^{p-1} dx  \big | 
\leq C\, |(h^1- h^2)(t) | \, \| \psi^2(\cdot,t) \|_{L^2}.$$
Similarly, one can see that
$$|\dot \la_1(t)| \leq C \, \big ( |(h^1- h^2)(t) | \, \| \pp_t \psi^2(\cdot,t) \|_{L^2} + |(\dot h^1- \dot h^2)(t) | \, \|  \psi^2(\cdot,t) \|_{L^2} \big ).$$
The estimates for $|\la_2(t)| $ and $|\dot \la_2(t)| $ are the same. Combining the last estimates,
readily yields the bound
$$\| (z^1)^{1-p} \, |L_t^1 (\psi^2 - \bar \psi^2) \|_{*,\sigma,t_0}^\nu \leq  C\, |t_0|^{-\mu}\, 
\| h^1 - h^2 \|_{1,\sigma,t_0}^{\mu,\mu+1}.$$

\no To  finish  the proof of the claim,  we need to show that 
 $$\| \psi^1-\psi^2 \|_{*, 2, \sigma,t_0}^\nu \leq  \| Y \|_{*, 2, \sigma,t_0}^\nu + C \, |t_0|^{-\mu} \, \| h^1 - h^2 \|_{1,\sigma,t_0}^{\mu,\mu+1}.$$ 
Since
$\| \psi^1-\psi^2 \|_{*, 2, \sigma,t_0}^\nu \leq \| Y \|_{*, 2, \sigma,t_0}^\nu + 
\| \la_1 w'(x-\xi^1) + \la_2 w(x-\xi^1) \|_{*, 2, \sigma,t_0}^\nu$
this estimate readily follows from the previous bounds on $\la_i$. 

\end{proof}

The proof of   \eqref{eqn-contr1} now readily follows by combining  \eqref{eqn-YY} and the above claim
and choosing $|t_0|$ sufficiently large. 

\smallskip 

\no {\bf (d)} {\em There exists  $t_0 <0$, such that for any $(h,\eta^1),(h,\eta^2) \in K$,  \eqref{eqn-contr2} holds.} 

\smallskip
This proof is an easy consequence of \eqref{eqn-Meta2}, \eqref{eqn-bepsi2}, \eqref{eq-contr-Epsi2}
and \eqref{eqn-basic0}, since 
for $\psi^1=\Psi(h,\eta^1)$ and $\psi^2=\Psi(h,\eta^2)$, we have 
$$\| \psi^1- \psi^2\|_{*,2,\sigma,t_0}^\nu = \| T ( \bar E(\psi^1,\eta^1)) - 
T ( \bar E(\psi^2,\eta^2))\|_{*,2,\sigma,t_0}^\nu$$
where now the operator $T$ depends only on $h$ (not the $\eta_i$) and $h$ is fixed. 


\section{Solving for $\xi$ and $\eta$}\label{sec-system}
\label{sec-xieta}

We recall the definition of $K$ given by \eqref{defn-K}. 
In the previous section we have established that for any given $(h,\eta) \in K$,  there exists a solution $\psi=\Psi(h,\eta)$ of the auxiliary equation \eqref{eqn-fp}. Recall that $c_1(t)$ and $c_2(t)$ are  chosen so that the error term 
$
\bar E(\psi):= E(\psi) - (c_1(t) \, z + c_2(t) \, \bar  z)  
$
satisfies the orthogonality conditions \eqref{eqn-orth11} and \eqref{eqn-orth22} whenever $\psi$ does. 
The error term $E(\psi)$ is given  by \eqref{eq-Epsi}. 
Thus $\psi =\Psi(h,\eta)$ defines a solution to our original equation \eqref{eqn-Psi} 
if we manage to adjust the parameter functions $(h,\eta)$ in such a way that $c_1 \equiv 0$ and $c_2  \equiv 0$.
This is equivalent to choosing $(h,\eta)$ so that 
\begin{equation}\label{eqn-Epsio1}
\inti  E(\psi) \, w^p(x+\xi(t))   \, dx = 0
\end{equation}
and
\be \label{eqn-Epsio2}
\inti  E(\psi)\, w'(x+\xi(t)) \, w^{p-1}(x+\xi(t)) \, dx = 0.
\ee
In fact, the main result in this section is the following.

\begin{prop}\label{prop-etaxi}  There exists   $(h,\eta) \in K$  such that \eqref{eqn-Epsio1} and 
\eqref{eqn-Epsio2}
are satisfied. It follows that  the solution $\psi=\Psi(h,\eta)$ of the auxiliary equation \eqref{eqn-fp} given by Proposition
\ref{prop-fp}    defines  a solution of  our original equation \eqref{eqn-Psi} and Theorem \ref{thm-main} holds.

\end{prop}
 
\subsection{Computation of error Projections}\label{sec-proj}

The proof of Proposition \ref{prop-etaxi} is based on careful  expansions for the 
projections of the error terms (given by the left hand sides of  \eqref{eqn-Epsio1} and 
\eqref{eqn-Epsio2})  which  lead us to a system of ODE for  the functions 
$\xi:=\xi(t)$ and $\eta:=\eta(t)$. We will then solve this system  by employing the fixed point theorem. We will see that the main order terms in the system are all coming from the 
projections of the term $z^{1-p} M$ in \eqref{eq-Epsi}. Let us first expand these projections in terms of $\xi$, $\eta$
and their derivatives.

\begin{lem}[{\bf Projections of the error term $M$}]
\label{lem-projection} We have 
\begin{equation*}
\inti  M \, w^p (x+\xi) \, z^{1-p}\, dx =
  c_1 \,  \big (\dot{\eta} - \frac{p-1}{p}\eta - a \, e^{-2\xi}\big) + \mathcal{R}_1(\xi,\dot{\xi},\eta,\dot{\eta})
\end{equation*}
and
\begin{equation*} \inti  M \, w'(x+\xi) \, w^{p-1}(x+\xi) \,  z^{1-p} dx \\
=  c_2  \, \big( \dot{\xi} + b \, e^{-2\xi}) + \mathcal{R}_2(\xi,\dot{\xi},\eta,\dot{\eta})
\end{equation*}
where $c_1, c_2$ are universal constants and 
\be\label{eqn-ab}
a = \frac{(p-1)\int_0^{\infty} w^p e^x \, dx + p\int_{-\infty}^0 w^p e^x \, dx}{p\int  w^{p+1}\, dx}
\quad \mbox{and}  \quad b = \frac{\int_0^\infty w^p e^{-x}\, dx}{p\int w'^2 w^{p-1}\, dx}.
\ee  
Moreover,
$$\|\mathcal{R}_i(\xi,\dot{\xi},\eta,\dot{\eta})\|_{\sigma,t_0}^{1+\gamma} \le C, \qquad \mbox{for} \,\,\, i\in \{1,2\},$$
and for some $0 < \gamma < 1$, that depends on dimension only.
\end{lem}

\begin{proof}[Proof of Lemma \ref{lem-projection}]
We will use the notation of previous sections. 
Let us write  $M = M_1 + M_2 + M_3$ with 
$$M_1 = (1+\eta)^p\big((w_1+w_2)^p - w_2^p\big), \qquad M_2 = -(1+\eta)^p w_1^p$$
and
$$M_3 =  \underbrace{\big((1+\eta)^p - (1+\eta)\big)(w_1^p + w_2^p)}_{M_{31}} - \underbrace{p(1+\eta)^{p-1}\dot{\eta} z^p}_{M_{32}} - \underbrace{p(1+\eta)^p \, z^{p-1}\, \dot{\xi}\, (\pp_x w_2 - \pp_x w_1)}_{M_{33}}.$$
We  first  compute  the expansion of the term $\int  z^{1-p} M \, w^p_2 \,  dx$. We have
\bee
\bs
\label{eq-first-term-E}
&\inti  M_1   \, \frac{w_2^p}{z^{p-1}}\, dx = p (1+\eta)^p \inti  \big [ w_1  \int_0^1  (w_2+sw_1)^{p-1} \, ds \big ]  \, \frac{w_2^p}{z^{p-1}}\,dx \\
 &= p(1+\eta)^p \left ( \inti w_2^{p-1}w_1 \frac{w_2^p}{z^{p-1}}\, dx   +
(p-1) \inti \big [ w_1^2   \int_0^1 (w_2 + sw_1)^{p-2} (1-s) ds \,\big ] \frac{w_2^p}{z^{p-1}}\, dx \right )
\end{split}
\eee
where we have used the notation $\bar w(x) := w(x) + w(x-2\xi)$. 

\no We  next analyze the terms on the right hand side of  the last equation. For the first term we  have 
\bee
\bs
 \inti w_2^{p-1}w_1 \frac{w_2^p}{z^{p-1}}\, dx  &= \inti w^{p-1} w(x-2\xi)\frac{w^p}{\bar{w}^{p-1}}\, dx  
= \int_{x \le 2\xi} w(x-2\xi) w^p\, dx + g_1(2\xi) \\
 &= e^{-2\xi} \int_{x\le 2\xi}  e^x w^p\, dx + g_1(2\xi) = e^{-2\xi} \inti e^x w^p\ dx + g_1(2\xi)
 \end{split}
 \eee
 where we denote by $g_1(2\xi)$ various error terms having the following decay
 $$|g_1(2\xi)| \le C e^{-(1+\gamma)2\xi},\qquad  \mbox{with} \,\,\,  0 < \gamma < 1.$$
The other  term  turns out to be of a lower order and absorbable in $g_1(2\xi)$. To see that, since $w_2 \le w_1$ 
for $x > 0$ and $w_1 \le w_2$ for $x < 0$, we have
\bee
\bs
\int w_1^2 \frac{w_2^p}{z^{p-1}}   \,  \big [ \int_0^1 &(w_2 + sw_1)^{p-2} (1-s)\, ds \,\big ] dx \leq 
C \int_{x \ge 0}  w_1 w_2^p \, dx +  C\int_{x\le 0} w_1^2 w_2^{p-1}\, dx \\
&\le g_1(2\xi).
\end{split}
 \eee

\smallskip

\no For the term $M_2$ we have
\bee
\bs
 -\inti M_2 &\frac{w_2^p}{z^{p-1}} \, dx = (1+\eta)^p\, \inti  \frac{w^p(x-2\xi)}{\bar{w}^{p-1}} \, w^p\, dx \\
&=
  \int_0^{2\xi} w^p(x-2\xi)\,  w\, dx + g_2(2\xi,\eta)  = \int_{-2\xi}^0 w^p\,  w(y+2\xi)\, dy + g_2(2\xi,\eta) \\
& =
 e^{-2\xi}\int_{-\infty}^0 w^p e^{-y}\, dy + g_2(2\xi,\eta)  = e^{-2\xi} \int_0^{\infty} w^p e^x\, dx + g_2(2\xi,\eta)
\end{split}
 \eee
where we denote by $g_2(2\xi,\eta)$ various error terms having the following behavior,
$$|g_2(2\xi, \eta)| \le Ce^{-2\xi}(e^{-2\xi \gamma} + |\eta|), \qquad  \mbox{with} \,\,\,\gamma > 0.$$
 
\no Furthermore,
 \bee
\bs
&\int M_{31} \frac{w_2^p}{z^{p-1}}  \, dx \\
&= [(1+\eta)^p - (1+\eta)]\, \left(\inti w^{2p} \bar{w}^{1-p}\, dx + \inti w^pw^p(x-2\xi)\bar{w}^{1-p}\, dx\right) \\
&= (p-1)\eta\inti w^{2p} \bar{w}^{1-p}\, dx + g_3(2\xi,\eta) = (p-1)\eta\inti w^{p+1}\, dx + g_3(2\xi,\eta)
\end{split}
 \eee
where we denote by $g_3(2\xi,\eta)$ various error terms with the behavior 
$$|g_3(2\xi,\eta)| \le C \big(e^{-2\xi}(e^{-2\xi \gamma} + |\eta|) + |\eta|^2\big), \qquad  \mbox{with} \,\,\,\gamma > 0.
$$
Also, 
$$
\inti M_{32}\,  \frac{w_2^p}{z^{p-1}}\, dx =
p(1+\eta)^{p-1}\, \dot{\eta}\, \inti  z\, w^p_2\, dx = p\, \dot{\eta}\, \inti w^{p+1}\, dx + g_4(2\xi,\dot{\eta})
$$
with  $|g_4(2\xi,\dot{\eta})| \le C e^{-2\xi}|\dot{\eta}|$,  and 
\bee
\bs
\inti M_{33} \, \frac{w_2^p}{z^{p-1}}\, dx &=
p\, (1+\eta)^p\, \dot{\xi}\,  \inti (\pp_xw_2 - \pp_xw_1) \, w_2^p\, dx  \\
&= p\, \dot{\xi}\inti  w'(x-2\xi) \, w^p\, dx + g_5(2\xi,\dot{\xi},\eta) 
= g_5(2\xi, \dot{\xi},\eta)
\end{split}
\eee
with  $|g_5(2\xi,\dot{\xi},\eta)| \le C(e^{-2\xi}(|\eta| + |\dot{\xi}|) + |\dot{\xi}| |\eta|)$.

\smallskip
\no Combining the previous estimates for $M_1, M_2, M_3$ we obtain
\begin{equation} \label{eq-leading-E}
\bs
\inti M \frac{ w_2^p}{z^{p-1}}  \, dx &=  
-p\,  \big (\dot{\eta} - \frac{p-1}{p}\eta \big )\inti w^{p+1}\, dx  \\
&  \quad +  e^{-2\xi} \left((p-1)\int_0^{\infty} e^x w^p\, dx + p\int_{-\infty}^0 e^x w^p\, dx\right)\\
& \quad + \mathcal{R}_1(\xi,\dot{\xi},\eta,\dot{\eta})
\end{split}
\ee
where 
\begin{equation}
\label{eq-error-R}
|\mathcal{R}_1(\xi,\dot{\xi},\eta,\dot{\eta})| \le C\big ( \, e^{-2\xi}\, ( \,  e^{-2\gamma\xi}+ |\eta| + |\dot{\xi}| + |\dot{\eta}|  \, ) + |\eta|\, (\, |\dot{\xi}| + |\eta| + |\dot{\eta}| \, ) \, \big).
\end{equation}

\medskip
Let  us now expand $\int M \, z^{1-p} \pp_xw_2 \, w_2^{p-1}\, dx$. Similarly as before, analyzing term by term  we obtain
\bee
\bs
\inti  M_1  \pp_x w_2 \, \frac{ w_2^{p-1}}{z^{p-1}} \, dx &= 
p\, (1+\eta)^p   \inti w_2^{p-1} w_1 \pp_x w_2\,  \frac{w_2^{p-1}}{z^{p-1}} \,  dx + \mathcal{R}_2(\xi,\dot{\xi},\eta,\dot{\eta})\\
&= p\, \int_{x\le 2\xi} w^{p-1} w' \,  w(x-2\xi) \, dx + \mathcal{R}_2(\xi,\dot{\xi},\eta,\dot{\eta}) \\
&= p\, e^{-2\xi}\, \inti  w^{p-1} w' e^x\, dx + \mathcal{R}_2(\xi,\dot{\xi},\eta,\dot{\eta}) \\
&= -e^{-2\xi}\, \inti  w^p e^x\, dx + \mathcal{R}_2(\xi,\dot{\xi},\eta,\dot{\eta})
\end{split}
\eee
where by $\mathcal{R}_2(\xi,\dot{\xi},\eta,\dot{\eta})$ we have denoted various error terms that have the same behavior as in (\ref{eq-error-R}). Also, 
\bee
\bs
-\inti M_2 \pp_x w_2 \frac{ w_2^{p-1}}{z^{p-1}} \, dx 
&= (1+\eta)^p \int \frac{w^p(x-2\xi)}{\bar{w}^{p-1}(x)} w' w^{p-1}\, dx  \\
&= \int_0^{2\xi} w^p(x-2\xi) w'\, dx + \mathcal{R}_2(\xi,\dot{\xi},\eta,\dot{\eta}) \\
&= \int_{-2\xi}^0 w^p w'(y+2\xi)\, dy + \mathcal{R}_2(\xi,\dot{\xi},\eta,\dot{\eta}) \\
&= -e^{-2\xi}\int_0^\infty  w^p e^x\, dx + \mathcal{R}_2(\xi,\dot{\xi},\eta,\dot{\eta}), 
\end{split}
\eee
\no where $\mathcal{R}_2(\xi,\dot{\xi},\eta,\dot{\eta})$ is the error term satisfying (\ref{eq-error-R}). 

\no Next, using that 
$\int w' w^p\, dx = 0$ and that $\big (1- \big(\frac{w}{\bar{w}}\big)^{p-1}\big ) \le C(p) \big (1-\frac{w}{\bar{w}}\big )$, for $x < \xi$, 
and $1- \big (\frac{w}{\bar{w}}\big)^{p-1} \le 1$ otherwise, we obtain
\bee
\bs
\inti M_{31} \pp_x w_2 \frac {w_2^{p-1}}{z^{p-1}} \, dx &= (p-1)\, \eta  \inti w^p  w' \, \frac{w^{p-1}}{\bar{w}^{p-1}}   \, dx + \mathcal{R}_2(\xi,\dot{\xi},\eta,\dot{\eta}) \\
&= (p-1)\, \eta \inti \left(\big(\frac{w}{\bar{w}}\big)^{p-1} - 1\right) w' w^p \, dx  + \mathcal{R}_2(\xi,\dot{\xi},\eta,\dot{\eta}), \\
&= \mathcal{R}_2(\xi,\dot{\xi},\eta,\dot{\eta})
\end{split}
\eee
\no where $\mathcal{R}_2(\xi,\dot{\xi},\eta,\dot{\eta})$ satisfies (\ref{eq-error-R}). Using again that 
$\int w' w^p\, dx = 0$, we obtain 
\bee
\bs
\inti  M_{32}  \pp_xw_2 \frac{w_2^{p-1}}{ z^{p-1}} \, dx &= p\, (1+\eta)^{p-1}\, \dot{\eta}\inti  \bar{w} \, w' \, w^{p-1}\, dx \\
&= p\, \dot{\eta}\inti w(x-2\xi) \, w' w^{p-1}\, dx = \mathcal{R}_2(\xi,\dot{\xi},\eta,\dot{\eta}).
\end{split}
\eee
\no Finally,
\bee
\bs
\int M_{33} \, \pp_x w_2 \, \frac{w_2^{p-1}}{ z^{p-1}}\, dx  &= p(1+\eta)^{p-1}\, \dot{\xi}\int (w'(x) - w'(x-2\xi))\, w' \,  
w^{p-1}\, dx \\
&= p\, \dot{\xi} \inti (w')^2 w^{p-1}\, dx + \mathcal{R}_2(\xi,\dot{\xi},\eta,\dot{\eta}).
\end{split}
\eee
\no Combining the above estimates we conclude the bound 
\be\label{eq-form-c200}
\bs
\inti  M \,  \pp_x w_2 \, \frac{w_2^{p-1}}{ z^{p-1}} \, dx  &=  - p\, \dot{\xi}\, \inti (w')^2 \, w^{p-1}\, dx \\
&\quad  -e^{-2\xi} \int_0^\infty  w^p \, e^{-x}\,  dx + \mathcal{R}_2(\xi,\dot{\xi},\eta,\dot{\eta}).
\end{split}
\ee

\no Combining (\ref{eq-leading-E}), (\ref{eq-form-c200}) and (\ref{eq-error-R}) finishes the proof of the Lemma.
\end{proof}

As an  immediate Corollary of the previous lemma we obtain: 

\begin{corollary}
\label{cor-system} Set $Q(\psi) := E(\psi) - Mz^{1-p}$. 
With the same  notation as in  Lemma \ref{lem-projection}, equations \eqref{eqn-Epsio1} and \eqref{eqn-Epsio2}
are equivalent to the system 
\begin{equation}
\label{eq-eta}
\dot{\eta} - \frac{p-1}{p}\eta - a\,  e^{-2\xi} = \mathcal{R}_1(\xi,\dot{\xi},\eta,\dot{\eta}) + G_1(\psi,\xi,\eta)
\end{equation}
and
\begin{equation}
\label{eq-xi}
\dot{\xi} + b \, e^{-2\xi} = \mathcal{R}_2(\xi,\dot{\xi},\eta,\dot{\eta}) + G_2(\psi,\xi,\eta)
\end{equation}
where
$$G_1(\psi,\xi,\eta) := c_1^{-1}  \inti Q(\psi) \, w^p(x+\xi) \, dx$$
and 
$$G_2(\psi,\xi,\eta) := c_2^{-1}  \inti  Q(\psi) \, w'(x+\xi) \, w^{p-1} (x+\xi) \, dx.$$
The error terms $\mathcal{R}_i(\xi,\dot{\xi},\eta,\dot{\eta})$ satisfy  
$$\|\mathcal{R}_i(\xi,\dot{\xi},\eta,\dot{\eta})\|_{\sigma,t_0}^{1+\gamma} \le C, \qquad \mbox{for} \,\,\, i\in \{1,2\}.$$
\end{corollary}

\begin{remark}
If we look at the proof of Lemma \ref{lem-projection} we can trace all the error terms we have denoted by $\mathcal{R}_i(\xi,\dot{\xi},\eta,\dot{\eta})$. 
\no Observe that 
\be
\label{eq-ppR}
\bs
\big|\pp_{\eta}\mathcal{R}_i\big| + \big|\pp_{\xi}\mathcal{R}_i\big| &\le Ce^{-2\xi}(e^{-2\gamma\xi} + |\dot{\xi}| + |\eta| + |\dot{\eta}| + 1)\\
\big|\pp_{\dot{\eta}}\mathcal{R}_i\big| + \big|\pp_{\dot{\xi}}\mathcal{R}_i\big| &\le C(e^{-2\xi} + |\eta|)
\end{split}
\ee
\end{remark}

\medskip

Our strategy in solving the system \eqref{eq-eta}-\eqref{eq-xi}  is as follows:  For any given  
$\xi := \frac 12 \log (2b \,|t|) + h$ with $\| h \|_{1,\sigma,t_0}^{\mu,\mu+1} \leq 1$,  we will first  find a solution 
$\eta(\xi)$ to (\ref{eq-eta}) by the Fixed Point Theorem. The existence of $(\xi,\eta)$ will be 
given by plugging    $\eta(\xi)$  in (\ref{eq-xi}) and applying  the Fixed Point Theorem once more. 

\subsection{Solving for $\eta$}
\label{sec-eta}

In this section we will fix a function $h$ on $(-\infty,t_0]$ with $\| h \|_{1,\sigma,t_0}^{\mu,\mu+1} \leq 1$ and  solve the equation
\begin{equation}
\label{eq-eta-100}
\dot{\eta} - \frac{p-1}{p}\eta = F_h(\eta,\dot{\eta},\psi),
\end{equation}
with 
\begin{equation}
\label{eqn-Fhh}
F_h(\eta,\dot{\eta},\psi) := a e^{-2\xi} + \mathcal{R}_1(\xi,\dot{\xi},\eta,\dot{\eta}) + G_1(\psi,h,\eta)
\ee
 and $\mathcal{R}_1(\xi,\dot{\xi},\eta,\dot{\eta})$, $G_1(\psi,h,\eta)$ are as in Corollary \ref{cor-system}. Recall that
 for any given $(h,\eta) \in K$,  $\psi=\Psi(h,\eta)$ is the solution of \eqref{eqn-fp}, which was proved in Proposition \ref{prop-fp}. 
 For simplicity we will denote, most of the time,  those functions as $F_h$, $R_1$ and $G_1$ respectively. 

\medskip

Let  $\la = \frac {p-1}p$. A function  $\eta$ is a solution of equation \eqref{eq-eta-100} on $(-\infty,t_0]$ 
if 
\begin{equation}\label{eqn-AAA}
A(\eta)(t) := -\int_t^{t_0} e^{\la (t-s)} F_h(\eta(s), \dot{\eta}(s),\psi(s) )\, ds
\end{equation}
satisfies $A(\eta)=\eta$. We have the following result. 

\begin{prop}
\label{prop-eta-contr}
For any  fixed $h\in K$   there is  an $\eta = \eta(h) \in K$ so that $A(\eta) = \eta$. Moreover, for any $h_1, h_2 \in K$, we have  
\begin{equation}\label{eqn-etah2}
\|\eta(h^1) - \eta(h^2)\|^1_{1,\sigma,t_0} \le C\, |t_0|^{-\delta}\, \|h^1 - h^2\|^{\mu,\mu+1}_{1,\sigma,t_0}
\end{equation}
where $t_0<0$ and $C$ are  universal constants  and $\delta>0$ is a small constant depending on $\mu$ and $\nu$.

\end{prop}

\begin{proof} Let $A$ be the operator defined by \eqref{eqn-AAA}.

\no {\bf (a)} {\it There exist a universal  constant   $t_0 < 0$, so that $K$ is invariant under $A$,
namely  $A(K)  \subset K$.} 

We will first show that for $\sigma=n+2$ and $|t_0|$ sufficiently large, we have 
\begin{equation}
\label{eq-Ainfty}
\sup_{\tau \leq  t_0} |\tau|\,  |A(\eta)| \le C\|F_h\|^1_{2,t_0}  \leq C \,  \|F_h\|^1_{\sigma,t_0}.
\end{equation}
\no Indeed, if $t_0 < -1$ then for $\tau < t_0$, 
\bee
\bs
|\tau||\, A(\eta)| &\le C\, e^{\la \tau}\, |\tau| \,\sum_{j=0}^{[t_0-\tau]}\int_{\tau+j}^{\tau+j + 1} e^{- \la s}\,  |F_h(s)|\, ds \\
&\le C\, \|F_h\|^1_{2,t_0} \, e^{\la \tau}\, |\tau|\, \sum_{j=0}^{[t_0-\tau]}\frac{e^{-\la  (\tau+j)}}{|\tau+j |} \\
&\le C\, \|F_h\|^1_{2,t_0} \, e^{\la \tau}\, |\tau|\, \int_\tau^{t_0} \frac{e^{- \la s}}{|s|}  \, ds \leq  C\, \|F_h\|^1_{2,t_0}
\end{split}
\eee
since
\bee
e^{\la \tau}\, |\tau|\,  \int_t^{t_0} \frac{e^{- \la s}}{|s|} ds \leq C
\eee
for a uniform constant $C$. Denoting for simplicity by $I_\tau=[\tau,\tau+1]$ and  using (\ref{eq-Ainfty}), observe that 
\be
\label{eq-const1011}
\|A(\eta)\|^1_{1,\sigma,t_0} = \sup_{\tau\le t_0-1} \ |\tau|\, \|A(\eta) \|_{L^{\sigma}(I_{\tau})} + \sup_{\tau\le t_0-1} |\tau|\, \|\frac{d}{dt}A(\eta) \|_{L^{\sigma}(I_{\tau})} \leq   C\, \|F_h\|^1_{\sigma,t_0},
\ee
where $C$ is a uniform constant.
Next  we want to use the previous estimate to show that $A(K) \subset K$, for an appropriately chosen constant $C_0$.  
By  \eqref{eqn-Fhh} we have  
$$\|F_h\|^1_{\sigma,t_0} := \sup_{\tau \leq t_0-1} |\tau |  \| F_h \|_{L^{\sigma}(I_\tau)} 
 \le \sup_{\tau\le t_0-1} |\tau| \big (a\,  \|  e^{-2\xi} \|_{L^{\sigma}(I_\tau)}  + \| \mathcal{R}_1 \|_{L^{\sigma}(I_\tau)}  + \|  G_1 \|_{L^{\sigma}(I_\tau)} \big ).$$
Since  $\xi = \frac 12 \log (2b\, |t|) + h$  with $ \|h\|^{\mu,1+\mu}_{1,\sigma,t_0} \le 1$, we have   
$|\tau|^\mu \,  \| h \|_{L^\infty(I_\tau)} \leq 1.$ 
Hence,  for $\tau < t_0-1$, 
$$a \, |\tau|\|  e^{-2\xi} \|_{L^{\sigma}(I_\tau)} \le \frac {a }{2b}  \|  e^{-2h} \|_{L^\infty(I_\tau)} \leq \frac{a}{b}.$$
provided that $|t_0|$ is chosen sufficiently large so that 
$\|  e^{-2h} \|_{L^\infty(I_\tau)} \leq 2$. 
Set
\be \label{eqn-defnco}
C_0:= \frac{2a}{Cb},
\ee
where $C$ is the same constant as in \eqref{eq-const1011} and constants $a$ and $b$ are defined in \eqref{eqn-ab}. This implies  $C_0$ is a universal constant
depending only on the dimension $n$. 
We claim that
\be \label{eqn-Fh}
\|F_h\|^1_{\sigma,t_0} \le \frac{C_0}{2} + C\, |t_0|^{-\delta}, \,\,\, \delta > 0.
\ee
To show this claim, we recall that  by   Corollary \ref{cor-system} we have
$$\sup_{\tau\le t_0-1} |\tau | \,  \| \mathcal{R}_1 \|_{L^{\sigma}(I_\tau)} 
\le C\, |t_0|^{-\gamma}$$
\no so we only need to  show that 
\bee
\sup_{\tau\le t_0-1} |\tau| \|  G_1 \|_{L^{\sigma}(I_\tau)}  \leq  C \, |t_0|^{-\delta},
\eee
for some $\delta > 0$.
To this end, we recall that 
$$
\|  G_1 \|_{L^{\sigma}(I_\tau)}  = \left ( \int_\tau^{\tau+1}
\left (\inti Q(\psi)\, w_2^p\, dx \right  )^{\sigma}\, dx \right )^{\frac{1}{\sigma}} 
$$
where
$Q(\psi) := E(\psi) - z^{1-p} \, M$ is given in \eqref{eq-Epsi}. To establish the above bound, we estimate term by term
similarly as in the proof of Lemma \ref{lem-N}. For example,  for the term  $z^{1-p} \, N_{13}$
which is  given in \eqref{eqn-N7} and satisfies the estimate \eqref{eqn-N17}, if we also use that $w_2 \leq z$,  we have
\bee
\bs
 \left(\int_{\tau}^{\tau+1} \left  (\inti N_{13} \frac{w_2^p}{z^{p-1}} \,  dx  \right)^{\sigma}\, dt\right)^{1/\sigma} \le 
 C\left(\int_{\tau}^{\tau+1}\left(\inti 
|\psi|\big( |\psi_t| + |\dot{\xi}|+|\dot{\eta}|\, |\psi| \big)\, w_2^{p-1} \, dx\right)^{\sigma}\, dt\right)^{1/\sigma} 
\end{split}
\eee

\no Recalling that $2\beta=p-1$ and using the bounds $\| \psi \|_{L^\infty (\R \times [\tau,\tau+1])}  \leq |\tau|^{-\nu} \| \psi \|_{*,2,\sigma,t_0}^\nu \leq 1$, $w_2 \leq z$ 
together with  H\"older's inequality we obtain  
\bee
 \left(\int_{\tau}^{\tau+1}  \left  (\inti N_{13} \frac{w_2^p}{z^{p-1}} \,  dx  \right)^{\sigma}\, dt\right)^{1/\sigma} \le  
 C\left(\int_{\tau}^{\tau+1}  \left
 (\int_{|x|\ge\xi} |\psi|\, |\psi_t| \, z^{\frac{n\beta-\sigma}{\sigma}} \, w_2^{p-\frac{n\beta}{\sigma}}\, dx\right)^{\sigma}\, dt\right)^{1/\sigma} \qquad \qquad \qquad 
\eee
\bee
\quad +  \, C\left(\int_{\tau}^{\tau+1}  \left  (\int_{|x|\le\xi} |\psi|\, |\psi_t| \, z^{2\beta+\theta} \, w_2^{-\theta}\, dx\right )^{\sigma} \, dt  \right)^{1/\sigma} +  C\, |\tau|^{-1-\nu} 
\eee
\bee \le C\left(\int_{\tau}^{\tau+1}  \inti |\psi_t|^{\sigma}\alpha_{\sigma}\, dx\, dt\right)^{1/\sigma}(|\tau|^{-\nu} + |\tau|^{-\nu+\theta}) + C\, |\tau|^{-1-\nu} 
\eee
\bee \le C\, |\tau|^{-2\nu + \theta}  \le C\, |\tau|^{-1-\delta}
\eee
for some $\delta > 0$ and sufficiently small $\theta$.  Similar bounds hold for all the other terms which readily give  the bound for $\|  G_1 \|_{L^{\sigma}(I_\tau)}$.  

The above discussion establishes the bound \eqref{eqn-Fh}. Using this bound, we finally obtain 
$$\|A(\eta)\|^1_{1,\sigma,t_0}  \le C\|F_h\|^1_{\sigma,t_0} \le \frac{C_0}{2} + C|t_0|^{-\delta} \le \frac{2}{3}\, C_0 < C_0$$
provided that  $|t_0|$ is sufficiently large. We conclude that   $A(K) \subset K$,  where
$C_0$ is the universal constant on the right hand side of \eqref{eqn-Fh}, finishing the proof of  {\bf (a)}.

\smallskip

\no {\bf (b)} {\it There exists a    universal constant   $t_0 < 0$,  for which $A: K \to K$ defines a contraction map.}

Since $h$ is fixed,  we only  write in $F_h, \mathcal{R}_1$ and $G_1$
their dependence on $\psi$ and  $\eta$. As in part {\bf (a)}, for every $\eta^1, \eta^2 \in K$, if $\psi^i=\Psi(h,\eta^i)$, we have
\be\label{eq-contr-A}
\bs
&\|A(\eta^1) - A(\eta^2)\|^1_{1,\sigma,t_0} \le C\, \|F_h(\eta^1,\dot{\eta}^1,\psi^1) - 
F_h(\eta^2,\dot{\eta}^2,\psi^2)\|^1_{L^{\sigma}_{t_0}} \\
&\le C\, \big(\|\mathcal{R}_1(\eta^1,\dot{\eta^1}) - \mathcal{R}_1(\eta^2,\dot{\eta^2})\|^1_{L^{\sigma}_{t_0}} + \|G_1(\eta^1,\psi^1) - G_1(\eta^2,\psi^2)\|^1_{L^{\sigma}_{t_0}} \big ) 
\end{split}
\ee
Observe first, using  \eqref{eq-ppR} that
\bee
\bs
 |\mathcal{R}_1(\eta^1&,\dot{\eta}^1)  - \mathcal{R}_1(\eta^2,\dot{\eta}^2)| \le
C\, \left|\int_{\eta^1}^{\eta^2} \pp_{\eta} \mathcal{R}_1(\eta,\dot{\eta}^1)\, d\eta\right| + \left |\int_{\dot{\eta}^1}^{\dot{\eta}^2}\pp_{\dot{\eta}}\mathcal{R}(\eta^2,\dot{\eta})\, d{\dot{\eta}}\right| \\
&\le C\left(\int_{\eta^1}^{\eta^2}  e^{-2\xi} \, (e^{-2\gamma \xi} +  |\dot{\xi}| + |\eta| + |\dot{\eta}| + 1)\, d\eta + \int_{\dot{\eta}^1}^{\dot{\eta}^2}(e^{-2\xi} + |\eta^2|)\, d\dot{\eta}\right)
\end{split}
\eee
implying that
\be
\label{eq-est-R}
\|\mathcal{R}_1(\eta^1,\dot{\eta^1}) - \mathcal{R}_1(\eta^2,\dot{\eta^2})\|^1_{\sigma,t_0} \le    \frac{C}{|t_0|} \, \|\eta^1 - \eta^2\|^1_{1,\sigma,t_0}.
\end{equation}
Furthermore, we claim,
$$\|G_1(\eta^1,\psi^1) - G_1(\eta^2,\psi^2)\|^1_{\sigma,t_0} \leq \|G_1(\eta^1,\psi^1) - G_1(\eta^2,\psi^1)\|^1_{\sigma,t_0}  + \|G_1(\eta^2,\psi^1) - G_1(\eta^2,\psi^2)\|^1_{\sigma,t_0} $$
where
\be
\label{eq-G1-est}
 \|G_1(\eta^1,\psi^1) - G_1(\eta^2,\psi^1)\|^1_{\sigma,t_0}  
\le C\, |t_0|^{-2\nu+\theta}\, \|\eta^1 - \eta^2\|^1_{1,\sigma,t_0}
\ee
and
\be
\label{eq-G1-est1}
\|G_1(\eta^2,\psi^1) - G_1(\eta^2,\psi^2)\|^1_{\sigma,t_0}   \le C\,|t_0|^{1-2\nu + \theta} \| \psi^1 - \psi^2 \|_{*,2,\sigma,t_0}^\nu  \leq C \, |t_0|^{-2\nu + \theta} \,  \|\eta^1 - \eta^2\|^1_{1,\sigma,t_0}. 
\ee

\no To establish \eqref{eq-G1-est}, as in part {\bf (a)} let us  look at the term coming from $z^{1-p} \, N_{13}(\psi)$. Using the estimate of this term 
obtained in the proof of Lemma \ref{lem-E3} and a similar analysis as  in part {\bf (a)},  we obtain  
$$\sup_{\tau\le t_0-1}|\tau| \left(\int^{\tau+1}_{\tau}\left (\int \big (N_{13}(\psi^1)(\eta^1) - N_{13}(\psi^1)(\eta^2) \big ) \frac{w_2^p}{z^{p-1}} \, dx\right)^{\sigma}\, dt\right)^{\frac{1}{\sigma}} \le C\, |t_0|^{-2\nu+\theta} \, \|\eta^1 - \eta^2\|^1_{1,\sigma,t_0},$$
where we used that $\|\psi^1 \|_{*,2,\sigma,t_0}^\nu \leq 1$ and that $(h, \eta_i) \in K$. 
All other terms in (\ref{eq-G1-est}) can be estimated similarly so the estimate (\ref{eq-G1-est}) follows. 
To establish \eqref{eq-G1-est1} one argues similarly as above using the  established bounds in the proof of Lemma \ref{lem-E2}
and in \eqref{eqn-contr2}.
Recall that $\nu >1/2$ and $\theta >0$ is small.  Then combining (\ref{eq-contr-A})- (\ref{eq-G1-est1}) and taking $|t_0|$ sufficiently large  yields to
the contraction bound 
\be\label{contr-eta5}
\|A(\eta^1) - A(\eta^2)\|^1_{1,\sigma,t_0} \le \frac{1}{2}\, \|\eta^1 - \eta^2\|^1_{1,\sigma,t_0}.
\ee
This finishes the proof of part {\bf (b)}. 

\smallskip
Having {\bf (a)} and {\bf (b)}, we may apply the  Fixed Point Theorem to the operator $A : K \to K$ to conclude the existence of an $\eta = \eta(h) \in K$ so that $A(\eta) = \eta$.

%
%

\medskip

\no {\bf (c)}{ \it For any $h_1, h_2 \in K$, \eqref{eqn-etah2} holds}.

\smallskip

Since $A(\eta)=\eta$, we have   $\eta(h) = \int_t^{t_0} e^{\lambda(t-s)} F_h(s)\, ds$. Hence, similarly as in
the proof of \eqref{eq-Ainfty}, 
$$\sup_{\tau\le t_0-1} |\tau| |\eta(h^1)-\eta(h^2)| \le C\, \|F_{h_1} - F_{h_2}\|^1_{2,t_0} \leq 
C\, \|F_{h_1} - F_{h_2}\|^1_{\sigma,t_0}$$
which  yields 
\begin{equation}
\label{eq-contr-eta-h}
\|\eta(h^1) - \eta(h^2)\|^1_{1,\sigma,t_0} \le   C\, \|F_{h_1} - F_{h_2}\|^1_{\sigma,t_0}.
\end{equation}
Recall that $F_h = a e^{-2\xi} + \mathcal{R}_1(\xi,\dot{\xi},\eta,\dot{\eta}) + G_1(\psi,h,\eta)$ and  $\xi=\xi_0+h$ with 
$\xi_0(t)=\frac 12 \log (2b\, |t|)$ and $|h|$ small. Applying that to $h^1$ and $h^2$ we get
$$\|e^{-2(\xi_0 + h^1)} - e^{-2(\xi_0 + h^2)}\|^1_{\sigma,t_0} \le C\, \sup_{\tau\le t_0-1}\left (\int^{\tau+1}_{\tau} |h^1-h^2|^\sigma\, dt\right)^{1/\sigma} \le |t_0|^{-\mu}\, \|h^1 - h^2\|^{\mu}_{\sigma,t_0}.$$
Next, using (\ref{eq-ppR}) and the notation $\eta^i=\eta(h^i)=\eta(\xi^i)$, where $\xi^i=\frac 12 \log (2b\, |t|) + h^i$, 
 we have
\bee\label{eq-contr-R1-h}
\bs
& \qquad \qquad \qquad \qquad \qquad\|\mathcal{R}_1(\xi^1,\dot{\xi}^1,\eta^1,\dot{\eta}^1) - \mathcal{R}_1(\xi^2,\dot{\xi}^2,\eta^2,\dot{\eta}^2)\|^1_{\sigma,t_0} \\
&\quad \qquad \qquad \le C\, \big\| \int_{\xi_1}^{\xi^2} \big |\frac{\partial\mathcal{R}_1}{\partial\xi}(\xi,\dot{\xi}^1,\eta^1,\dot{\eta}^1) \big |\, d\xi + \int_{\dot{\xi}^1}^{\dot{\xi}^2}\big |\frac{\partial\mathcal{R}_1}{\partial\dot{\xi}}(\xi^2,\dot{\xi},\eta^1,\dot{\eta}^1)\big  |\, d\dot{\xi}\\
& \quad \qquad \qquad \quad + \int_{\eta^1}^{\eta^2}\big  | \frac{\partial\mathcal{R}_1}{\partial\eta}(\xi^2,\dot{\xi}^2,\eta,\dot{\eta}^1)\big |\, d\eta + \int_{\dot{\eta}^1}^{\dot{\eta}^2}\big |\frac{\partial\mathcal{R}_1}{\partial\dot{\eta}}(\xi^2,\dot{\xi}^2,\eta^2,\dot{\eta})\big |\, d\dot{\eta}\, \big \|_{\sigma,t_0}^1 \\
&\le C\, \big (|t_0|^{-\mu}\|h^1-h^2\|^{\mu}_{\sigma,t_0} + |t_0|^{-1}\|\dot{h}^1-\dot{h}^2\|^{\mu+1}_{\sigma,t_0} + |t_0|^{-1}\|\eta^1-\eta^2\|^1_{\sigma,t_0} + |t_0|^{-1}\|\dot{\eta}^1 - \dot{\eta}^2\|^1_{\sigma,t_0}\big )  \\
&\qquad \qquad \qquad \qquad \qquad \le C\, |t_0|^{-\delta}\big(\|h^1-h^2\|^{\mu,1+\mu}_{1,\sigma,t_0} + \|\eta(h^1)-\eta(h^2)\|^1_{1,\sigma,t_0}\big)
\end{split}
\eee
if  $\delta <  \min\{1, \mu\}$. Finally, similarly to the estimates obtained in the proof of Proposition \ref{prop-eta-contr} and Lemma \ref{lem-E3}, if we denote for simplicity  $\psi^i = \psi(h^i,\eta^i)$
with $\eta^i=\eta(h^i)$, using (\ref{eq-G1-est}) and (\ref{eq-G1-est1}) we have
\bee
\label{eq-G1-contr-h}
\bs
\|G_1(\psi^1,&h^1,\eta^1) - G_1(\psi^2,h^2,\eta^2)\|^1_{\sigma,t_0} \le  \|G_1(\psi^1,h^1,\eta^1) - G_1(\psi^1,h^2,\eta^1)\|^1_{\sigma,t_0} \\&+ \|G_1(\psi^1,h^2,\eta^1) - G_1(\psi^1,h^2,\eta^2)\|^1_{\sigma,t_0}
+  \|G_1(\psi^1,h^2,\eta^2) - G_1(\psi^2,h^2,\eta^2)\|^1_{\sigma,t_0} \\
&\le C  \left(|t_0|^{1-2\nu-\mu+\theta}\|h^1-h^2\|^{\mu,1+\mu}_{1,\sigma,t_0} + |t_0|^{-2\nu+\theta}\|\eta^1-\eta^2\|^1_{1,\sigma,t_0} + |t_0|^{1-2\nu+\theta}\|\psi^1-\psi^2\|^{\nu}_{*,2,\sigma,t_0}\right)  \\
&\le C\big(|t_0|^{1-2\nu-\mu+\theta}\|h^1-h^2\|^{\mu,1+\mu}_{1,\sigma,t_0} + |t_0|^{-2\nu+\theta}\|\eta^1-\eta^2\|^1_{1,\sigma,t_0} \\
&+ |t_0|^{1-2\nu+\theta}\big(\|\psi(h^1,\eta(h^1)) - \psi(h^2,\eta(h^1))\|^{\nu}_{*,2,\sigma,t_0}  + \|\psi(h^2,\eta(h^1)) - \psi(h^2,\eta(h^2))\|^{\nu}_{*,2,\sigma,t_0}\big)\big) \\
&\le  C \left(|t_0|^{1-2\nu-\mu+\theta}\, \|h^1-h^2\|^{\mu,1+\mu}_{1,\sigma,t_0} + |t_0|^{-\nu+\theta}\|\eta^1 - \eta^2\|^1_{1,\sigma,t_0}\right)
\end{split}
\eee
where we have used (\ref{eqn-contr1}) and (\ref{eqn-contr2}). For the above estimate we only need to check that
$$\|G_1(\psi^1,h^1,\eta^1) - G_1(\psi^1,h^2,\eta^1)\|^1_{\sigma,t_0} \le C|t_0|^{1-2\nu-\mu+\theta}\|h^1-h^2\|^{\mu,1+\mu}_{1,\sigma,t_0}.$$
Indeed, if we pick the term $z^{1-p}N_{13}$, using the estimates from the proof of Lemma \ref{lem-E3} we have
\begin{equation*}
\begin{split}
&  \sup_{\tau\le t_0} |\tau| \left(\int^{\tau + 1}_{\tau}\left(\int_{-\infty}^{\infty}|N_{13}(\psi^1,\eta^1)(h^1) - N_{13}(\psi^1,\eta^1)(h^2)| w_2^p z^{1-p}\, dx\right)^{\sigma}\, dt\right)^{1/\sigma} \\
&\le
C\sup_{\tau\le t_0}|\tau|\left(\int^{\tau+1}_{\tau}\left(\inti |\psi||h^1 - h^2|(|\psi_{\tau}| + |\psi|(|\dot{\xi}_1| + |\dot{\xi}_2| + |\dot{\eta}|))w_2^{p-1}\, dx\right)^{\sigma}\, dt\right)^{1/\sigma} \\
&\le C\, |t_0|^{1-2\nu-\mu + \theta}\|h^1 - h^2\|_{1,\sigma,t_0}^{\mu,1+\mu}.
\end{split}
\end{equation*}

\no Combining the estimates from above yields to 
$$\|F_{h_1} - F_{h_2}\|^1_{2,t_0} \leq C\, \left(|t_0|^{-\delta} \|h^1-h^2\|^{\mu,1+\mu}_{1,\sigma,t_0} + |t_0|^{-\nu+\theta}\|\eta^1 - \eta^2\|^1_{1,\sigma,t_0}\right),$$
where $\delta < \min\{1, \mu, 1-2\nu-\mu+\theta\}$.
Hence by  \eqref{eq-contr-eta-h}  and  choosing $|t_0|$ sufficiently large  we conclude the  bound  \eqref{eqn-etah2}, 
finishing the proof of {\bf (c)}.
\end{proof}

\subsection{Solving for $\xi$}

Recall that  $\xi(t) = \xi_0(t) + h(t)$, $\xi_0(t) = \frac 12 \log (2b\, |t|)$. Moreover, $\xi_0(t)$ is a solution to
the homogeneous part of equation \eqref{eq-xi}, namely 
$$\dot \xi_0 + b\,e^{-2\xi_0} =0.$$
Using this last equation, we  may rewrite the equation (\ref{eq-xi})  the following way
\begin{equation}
\label{eq-h}
\dot{h} + \frac{1}{t}\, h = F(h)
\end{equation}
where $$F(h) := G_3(h) + \mathcal{R}_2(\xi,\dot{\xi},\eta,\dot{\eta}) + G_2(\psi,h,\eta)$$ 
with $G_2(\psi,h,\eta)$  as in Corollary \ref{cor-system} and
\begin{equation}
\label{eq-G3}
G_3(h) := - \big (  \mathcal{G}(\xi_0+h) - \mathcal{G}(\xi_0 ) - \mathcal{D}_{\xi}\mathcal{G}(\xi_0) h \big ), \qquad 
\mathcal{G}(\xi) = e^{-2\xi}.
\end{equation}
\no Equivalently, the left hand side of (\ref{eq-h}) is the linearization of equation (\ref{eq-xi}) around $\xi_0$. 
Recall also that $\eta = \eta(h)$ is obtained by solving (\ref{eq-eta}), and therefore 
$\psi = \psi(h,\eta(h))$.  
Let  
$$K_0   = \{h\, :\, (-\infty,t_0] \to\mathbb{R}\,\, |\,\, \|h\|^{\mu,1+\mu}_{1,\sigma,t_0} \le 1\}.$$
For any function $h \in K_0$
\be
B(h)(t) := |t|^{-2}\int_t^{t_0} s^{2} F(h)(s)\, ds
\ee
defines a solution of equation (\ref{eq-h}) on $(-\infty,t_0]$ with right hand side $F(h) = F(\psi(h,\eta(h)),h,\eta(h))$. Our goal  is to prove the following Proposition.

\begin{prop}
\label{prop-contr-h}
There exists a   universal constant    $t_0 <0$ and   function $h$  with $\| h \|_{1,\sigma,t_0}^{\mu,\mu+1} \leq 1$
 so  that $B(h) = h$. 
\end{prop}

\begin{proof} We will show that the map $B$ leaves the sets invariant and that it is a contraction.

\medskip

\no {\bf (a)} {\it There exists  a universal constant $t_0 < 0$, 
so that $B(K_0) \subset K_0$.}

It is easy to see  that if  $t_0 < 0$ and $\sigma \ge 2$ we have
\be\label{eq-B}
\sup_{\tau\le t_0-1} |\tau|^{\mu} |B(h)(\tau)| \le C\, \|F(h)\|_{2,t_0}^{1+\mu}
\ee
for a universal constant $C$. 
Denote by $I_{\tau} = [\tau, \tau+1]$. Using (\ref{eq-B}), we observe that
\be
\bs
\|B(h)\|^{\mu,\mu+1}_{1,\sigma,t_0} &= \sup_{\tau\le t_0-1} |\tau|^{\mu}\|B(h)\|_{L^{\sigma}(I_{\tau})} + \sup_{\tau\le t_0-1} |\tau|^{1+\mu}\|\frac{d}{dt} B(h)\|_{L^{\sigma}(I_{\tau})} \\
&\le  C\, \|F(h)\|^{1+\mu}_{\sigma,t_0}.
\end{split} 
\ee
We will use this estimate to show that $B(K_0) \subset K_0$. We have
$$\|F(h)\|^{1+\mu}_{\sigma,t_0} = \sup_{\tau\le t_0-1} |\tau|^{1+\mu}\|F(h)\|_{L^{\sigma}(I_{\tau})} \le
C\sup_{\tau\le t_0-1} |\tau|^{1+\mu} \, \big (\|\mathcal{R}_2\|_{L^{\sigma}(I_{\tau})} + \|G_2\|_{L^{\sigma}(I_{\tau})} + \|G_3\|_{L^{\sigma}(I_{\tau})} \big ).$$
A straightforward computation, recalling (\ref{eq-G3}),  shows that 
$$\|G_3(h)\|_{L^{\sigma}(I_{\tau})} \le \frac{C}{|\tau|}\, \|h^2\|_{L^{\sigma}(I_{\tau})}$$
implying (since $\|h\|^{\mu,1+\mu}_{1,\sigma,t_0} \le 1$) the bound 
\begin{equation}
\label{eq-G300}
\sup_{\tau\le t_0-1} |\tau|^{1+\mu}\|G_3(h)\|_{L^{\sigma}(I_{\tau})} \le \frac{C}{|t_0|^{\mu}}.
\end{equation}
By Corollary \ref{cor-system} we have
\begin{equation}
\label{eq-R200}
\sup_{\tau\le t_0-1} |\tau|^{1+\mu}\|\mathcal{R}_2\|_{L^{\sigma}(I_{\tau})} \le \frac{C}{|t_0|^{\gamma - \mu}}
\end{equation}
where we choose $\mu$ so that $\mu < \gamma$. To establish the bound on $\sup_{\tau\le t_0} |\tau|^{1+\mu}\|G_2(h)\|_{L^{\sigma}(I_{\tau})}$, we estimate term by term, similarly to the proof of Lemma \ref{lem-N}.  Similarly to deriving the estimate for the term $G_1$ in section \ref{sec-eta} we have
$$\|G_2\|_{L^{\sigma}(I_{\tau})} \le C|\tau|^{-2\nu+\theta}$$
implying
$$\|G_2\|^{1+\mu}_{\sigma,t_0} \le C|\tau|^{-2\nu+\theta+1+\mu}$$
where  we choose $\mu < \min\{\gamma, 2\nu - 1\}$ and $\theta > 0$ is chosen small  so that $2\nu-\theta-1-\mu > 0$. This together with (\ref{eq-G300}) and (\ref{eq-R200}) imply
$$\|B(h)\|^{\mu,1+\mu}_{1,\sigma,t_0} \le \|F(h)\|^{1+\mu}_{\sigma,t_0} \le \frac{C}{|t_0|^{\tilde{\delta}}},$$
where $\tilde{\delta} = \min\{\mu, \gamma - \mu, 2\nu-\theta-1-\mu\}$. We conclude that   for $|t_0|$ sufficiently large
 we have  $B(K_0) \subset K_0$.

\medskip

\no {\bf (b)} There exists a universal constant $t_0 < 0$ for which the map $B: K_0 \to K_0$ defines a contraction map.

Observe that similarly as in part {\bf (a)} we have
\begin{equation}
\label{eq-diff-B}
\|B(h^1) - B(h^2)\|^{\mu,\mu+1}_{1,\sigma,t_0} \le C\, \|F(h^1) - F(h^2)\|^{1+\mu}_{\sigma,t_0}.
\end{equation}
An easy computation (the same as in part {\bf (a)}) shows that 
$$|G_3(h^1) - G_3(h^2)| \le \frac{C}{|t|}\, (|h^1 - h^2|^2 + (|h^1| + |h^2|)|h^1 - h^2|)$$
implying that
$$\|G_3(h^1) - G_3(h^2)\|_{\sigma,t_0}^{1+\mu} \le \frac{C}{|t_0|^{\mu}}\, \|h^1 - h^2\|^{\mu,1+\mu}_{1,\sigma,t_0}.$$
Let $\xi^i=\xi_0 + h^i$ and recall the notation  $\eta^i = \eta(\xi^i)$. Then, similar to  a discussion in part (c) of the proof of  Proposition \ref{prop-eta-contr},  we  have
\be
\bs
\|\mathcal{R}_2(\xi^1,\dot{\xi}^1,\eta(\xi^1),&\dot{\eta}(\xi^1)) - \mathcal{R}_2(\xi^2,\dot{\xi}^2,\eta(\xi^2),\dot{\eta}(\xi^2))\|_{\sigma,t_0}^{1+\mu} \\
&\le C\, |t_0|^{-\delta} \big (\|h^1 - h^2\|_{1,\sigma,t_0}^{\mu,1+\mu} + \|\eta(h^1) - \eta(h^2)\|^1_{1,\sigma,t_0} \big ) \\
&\le C\, |t_0|^{-\delta}\|h^1 - h^2\|_{1,\sigma,t_0}^{\mu,1+\mu}
\end{split}
\ee
where in the last line we have used (\ref{eqn-etah2}) where  $\delta < \min\{1, \mu, 1-2\nu-\mu+\theta\}$. Furthermore, similar to deriving (\ref{eqn-etah2}) in part (c) of the proof of Proposition \ref{prop-eta-contr},  we get
$$\|G_2(\psi^1,h^1,\eta^1) - G_2(\psi^2,h^2,\eta^2)\|_{\sigma,t_0}^{\mu+1} \le C\, |t_0|^{1-2\nu+\theta}\|h^1-h^2\|^{\mu,1+\mu}_{1,\sigma,t_0}.$$
Using (\ref{eq-diff-B}), the definition of $F(h)$ and the estimates above, if we choose $|t_0|$ sufficiently large, we obtain the bound 
$$\|B(h^1) - B(h^2)\|^{\mu,1+\mu}_{1,\sigma,t_0} \le \frac 12\,  \|h^1 - h^2\|^{\mu,1+\mu}_{1,\sigma,t_0}$$
which finishes the proof of part (b). 

\medskip
By the fixed point theorem applied to the operator $B$, there exists an $h\in K_0$ so that $B(h) = h$, or in other words, $h$  solves (\ref{eq-h}). This concludes the proof of the Proposition. 
\end{proof}

We conclude section \ref{sec-xieta} by the proof of Proposition \ref{prop-etaxi}.

\begin{proof}[Proof of Proposition \ref{prop-etaxi}]
By Proposition \ref{prop-eta-contr}, for every $h\in K$, there exists an $\eta = \eta(h)$ so that \eqref{eqn-Epsio1} is satisfied. By Proposition \ref{prop-contr-h}, if we take $\eta = \eta(h)$, there exists an $h$, so that $\|h\|^{\mu,\mu+1}_{1,\sigma,t_0} \le 1$ and \eqref{eqn-Epsio2} holds. Take this pair of functions $(h,\eta) \in K$, for which both \eqref{eqn-Epsio1} and \eqref{eqn-Epsio2} are satisfied.  Then by Proposition \ref{prop-fp} there is a solution $\psi = {\bf \Psi}(h,\eta)$ of \eqref{eqn-Psi}, satisfying the orthogonality conditions \eqref{eqn-orth1} and \eqref{eqn-orth2}. This finishes the proof of Proposition \ref{prop-etaxi}.
\end{proof}

\section{Properties of the solution}

Unlike the contracting spheres \eqref{eq-spheres}  and the King solution 
\eqref{eq-king} to the Yamabe flow (\ref{eq-YF}) which are both Type I ancient solutions with positive Ricci curvature, the  solution that we construct in Theorem \ref{thm-main} is of Type II and its  Ricci curvature   changes its sign. More precisely we have the following Proposition.

\begin{prop}
\label{prop-prop}
The solution constructed in Theorem \ref{thm-main} is a Type II ancient solution in the sense of Definition \ref{defn-ancient}. Its Ricci curvature changes its sign. 
\end{prop}

\begin{proof}
We recall  that under the conformal change  of the metric
$g = e^{2 f}\, g_{_{\R^n}}$, the Ricci tensor changes as follows 
$$R_{ij} = -(n-2)(f_{ij} - f_i\, f_j) + (\Delta f - (n-2) |\nabla f|^2)\delta_{ij}.$$
In particular, if $f$ is a radially symmetric function and we denote by $R_{11}$ the Ricci tensor in the radial direction and by $R_{jj}$, with $j \ge 2$ the Ricci tensor in the spherical direction, then 
$$R_{11} = -(n-3)\, f_{rr} + \frac{n-1}{r}\, f_r$$
and for $j \geq 2$, 
$$R_{jj} = f_{rr} + \frac{n-1}{r}\, f_r - (n-2) f_r^2.$$
Observe that $R_{ij} = 0$ for $i\neq j$. 
We will next   express $R_{11}$  with respect to 
our conformal factor $u(x,t)$, expressed in cylindrical coordinates. 
Let $u(x,t)$ be the solution of \eqref{eqn-v} constructed in Theorem \ref{thm-main}. Recall that $u(x,t)$ 
represents the conformal factor of the rescaled flow  in cylindrical coordinates. Using the change of
variables 
$$f(r) = \frac{2}{n-2}\log u(x,\tau) - x, \,\,\, r = e^x$$
we find that   $R_{11}$ (the Ricci curvature of the rescaled metric in the radial direction)  in terms of $u$ is given by 
$$R_{11} = -\frac{(n-2)^2\, u^2 - (n-3)\, u_x^2 + u\, [(n-3)\, u_{xx} - 2(n-2)\, u_x]}{e^{2x}\, u^2}.$$

\medskip 
\no {\em The Ricci curvature changes its sign:} We will show that $R_{11}$ changes sign. 
Note that the sign of $R_{11}$ is determined by the sign of 
$$Q := -\big(\, (n-2)^2\, u^2 - (n-3)\, u_x^2 + u\, [(n-3)\, u_{xx} - 2(n-2)\, u_x]\, \big ).$$
Recall that our solution $u$ is given by
$$u = (1+\eta) \, z + \psi = z + \tilde{\psi}, \qquad \tilde{\psi} := \eta\, z + \psi.$$
Let $Q := Q_1 + Q_2$, where
$$Q_1 := (n-3)\, z_x^2 - (n-2)^2\, z^2 - z\, [(n-3) z_{xx} - 2(n-2) z_x]$$
and $Q_2$ is the error term that is a linear combination of $\tilde{\psi}_x^2$, $\tilde{\psi}_x\, z_x$, $\tilde{\psi}^2$, $\tilde{\psi} z$, $\tilde{\psi}\, \tilde{\psi}_{xx}$, 
$\tilde{\psi}\, z_{xx}$, $\tilde{\psi}\, z_x$, $\tilde{\psi}_{xx} z$.  An easy computation shows that for $\tau < t_0 -1$, 
\begin{equation}
\label{eq-Q1}
\int_{\tau}^{\tau + 1}\int_{-1}^1 \,Q_1\, z^{p-1}\, dx\, dt \sim  -\frac{C(n)}{|\tau|^{\frac{p+1}{2}}}  <0, \qquad \mbox{as} \,\, \tau \to -\infty.
\end{equation}
Hence, it will be sufficient to show that  
$$\int_{\tau}^{\tau+1}\int_{-1}^1 \, Q_2\, z^{p-1}\, dx\, dt = o  ( \, |\tau|^{-\frac{p+1}{2}}   )\qquad \mbox{as} \,\, \tau \to -\infty.$$
Let us  check that that is the case for some of the terms that enter in the expression for $Q_2$. All the other terms can be checked similarly. 
To simplify the notation set $R_\tau:= [-1,1]\times [\tau,\tau+1]$. 
Using the energy estimate (\ref{eqn-l224}) and the fact that $z^{p-1} \le C|\tau|^{-(p-1)/2}$ on $R_\tau$  we obtain the bound 
\begin{equation}
\begin{split}
\label{eq-oneterm}
\iint_{R_\tau} (\,\tilde{\psi}^2  + \tilde{\psi}_x^2\,)\, z^{p-1}\, dx\, dt  & \le \, |\tau|^{-\frac{p-1}{2}}\, \left(\big(\,\|\eta\|_{\infty,t_0}^1\, \big)^2\, \iint_{R_\tau}  z^2\, dx\, dt + \iint_{R_\tau} (\,\psi^2 + \psi_x^2\,)\, dx\, dt\,\right)  \\
&\le \, |\tau|^{-\frac{p-1}{2}}\, \big(\, |\tau|^{-2} + |\tau|^{-2\nu}\,\big) \le C\, |\tau|^{-\frac{p+3}{2}}
\end{split}
\end{equation}
and also 
\begin{equation}
\begin{split}
\label{eq-twoterm}
 \iint_{R_\tau} \tilde{\psi}\tilde{\psi}_{xx} &z^{p-1}\, dx\, dt \\
&\le C\, |\tau|^{-\frac{p-1}{2}}\, \left(\, (\|\eta\|_{\infty,t_0}^1)^2  \iint_{R_\tau} \, z^2\, dx\, dt + \|\eta\|_{\infty,t_0}^1\, \big(\, \iint_{R_\tau} z\psi_{xx} + \psi z_{xx}\, dx\, dt  \right) \\
&+ \iint_{R_\tau} \psi\psi_{xx}\, z^{p-1}\, dx\, dt \le  C|\tau|^{-\frac{p+2}{2}-\nu} + \iint_{R_\tau} \psi\psi_{xx}\,z^{p-1}\,  dx\, dt. 
\end{split}
\end{equation}
For the last term in (\ref{eq-twoterm}), using the energy estimates (\ref{eqn-l224}) again  we have
\begin{equation}
\begin{split}
\label{eq-threeterm}
\iint_{R_\tau} \psi\psi_{xx} z^{p-1}\, dx\, dt &\le \left(\, \iint_{R_\tau} \psi^2 z^{p-1}\, dx\, dt\, \right)^{1/2}\, \left(\iint_{R_\tau} \psi_{xx}^2 z^{p-1}\, dx\, dt\, \right)^{1/2} \\
&\le C\|\psi\|_{L^{\infty}(\Lambda_{\tau})} |\tau|^{-\frac{p-1}{2}}\,\left(\, \iint_{R_\tau} \psi_{xx}^2\, dx\, dt \right)^{1/2} \le C|\tau|^{-\frac{p-1}{2} - 2\nu}.
\end{split}
\end{equation}
We see that $(p-1)/2 + 2\nu > (p+1)/2$ is equivalent to $2\nu > 1$, which is true. Using (\ref{eqn-l224}) we have
\begin{equation}
\begin{split}
\label{eq-fourterm}
 \iint_{R_\tau} \tilde{\psi}_{xx}\, z^p \, dx\, dt \
&\le C\, \|\eta\|_{\infty,t_0}^1\, |\tau|^{-(p+1)/2} + \iint_{R_\tau} \psi_{xx}\, z^p\, dx\, dt  \\
&\le C\, (\, |\tau|^{- \frac{p+3}2} + |\tau|^{-\frac p2-\nu})
\end{split}
\end{equation}
Combining (\ref{eq-Q1}), (\ref{eq-oneterm}), (\ref{eq-twoterm}), (\ref{eq-threeterm}), (\ref{eq-fourterm}) yields 
$$\iint_{R_\tau} \, Q\,z^{p-1}\, dx\, dt \le - C(n)\, |\tau|^{-\frac{p+1}{2}}$$
 if $\tau$ is sufficiently close to $-\infty$. 
Hence   $R_{11}$  has to be negative somewhere on $R_\tau$,  if $\tau$ is sufficiently close to $-\infty$. On the other hand, since the scalar curvature 
$R $ of our ancient metric is positive, the Ricci curvature must be positive somewhere. 
The conclusion is that the Ricci curvature of our ancient solution changes its sign all the way to $t \to -\infty$.

\medskip

\no {\em The solution is of   Type II:}
We observe that our  rescaled ancient solution is of Type II   if 
$$
\limsup_{t\to -\infty} |\mbox{Rm}| (\cdot,t) = + \infty
$$
Hence, it  will be sufficient to show that 
\begin{equation}
\label{eq-typeII}
\limsup_{t\to -\infty} |\mbox{Ric}|(\cdot,t) = +\infty.
\end{equation}
Since, as we have noticed above,   the Ricci curvature of our radially symmetric solution is diagonal, we have,
$$| \mbox{\mbox{Ric}} |^2 = R_{11}^2 (g^{11})^2 + \sum_{i\ge 2} R_{ii}^2(g^{ii})^2.$$
Look at the $R_{11}g^{11} = R_{11}\, u^{-\frac{4}{n-2}}$, which is,
$$R_{11}g^{11} = \frac{(n-3) u_x^2 - (n-2)^2 u^2 - u\, [(n-3) u_{xx} - 2(n-2) u_x]}{e^{2x} u^{\frac{2n}{n-2}}}.$$
As in part (a), write $u = z + \tilde{\psi}$. Since $\|\psi\|_{L^{\infty}(\Lambda_{\tau})} \le C|\tau|^{-\nu}$, and $z \sim \tilde{C}\, |\tau|^{-1/2}$, on  $R_\tau$, we have ${z}/{2} \le u \le 2z$ on $R_\tau$. Let us  write
$$R_{11}g^{11} = J_1 + J_2,$$
where the term $J_1$ we obtain from  $R_{11}g^{11} $, after replacing $u$ by $z$, and $J_2$ is the difference of those two terms. An easy computation shows that 
\begin{equation}
\label{eq-J1}
\left | \iint_{R_\tau} \, J_1\, z^{p-1}\, dx\, dt \right| \sim C|\tau|^{\frac{2}{n-2}}\, \iint_{R_\tau}  z^{p-1}\, dx\, dt  \sim C, \qquad \mbox{as} \,\, \tau \to -\infty.
\end{equation}
Using the energy estimate (\ref{eqn-l224}), the fact that $z \sim |\tau|^{-1/2}$ on $R_\tau$,   very similar estimates to those in part (a) show that
\begin{equation}
\label{eq-J2}
\left |  \iint_{R_\tau} J_2\, z^{p-1}\, dx\, dt\right| \le C\, |\tau|^{-q},
\end{equation}
for some $q > 0$. Combining (\ref{eq-J1}) and (\ref{eq-J2}) we see that
$$\left|  \iint_{R_\tau}  R_{11}g^{11}\, z^{p-1}\, dx\, dt\right| \ge c > 0,$$
for all $\tau$ sufficiently close to $-\infty$.   Since $z^{p-1} \sim |\tau|^{-2/(n-2)}$ on $ R_\tau$,  the last estimate  implies there exists a uniform constant $\delta > 0$ so that for every $\tau \le \tau_0$, 
with $\tau_0$ sufficiently close to $-\infty$, there exists an $(x_{\tau}, \tau)  \in R_\tau$, so that 
$$R_{11} g^{11} (x_{\tau},\tau) \ge \delta \,  |\tau|^{\frac 2{n-2}}.$$ We conclude that  (\ref{eq-typeII}) holds and  our solution is of   Type II.
\end{proof}

We conclude this section with a final remark on our shape of our solution, as $t \to -\infty$.  
\begin{remark}
The  ancient solution $u(x,t)$  constructed  in Theorem \ref{thm-main} looks like a tower of two bubbles as $t\to -\infty$. 
\end{remark}

More precisely, for any $\delta \in (0,1)$ it is easy to check that we have the following:
\begin{enumerate}
\item[(a)]
For $x < \xi (t) \,  (1-\delta)$ we have $|u(x,t) - w(x+\xi(t))| < C\, |t|^{-\delta/2}$, which means that in this considered region we are close to one of the spheres (bubbles).
\item[(b)]
For $x > -\xi (t) \, (1-\delta)$ we have $|u(x,t) - w(x-\xi(t))| < C \, |t|^{-\delta/2}$,
 which means that in this considered region we are close to the other sphere (bubble).
\item[(c)]
For any $x_1, x_2 \in [-\xi (t) \, (1-\delta), \xi (t)  (1-\delta)]$ and any two corresponding points $p_1, p_2 \in \R^n$, whose radial variables correspond to $x_1, x_2$ in cylindrical coordinates, respectively,  we have that the $$\dist_{g(t)}(p_1,p_2) \le C \, |t|^{-\frac{\delta}{n-2}} \, \log|t|$$
 which means that in the region where the two spheres interfere we have a short and narrow neck connecting the two bubbles and as $t\to -\infty$ this neck becomes shorter and narrower. We denoted by $g(\cdot)$ the rescaled metric.
\end{enumerate}

To check (c) observe that
\begin{equation*}
\begin{split}
\dist_{g(t)}(p_1, p_2) &\le \int_{-\xi (t) \, (1-\delta)}^{\xi(t)  (1-\delta)} (\,w(x-\xi (t)) + w(x+\xi(t) ) + \psi(x,t)\,)^{\frac 2{n-2}}\, dx \\
&\le C\, \big(\, |t|^{-\delta/2} + |t|^{-\nu}\,\big)^{2/(n-2)}\, \log|t| \le C\, |t|^{-\frac \delta{n-2}} \,\log |t|. 
\end{split}
\end{equation*}

We finish the paper with the proof of the main Theorem \ref{thm-main}, which we restate below.

\begin{theorem}
Let $p:=(n+2)/(n-2)$  with  $n \geq 3$. There exists a  constant $t_0=t_0(n)$ and  a radially symmetric solution $u(x,t)$ to (\ref{eqn-v})
defined on   $R \times (-\infty, t_0]$, 
of the form (\ref{eq-form100})-(\ref{eq-z100}), where the functions    $\psi:=\psi(x,t)$, $\xi:= \frac{1}{2}\log(2b\,|t|)+h(t)$ and $\eta:=\eta(t)$   satisfy 
$\|\psi\|_{*,2,\sigma,t_0}^{\nu}  < \infty$, $\|h\|^{\mu,1+\mu}_{1,\sigma,t_0} < \infty$ and $\|\eta\|^{\mu,}_{1,\sigma,t_0} < \infty$  (according to the Definitions \ref{defn-norm} and \ref{dfn-heta}). The constants $\sigma, \mu, \nu$ and $b$ are
all positive and depend only on the dimension $n$. 

It follows that the solution $u$ defines a  radially symmetric ancient solution to  the Yamabe flow \eqref{eq-YF}
on $S^n$ 
which is of type II   (in the sense of Definition \ref{defn-ancient}) and its  Ricci curvature changes its sign. 

\end{theorem}

\begin{proof}
Proposition \ref{prop-etaxi} gives us, for $|t_0|$ sufficiently large,   the existence of a radially symmetric solution $\psi$ to \eqref{eqn-Psi}
on $\R \times (-\infty, t_0]$, which is equivalent to the existence of a radially symmetric solution $u$ to \eqref{eqn-v}, defined on $\R\times (-\infty,t_0]$. This finishes the proof of the first part, that is, the existence part of Theorem \ref{thm-main}. Furthermore, by Proposition \ref{prop-prop} our constructed solution   is a Type II ancient solution to the Yamabe flow, with the Ricci curvature that changes its sign.
\end{proof}


\begin{thebibliography}{100}

\bibitem{S2} Brendle, S., {\em  Convergence of the Yamabe flow for arbitrary initial energy}, J. Differential Geom. {\bf 69}  (2005), 217--278.


\bibitem{S1} Brendle, S., {\em  Convergence of the Yamabe flow in dimension 6 and higher},  Invent. Math. {\bf 170}  (2007), 541--576.

 
 \bibitem{Ch} Chow, B., {\em The Yamabe flow on locally conformally
flat manifolds with positive Ricci curvature}, Comm. Pure Appl.
Math. {\bf 65}  (1992), 1003--1014.

\bibitem{DS1} Daskalopoulos,P., Sesum,N., {\em Classification of ancient solutions to the Ricci flow on surfaces}, arXiv: 0902.1158, to appear in J.Diff.Geom. 

\bibitem{DDM} Del Pino,M., Dolbeault,J., Musso,M.,  {\em "Bubble-tower'' radial solutions in the slightly supercritical Brezis-Nirenberg problem}; J. Differential Equations 193 (2003), no. 2, 280--306.

\bibitem{DS} del Pino, M.; S\'aez, M., On the extinction profile for
solutions of $u\sb t=\Delta u\sp {(N-2)/(N+2)}$. Indiana Univ.
Math. J.  {\bf 50}  (2001),  611--628.

\bibitem{GJP} Ge,Y., Jing,R., Pacard,F., {\em Bubble towers for supercritical semilinear elliptic equations}; J. Funct. Anal. 221 (2005), no. 2, 251--302.

\bibitem{H} R. S. Hamilton, Lectures on geometric flows, 1989, unpublished.

\bibitem{K} Kapouleas,N., {\em  Complete constant mean curvature surfaces in Euclidean three-space},
Ann. of Math. 131 (1990), 239-330.

\bibitem{K1} King, J.R., {\em  Exact polynomial solutions to some nonlinear diffusion equations},   Physica. D {\bf 64} (1993), 
39--65. 

\bibitem{K2} King, J.R.,  {\em Asymptotic results for nonlinear diffusion},  European 
J. Appl. Math. {\bf 5} (1994),  359--390. 


\bibitem{MPP} Mazzeo,R., Pacard,F., Pollack,D.,  {\em Connected sums of constant mean curvature surfaces in Euclidean 3 space}, J. Reine Angew. Math. 536 (2001), 115--165.

\bibitem{MPU} Mazzeo,R., Pollack,D., Uhlenbeck,K., {\em Connected sum constructions for constant scalar curvature metrics}, Topol. Methods Nonlinear Anal. 6 (1995), no. 2, 207ï--233. 

\bibitem{N} Nguyen,X.H., {\em Construction of Complete Embedded Self-Similar Surfaces under Mean Curvature Flow}; arXiv:1004.2657.


\bibitem{R} Rosenau, P., {\em Fast and superfast diffusion processes}, Phys. Rev. Lett. {\bf 74}  (1995), 1056--1059.


\bibitem{S} Schoen,R.,  {\em The existence of weak solutions with prescribed singular behavior for a con- formally invariant scalar equation}, Comm. Pure and Appl. Math. XLI (1988), 317-392.

\bibitem{SS} Schwetlick, H.; Struwe, M. {\em  Convergence of the Yamabe flow for "large'' energies}, 
J. Reine Angew. Math. {\bf 562}  (2003), 59--100.

\bibitem{P} Pollack,D.,  {\em Nonuniqueness and high energy solutions for a conformally invariant scalar
equation}, Comm. Anal. and Geom. 1 (1993), 347-414.

\bibitem{Y} Ye, R. {\em Global existence and convergence of Yamabe flow}, J. Differential Geom. {\bf 39} (1994), 35--50.



\end{thebibliography}
\end{document}